\documentclass[a4paper,10pt,reqno]{article}

\usepackage{amsfonts}
\usepackage{amsmath}
\usepackage{amsthm}
\usepackage{amssymb}
\usepackage{oldgerm}
\usepackage{fancyhdr}
\usepackage{fancybox}
\usepackage{mathrsfs}
\usepackage{graphicx}

\def\b{\mathbb }
\def\phi{\varphi }

\def\scr{\mathscr}\swapnumbers
\theoremstyle{plain}
\newtheorem{theorem}{Theorem}[section]
\newtheorem{corollary}[theorem]{Corollary}
\newtheorem{lemma}[theorem]{Lemma}
\newtheorem{proposition}[theorem]{Proposition}

\theoremstyle{definition}
\newtheorem{definition}[theorem]{Definition}

\newtheorem{examples}[theorem]{Examples}
\newtheorem{example}[theorem]{Example}
\newtheorem{remark}[theorem]{Remark}
\newtheorem{remarks}[theorem]{Remarks}
\theoremstyle{remark}

\numberwithin{equation}{section}
\def\b{\mathbb }

\def\phi{\varphi }

\def\cal{\mathcal }
\def\scr{\mathscr}
\def\hat{\widehat}

\begin{document}

\title{Dunkl theory, convolution algebras, and related Markov processes}
\author{Margit R\"osler\\
Institut f\"ur Mathematik, TU Clausthal\\
Erzstr. 1\\
D-38678 Clausthal-Zellerfeld, Germany\\
roesler@math.tu-clausthal.de
\and
 Michael Voit\\
Fakult\"at f\"ur Mathematik, Technische Universit\"at Dortmund\\
          Vogelpothsweg 87\\
          D-44221 Dortmund, Germany\\
michael.voit@math.tu-dortmund.de}
\date{}

\maketitle

\begin{abstract}
These lecture notes are intended as an introduction to the theory of
rational Dunkl operators, the associated special functions and related
Markov processes with an emphasis  on examples which are related
to Riemannian symmetric spaces
of Euclidean type and  Bessel hypergroups on the matrix cones of positive
semidefinite matrices.

We start with a comprehensive introduction into Dunkl theory: Dunkl operators,
the  intertwining operator and its positivity, the Dunkl kernel and the Dunkl transform,
the Dunkl Laplacian and the associated heat semigroup. We further give an outline of
the connection with Calogero-Moser-Sutherland models and generalized
Hermite polynomials. Moreover, of central interest will  be   product
formulas, generalized translations and associated commutative hypergroup
structures on closed Weyl chambers.
In particular, we explain how  Dunkl theory for particular
multiplicities is related
to  Riemannian symmetric spaces of Euclidean type and
  Bessel hypergroups on the matrix cones,
 and how this leads to a bunch of
 multiplicities for which the Weyl-group invariant Dunkl theory admits a
 probability preserving translation and an associated commutative hypergroup
structure on the closed Weyl chamber. We finally discuss Markov
processes on $\b R^N$ which are related with Dunkl theory with an emphasis on
 connections to random walk on groups and hypergroups.
In particular  associated martingales, martingale
characterizations, moment functions and Appell
characters are studied.
i.e.,
diffusion-reflection processes with the Dunkl Laplacians as generators.
\end{abstract}

\newpage
\tableofcontents
\newpage

\section{Introduction}

Since their invention  twenty years ago, Dunkl operators have initiated an intense development within  the area of
harmonic analysis and special functions associated with root systems.
A basic motivation for this subject comes from the theory of  Riemannian symmetric spaces
whose spherical functions can be considered  as multi-variable special functions depending on certain
discrete sets of parameters. For spaces of rank one, they can be imbedded into classes of one-variable
hypergeometric functions.

Dunkl operators
provide a tool to extend the theory of spherical functions to a theory of multivariable hypergeometric functions. Roughly speaking, Dunkl operators are commuting
differential-reflection operators on a Euclidean space which are  associated with a finite
reflection group and have a continuous set of parameters, called the multiplicities. For fixed root system and multiplicities, the associated Dunkl operators commute. They lead to  commutative algebras which
generalize the algebras of invariant differential operators on Riemannian symmetric spaces. Actually, there are two levels of Dunkl operator theory: first the full theory, involving operators with reflection parts. Second,  the Weyl group invariant theory (in probability, such as in \cite{De} in this volume, sometimes called ``radial'' ), where the associated special functions are Weyl group invariant and
include the spherical functions of Riemannian symmetric spaces for particular multiplicity values.

The first class of Dunkl  operators, nowadays
 often called ``rational'' Dunkl operators,
were introduced
by C.F. Dunkl in a series of papers ([D1-5]), where he built up the
framework for a theory of special functions and integral transforms in several variables related
with reflection groups.
The rational operators generalize the theory of symmetric spaces of Euclidean type.
Besides them,  there are now various further classes
of Dunkl-type  operators, in particular the
trigonometric Dunkl operators of Heckman, Opdam and Cherednik, which generalize the theory
of Riemannian symmetric spaces of the compact and non-compact type (see \cite{He} or \cite{O3}), and
the important $q$-analogues of Macdonald and Cherednik. There are various kinds of limit transitions between these theories and their special functions. Apart from the context of symmetric spaces, Dunkl operators
are also relevant in mathematical physics, namely for the  analysis of
 quantum many body systems of  Calogero-Moser-Sutherland type. These describe algebraically integrable
systems  in one dimension; a good bibliography is contained in
\cite{vDV}.

\smallskip

In this article, we shall restrict ourselves to the rational
 case with multiplicity function $k\ge0$
in which case the most satisfying results and applications in
 analysis and probability are available.
We give a general introduction to rational Dunkl theory, discuss examples, and sketch
the beginning of applications in probability.

To be more precise, these lecture notes are organized as follows:
 We start with a general introduction to rational Dunkl theory: Dunkl operators,
the intertwining operator and its positivity, the Dunkl kernel and the Dunkl transform, as well as
the Dunkl Laplacian and the associated heat semigroup.
 We further give an outline of
the connection with Calogero-Moser-Sutherland models and generalized
Hermite polynomials, and we  derive  product
formulas and generalized translations which are positive and probability preserving in some important cases.
We also explain how  Dunkl theory for particular
multiplicities is related
to  Riemannian symmetric spaces of  Euclidean type.
Furthermore, we present hypergroup convolution algebras on
 cones of positive semidefinite matrices  and show how these lead to a bunch of
 multiplicities for which the Weyl-group invariant Dunkl theory admits
 probability preserving translations and  commutative hypergroup
structures on the closed Weyl chambers of type $B_N$. In the final chapter,
 we  discuss Markov
processes related to Dunkl operators  with an emphasis on
 connections to random walks on groups and hypergroups.
We study   associated martingales, martingale
characterizations, moment functions and Appell
characters.  Of  particular interest are the Dunkl processes, i.e.,
diffusion-reflection processes with the Dunkl Laplacians as generators,
 as these processes share many well-known features of Brownian motions on $\b R^N$.
For a more detailed discussion of these processes and their projections
 to Weyl chambers in view of stochastic analysis
we refer to the contributions \cite{De} and \cite{CGY} in this volume.

\vfill\newpage

\section{Dunkl theory}
\label{S2}

This section gives an introduction to
 the theory of rational Dunkl operators, which we call Dunkl operators for
short,  and to the Dunkl transform.
Main references
are \cite{Du1}, \cite{Du2}, \cite{Du3}, \cite{DX}, \cite{dJ1}, \cite{He2}, \cite{R2}, \cite{R3} and \cite{R6}.  For a background on reflection
groups and root systems the reader is
referred to \cite{Hu} and \cite{Kn}. We can by far not be complete in this survey. In particular, we do not touch the field of  rational Cherednik algebras, but rather focus on aspects which are of interest in connection
with stochastic analysis, such as Fourier analysis and positivity results.

\subsection{Root systems and reflection groups}\label{root systems}

The basic ingredient in the theory of Dunkl operators are root systems and finite reflection
groups, acting on some Euclidean space  $(E, \langle\,.,.\,\rangle)$ of finite dimension $N.$
It will be no restriction to assume that
 that $E=\b R^N$ with the standard Euclidean
 inner product $\,\langle x,y\rangle \,=\, \sum_{j=1}^N x_j y_j$.
For $\alpha\in \b R^N\setminus\{0\}$, we denote by
$\sigma_\alpha $ the orthogonal reflection in
the hyperplane $\langle\alpha\rangle^\perp$ perpendicular to $\alpha$, i.e.
\[\sigma_\alpha(x)\,=\, x -
2\,\frac{\langle\alpha,x\rangle}{|\alpha|^2}\,\alpha\,,\]
where $|x|:=\sqrt{\langle x,x\rangle}$. Each reflection $\sigma_\alpha$ is orthogonal with respect to the standard
inner product.

 \begin{definition}
A finite subset $ R\subset \b R^N\setminus\{0\}$  is called a
 \textit{root system}, if
\[\sigma_\alpha(R) = R \,\,\text{ for all }\, \alpha\in R.\]
The dimension of $span_{\b R}R$
is called the rank of $R$.
There are two possible additional requirements:  $R$ is called
\begin{itemize}
\itemsep=-1pt  \item
\textit{reduced}, if $R\cap \mathbb R\alpha =\{\pm\alpha\}$ for all $\alpha\in R.$ 
 \item \textit{crystallographic},  if $R$ has full rank $N$ and
 \[ \frac{2\langle\alpha,\beta\rangle}{\langle\beta,\beta\rangle}\,\in \b Z \quad\text{for all }\, \alpha,\beta\in R.\]
\end{itemize}
The group $\, W= W(R)\subseteq O(N)$ which is generated by the
reflections $\{\sigma_\alpha,\,\alpha\in R\}$
is called the \textit{reflection group}
(or \textit{Coxeter group}) associated with $R$.
\end{definition}
If $R$ is crystallographic, then
$span_{\b Z} R$ forms a lattice in $\b R^N$ (called the \textit{root-lattice}) which is
 stabilized by the action of the associated reflection group.

In rational Dunkl theory, one usually works with reduced root systems which are not necessarily crystallographic.
On the other hand, the root systems occuring in Lie theory and in geometric contexts associated with Riemannian symmetric spaces are always crystallographic, and this requirement is also fundamental in the theory of trigonometric Dunkl operators.

\begin{lemma}\label{rootsystemdef}\begin{enumerate}\itemsep=-1pt
\item[\rm{(1)}] If $\alpha$ is in $R$, then also $-\alpha$ is in $R$.
\item[\rm{(2)}] For any root system $R$ in $\b R^N,$ the reflection group $W=W(R)$ is finite.
\item[\rm{(3)}] The set of reflections contained in $W$ is exactly
$\{\sigma_\alpha\,,\,\alpha\in R\}.$
\item[\rm{(4)}] $w\sigma_\alpha w^{-1} = \sigma_{w\alpha}$ for all $w\in W$ and $\alpha\in R$.
\end{enumerate}
\end{lemma}

\begin{proof} (1) This follows since $\sigma_\alpha(\alpha) = -\alpha$. (2)
As $R$ is fixed under the action of $W,$ the assignment $ \varphi(w)(\alpha):= w\alpha$ defines a
homomorphism
$\varphi: W \to S(R)$
of $W$ into the symmetric group $S(R)$ of $R$. This homomorphism is easily checked to be injective. Thus $W$ is naturally identified with a subgroup of $S(R)$, which is finite.
 Part (3)
is slightly more involved. An elegant proof can be found in Section 4.2 of \cite{DX}. Part (4) is straight forward.
\end{proof}

Properties  (3) and (4) imply in particular that  there
is a bijective correspondence between  the conjugacy classes of reflections in $W$ and the
orbits in $R$ under the natural action of $W$.
 We shall need some more concepts:
Each root system can  be written as a disjoint union $R= R_+ \cup (-R_+),$
where $R_+$ and $-R_+$ are separated by a hyperplane $ \langle\{x\in \mathbb R^N: \langle \beta,x\rangle = 0\}$
with $\beta\notin R$.
Such a set $R_+$ is
called a \textit{positive subsystem}.
The set of reflecting
hyperplanes $\{ \langle\alpha\rangle^\perp, \,\alpha\in R\}$ divides $\b R^N$ into connected
open components, called the \textit{Weyl chambers} of $R$.
It can be shown that  the topological closure
$\overline C$ of any chamber $C$ is a fundamental domain for $W$,
i.e. $\overline C$ is naturally homeomorphic with the
 space $(\b R^N)^W$ of all $W$-orbits in $\b R^N$, endowed
with the quotient topology (see Section 1.12 of \cite{Hu}). $W$ permutes the reflecting hyperplanes as well as the chambers.

\begin{examples}\label{Ex_root} \begin{enumerate}
\item[(1)]
$  I_2(n), \,n\geq 3$. Root systems of the \textit{dihedral groups.}
Define $\mathcal D_n$ to be the dihedral group of order $2n$, consisting of the orthogonal transformations in the Euclidean plane $\mathbb R^2$ which preserve a regular $n$-sided polygon centered at the origin. It is generated by the reflection  at the $x$-axis and the reflection at the line through the origin which meets the $x$-axis at the angle $\pi/n$. Root system $I_2(n)$ is crystallographic only for $n=2,3,4,6$.

 \item[(2)] $A_{N-1}.$ Let $S_N$ denote  the  symmetric group in $N$ elements.
It acts faithfully
 on $\b R^N$ by permuting the standard basis vectors $e_1,\ldots,e_N$. Each
transposition $(ij)$ acts as a reflection $\,\sigma_{ij}$ sending
$ e_i - e_j$ to its negative. Since $S_N$ is generated by transpositions, it is a finite reflection group.
The root system of $S_N$ is called $A_{N-1}$ and is given by
\[ A_{N-1} = \{ \pm(e_i - e_j), \,1\leq i<j\leq N\}.\]
This root system is crystallographic. Its span is the orthogonal complement of the vector $e_1 + \ldots + e_N$, and thus the   rank is $N-1$.

\item[(3)] $ B_N.$ Here $W$ is the reflection group in $\b R^N$ generated by the
transpositions $\sigma_{ij}$ as above, as well as the sign changes
$\sigma_i: e_i\mapsto -e_i\,,\> i=1,\ldots,N.$
The group of sign changes is isomorphic to $\b Z_2^N$, intersects
$S_N$ trivially and is normalized by $S_N$, so  $W$ is isomorphic to the semidirect product
$S_N\ltimes \b Z_2^N$. The corresponding root system is called $B_N$;  it is given by
\[ B_N\,=\{ \pm e_i,\, 1\leq i \leq N\} \cup \{\pm(e_i\pm e_j), 1\leq i < j\leq N\}.\]
$B_N$ is crystallographic and has rank $N$.

\item[(4)]   $BC_N$. This is the root system in $\mathbb R^N$ given by
\[ BC_N\,=\{ \pm e_i,\, \pm 2e_i, \, 1\leq i \leq N\} \cup \{\pm(e_i\pm e_j), 1\leq i < j\leq N\}.\]
It is crystallographic, but not reduced.
\end{enumerate}
\end{examples}

A root system $R$ is called \textit{irreducible}, if it cannot be written as the
orthogonal disjoint union $R= R_1\cup R_2$ of two root systems $R_1\,,\, R_2$. Any root system can be uniquely written as an orthogonal disjoint union of irreducible root systems.
There exists a classification of all irreducible, reduced root systems in terms of
Coxeter graphs. There are  $5$ infinite series:
 $A_N, B_N, C_N, D_N$ (which are crystallographic), as well as the rank $2$ root systems $I_2(n)$ corresponding to the dihedral groups. Apart from those, there is a finite number of exceptional root systems. The root systems  $BC_N$ are the only irreducible crystallographic root systems which are not reduced.
The  root system of a complex semisimple Lie algebra
is always crystallographic and reduced, and it is irreducible exactly if the Lie algebra is simple. For further details on root systems, see \cite{Hu} and \cite{Kn}.

\subsection{Dunkl operators}\label{Dunkloperators}

Let $R$ be a reduced (not necessarily crystallographic) root system in $\b R^N$ and $W$ the associated reflection group.
The Dunkl operators attached with $R$ are
modifications of the usual partial derivatives by reflection parts, which are
coupled by parameters. The parameters  are given in terms of a
so-called multiplicity function:

\begin{definition}\label{D:multiplicity} A function $k:R\to \b C$ on the root system $R$ is called a
\textit{multiplicity function}, if it is invariant under the natural action of $W$ on $R$.
\end{definition}

The set of multiplicity functions forms a $\mathbb C$-vector space whose dimension is
equal to the number of $W$-orbits in $R$.

\begin{definition}\label{D:Dunkloperator} Let $k:R\to \b C$ be a multiplicity function on $R$. Then for $\xi\in \b R^N,$
the \textit{Dunkl operator} $T_\xi:= T_\xi(k)$ is defined on $C^1(\b R^N)$ by
\[ T_\xi f(x):= \partial_\xi f(x) + \sum_{\alpha\in R_+}
   k(\alpha)\,\langle\alpha, \xi\rangle\,
    \frac{f(x) - f(\sigma_\alpha x)}{\langle\alpha, x\rangle}.\]
Here $\partial_\xi$ denotes the directional  derivative corresponding
    to $\xi$, and $R_+$ is a fixed positive subsystem of $R$.
For the $i$-th standard basis vector $\xi=e_i\in \b R^N$ we use
the abbreviation $T_i=T_{e_i}.$
\end{definition}

The operators $T_\xi$ were introduced and first studied by
C.F. Dunkl ([D1-4]).
By the $W$-invariance of $k$, their definition  does
 not depend on the special choice of $R_+$. Also, the length of the roots is irrelevant in the formula for $T_\xi$.
This is the basic reason for the convention which requires reduced root systems: a Dunkl operator with summation over a non-reduced root system can be replaced by a counterpart with summation about an associated reduced counterpart,
the multiplicities being modified accordingly. Note further that the  dependence of $T_\xi$ on $\xi$ is linear.  In case $k=0$, the $T_\xi(k)$ reduce to the
corresponding directional
derivatives.

Dunkl operators enjoy regularity properties similar
to usual partial derivatives on various spaces of functions. In order to formulate them, we introduce some further standard notation. We denote by
$\mathcal P := \b C[\b R^N]$  the $\b C$-algebra of polynomial
functions on $\b R^N$. It  has a natural grading
\[ \mathcal P =
\bigoplus_{n\geq 0} \mathcal P_n\,,\]
 where $\mathcal P_n$ is the subspace of homogeneous polynomials of
(total) degree $n$. Further, $\mathcal S(\b R^N)$ denotes the Schwartz space of rapidly decreasing functions on $\b R^N$
 with the usual
locally convex topology.

\begin{lemma}\label{L:regular}
\begin{enumerate}\itemsep=-1pt
\item[\rm{(1)}] If $f \in C^m(\b R^N)$ with $m\geq 1$,  then $T_\xi f\in
  C^{m-1}(\b R^N)$.
\item[\rm{(2)}] $T_\xi $ leaves $C_c^\infty(\b R^N)$ and $\mathcal S(\b R^N)$ invariant.
\item[\rm{(3)}] $T_\xi $ is homogeneous of degree $-1$ on $\mathcal P$,
  that is, $T_\xi\,p\in \mathcal P_{n-1}$ for $p\in \mathcal
  P_n$.
\end{enumerate}
\end{lemma}

\begin{proof} By the fundamental theorem of calculus, one obtains for $\alpha\in R$ the representation
\[ \frac{f(\sigma_\alpha x)-f(x)}{\langle\alpha,x\rangle}\,=\,
\int_0^1\partial_\alpha
f\big(x-2t\frac{\langle\alpha,x\rangle}{\langle \alpha,\alpha\rangle}\alpha\big)dt.\]
From this, (1)  (3) and the first part of (2) are immediate; the proof of
(2) for $\mathcal S(\b R^N)$ is also straightforward but more technical; it
can be found in \cite{dJ1}.
\end{proof}

The Dunkl operators $T_\xi= T_\xi(k) $ are $W$-equivariant, that is
\begin{equation}\label{equiv} wT_\xi w^{-1} = T_{w\xi}\quad\text{ for all } w\in W
\end{equation}
where the action of $W$ on functions $f:\b R^N\to \b C$ is given by
\[ w\cdot f(x):= f(w^{-1}x).\]
Relation \eqref{equiv} is obtained from the definition of the Dunkl operators and the $W$-invariance of $k$.
Similar, one obtains the following product rule:

\begin{lemma}\label{L:Produktregel}
If $f,g\in C^1(\b R^N)$ and at least
one of them is $W$-invariant, then
\begin{equation}\label{(1.2)}
 T_\xi(fg)\,=\, T_\xi(f)\cdot g \,+\, f\cdot T_\xi(g).
\end{equation}
\end{lemma}

The most important property of the Dunkl operators, which is the basis for rich analytic
structures related with them, is the following Theorem of C.F. Dunkl, \cite{Du1}:

\begin{theorem}\label{T:commut} \enskip  For fixed $k,$ the Dunkl operators
$T_\xi = T_\xi(k), \>\xi\in \b R^N\,$ commute.
 \end{theorem}

This result was obtained in \cite{Du1} by a clever direct argumentation. There are also
alternative proofs.  In \cite{DJO}, a proof relying  on Koszul complex ideas is given. Another, indirect method is
to deduce the commutativity of the rational Dunkl operators by a contraction limit from the corresponding result in the trigonometric case, where the Dunkl (-Cherednik) operators are simultaneously diagonalized  by a certain family of trigonometric polynomials. See \cite{Op} for the Cherednik case and \cite{dJ5} for the contraction limit.

\noindent
As a consequence of Theorem \ref{T:commut}, the assignment
\[\Phi_k: x_i\,\to T_i(k), \,\, 1\to id\]
extends to an algebra homomorphism $\, \Phi: \mathcal P \to\,\text{End}_{\b C}(\mathcal P). $
For $p\in \mathcal P$ we write
\[ p(T(k)):= \Phi_k(p)\]
for the Dunkl operator associated with $p$.
The classical case $k=0$ will be distinguished by the notation $\Phi_0(p) =: p(\partial)$.

Let us denote by $\mathcal P^W$ the subalgebra of $\mathcal P$ consisting of those polynomials which are $W$-invariant. Suppose that
$p\in \mathcal P^W.$ Then it follows from the $W$- equivariance of the $T_\xi$ that
the associated Dunkl operator $p(T)=p(T(k))$ is $W$-invariant, that is $\, w(p(T(k))w^{-1} = p(T(k)).$ We denote by
$ \, \text{Res}\, p(T(k)):\mathcal P^W \to \mathcal P^W\,$
the restriction of this operator to $\mathcal P^W$. It has been shown by Heckman \cite{He2}
that this operator  -- as one expects -- acts as a differential operator with coefficients from $\mathcal P^W$.

Of particular importance is the \textit{Dunkl Laplacian}, which is
defined by
\[ \Delta_k := p(T(k)) \quad\text{with }\, p(x) = |x|^2.\]
As $p$ is $W$-invariant, it follows that $\Delta_k$ is $W$-invariant and
\[ \Delta_k = \sum_{i=1}^N T_{\xi_i}^2
\]
for any orthonormal basis $\{\xi_1,\ldots,\xi_N\}$ of $\b R^N.$
The Dunkl Laplacian can be written explicitly as follows (see \cite{Du1} or \cite{DX} for the proof):

\begin{proposition}\label{T:Laplace}
\begin{equation} \label{(1.3)}  \Delta_k  = \Delta   + \,2\sum_{\alpha\in R_+}
k_\alpha\delta_\alpha \quad \text{with}\quad \delta_\alpha f(x) =
\frac{\partial_\alpha f(x)}{\langle\alpha ,x\rangle} -\frac{ |\alpha|^2}{2}
\frac{f(x)-f(\sigma_\alpha x)}{\langle\alpha , x\rangle^2},\end{equation}
where $\Delta$ denotes the usual Laplacian on $\mathbb R^N.$
\end{proposition}

The restriction of $\Delta_k$ is therefore given by
\[ \text{Res} \,\Delta_k = \Delta + 2\sum_{\alpha\in R_+} k_\alpha \frac{\partial_\alpha}{\langle \alpha,\,.\,\rangle}.\]
Readers familiar with the theory of symmetric spaces will notice that this generalizes the radial part of the  Laplace-Beltrami operator of a Riemannian symmetric space of Euclidean type, where $R$ is always crystallographic and $k$ takes only certain values
(essentially the multiplicities of restricted root spaces). We shall return to this important aspect in Section \ref{Bessel-root-systems}.

\begin{examples} \label{E:dunklops}
 {\bf (1) The rank-one case.} \enskip If $N=1$ then $R=\{\pm \alpha\}$, which is a root system of type $A_1$.
The corresponding  reflection group is
$W=\{id,\sigma\}$ acting on $\b R$ by $\sigma(x)=-x$. The
Dunkl operator
$T= T_{1}(k)$  with  multiplicity parameter
$k\in \b C$ is given by
\[ Tf(x)\,=\, f^\prime(x) + k\,\frac{f(x)-f(-x)}{x}.\]
The restriction of $T^2$ to even functions is  a singular Sturm-Liouville operator,
\[ \text{Res}\,T^2 f(x)\,=\, f^{\prime\prime}(x) +
\frac{2k}{x}\cdot f^\prime(x)\,.\]
For $k= (n-1)/2$, this is just the radial part of the usual Laplacian on $\mathbb R^n$ with respect to standard polar coordinates.

{\bf (2) Dunkl operators  of type $A_{N-1}$.} Consider $W=S_N$ acting on $\mathbb R^N$.
As all transpositions  are conjugate in $S_N$,
 the vector space of multiplicity
functions  is one-dimensional. The Dunkl operators associated
with the multiplicity parameter $k\in \b C$ are given by
\[ T_i \,=\,\partial_i \,+\,k\cdot\sum_{j\not= i}
\frac{1-\sigma_{ij}}{x_i-x_j} \quad ( i=1,\ldots, N)\]
with $\sigma_{ij}$ as in Example \ref{Ex_root}, and the Dunkl Laplacian is
\[ \Delta_k^S\,=\, \Delta +2k\sum_{1\leq i<j\leq N}
\frac{1}{x_i-x_j}\Big[(\partial_i-\partial_j) -
\frac{1-\sigma_{ij}}{x_i-x_j}\Big].\]

{\bf (3)  Dunkl operators of type $B_N$.} There
are two conjugacy classes of
 reflections, corresponding to $\{\pm e_i\}$ and to $\{\pm e_i\pm e_j\}$, respectively. The multiplicity function is thus of the
form $k =
 (k_1, k_2)$ with $k_i\in \b C$. The associated Dunkl operators are given by
\[ T_i\,=\, \partial_i\, +\, k_1\frac{1-\sigma_i}{x_i}\,+\,
k_2\cdot\sum_{j\not=i} \Big[\frac{1-\sigma_{ij}}{x_i-x_j} + \frac{1-\tau_{ij}}{x_i+x_j}
\Big] \quad (i=1,\ldots, N),\]
where $\tau_{ij}:= \sigma_{ij}\sigma_i\sigma_j\,.$
\end{examples}

\begin{center}
\fbox{\parbox{4.5in}{
From now on, we shall always require that the multiplicity is \textit{non-negative}, that is $k(\alpha)\geq 0$ for all $\alpha\in R$. We write $k\geq 0$  for short. }}
\end{center}

Parts of the theory extend to a larger range of multiplicities (depending on $R$), but the condition $k\geq 0$  is essential for
positivity results and probability theory.

\subsection{A generalized Fischer pairing}\label{S:Fischer}

In the classical theory of spherical harmonics, the
following bilinear pairing on $\mathcal P$, sometimes called Fischer product, plays an important role:
\[ [p,q]_0 := \bigl(p(\partial)q\bigr)(0), \quad p,q\in \mathcal P.\]
In his theory of generalized spherical harmonics, Dunkl \cite{Du2} introduced the following analogue:

\begin{definition} \label{D:Fischer} For $p,q\in \mathcal P,$
\[ [p,q]_k\,:=\, \bigl(p(T(k))q\bigr)(0).\]
\end{definition}

We collect some of its basic properties:

\begin{lemma}\label{L:Skalprod}\parskip=-1pt
\begin{enumerate}\itemsep=-1pt
\item[\rm{(1)}] If $p\in \mathcal P_n$ and $q\in \mathcal P_m$
    with $n\not=m$, then $\,[p,q]_k =0.$
\item[\rm{(2)}] $\displaystyle [x_i\,p\,,q]_k\,=\, [p,T_i\,q]_k
     \quad (p,q\in \mathcal P,\> i=1,\ldots, N).$
\item[\rm{(3)}] $\displaystyle [w\cdot p\,,w\cdot q]_k\,=\, [p,q]_k
         \quad (p,q\in \mathcal P,\> w\in W).$
\end{enumerate}
\end{lemma}

\begin{proof} Parts (1) and (3) follow from the homogeneity and the $W$-equivariance of the Dunkl operators, respectively. (2) is clear from the definition.
\end{proof}

For $k\geq 0$,
let $w_k$ denote the weight function on $\mathbb R^N$ defined by
\[ w_k(x) = \prod_{\alpha\in R} |\langle\alpha,x\rangle|^{k_\alpha}. \]
It is $W$-invariant and homogeneous of degree $2\gamma$ with
\begin{equation}\label{(1.1a)}
 \gamma=\gamma(k):=\,\sum_{\alpha\in R_+} k(\alpha).
\end{equation}
(Notice that by $W$-invariance of $k$, we have  $k(-\alpha) =
k(\alpha)$ for all
 $\alpha\in R$. Hence  this definition does not depend on the special
choice of  $R_+$).  Further, we define the constant
\[ c_k:= \int_{\b R^N} e^{-|x|^2/2} w_k(x)dx,\]
called a  Macdonald-Mehta-Selberg integral.
There exists a closed form for it which was conjectured and proved
by Macdonald \cite{M1} for the infinite series of crystallographic root systems.
An extension to arbitrary crystallographic reflection groups is due to Opdam \cite{Op1},
and there are computer-assisted proofs for some non-crystallographic root systems.
As far as we know, a general proof for arbitrary root systems has not yet been found.

\noindent
It is an important fact that the Dunkl operators are anti-symmetric with respect to the weight $w_k$:

\begin{proposition}\label{P:Antisym} \cite{Du3} \enskip For $f\in \mathcal S(\b R^N)$ and $g\in C_b^1(\b R^N), $
\[ \int_{\b R^N} T_\xi f(x)g(x)w_k(x)dx\,=\, -\int_{\b R^N}
f(x)
T_\xi g(x)w_k(x)dx\,.\]
\end{proposition}

\begin{proof} A short calculation. In order to have the appearing
integrals well defined, one has to assume $k\geq 1$ first, and then extend the result to all $k\in \mathbb C$ with $\text{Re} k \geq 0$
by analytic continuation.
\end{proof}

The paring $[\,.\,,\,.\,]_k$ is closely related to the inner product in  $L^2(\mathbb R^N, e^{-|x|^2/2} w_k)$. More precisely,
             we have the following identity, which was first observed in the classical case $k=0$ by Macdonald \cite{M2} and  then  generalized to the Dunkl setting in \cite{Du2}:

\begin{proposition}\label{P:Macdonald}
For all $p,q\in \mathcal P$,
\begin{equation}\label{(1.10)}
 [p,q]_k\,=\, c_k^{-1} \int_{\b R^N} e^{-\Delta_k/2} p(x)\,
e^{-\Delta_k/2} q(x)\,e^{-|x|^2/2}\,w_k(x)dx.\end{equation}
\end{proposition}

As the Dunkl Laplacian is homogeneous of degree $-2$,
the operator $e^{-\Delta_k/2}$ is well-defined and bijective on $\mathcal P$, and it
 preserves the (total) degree.
An elegant proof (on the basis of our present knowledge only) was given by M. de Jeu in his thesis \cite{dJ3}. It is published in \cite{dJ5}.

\begin{corollary} \label{C:Skalarprod}
The  pairing  $\,[\,.\,,\,.\,]_k\,$ on $\mathcal P$ is symmetric and non-degenerate.
\end{corollary}

\subsection{Dunkl's intertwining operator}\label{Dunkl's intertwining operator}

Besides commutativity of the Dunkl operators, the second important result in rational Dunkl theory is the existence of an intertwining operator. This is an isomorphism on $\mathcal P$  which intertwines the commutative algebra of Dunkl operators  with the algebra of partial differential operators with constant coefficients. This operator was first constructed in \cite{Du2} for non-negative multiplicities. A  thorough analysis in \cite{DJO} subsequently revealed that for general $k,$
such  an intertwining operator exists if and only if the common kernel of the $T_\xi(k)$,
considered as linear operators on $\mathcal P$, contains no "singular" polynomials
 besides the
constants. As we are only interested in non-negative multiplicities, we restrict ourselves to the results of \cite{Du2}.

\begin{theorem}\label{T:Inter}\cite{Du2}. \enskip
Let $k\geq 0$. Then there exists a unique linear isomorphism
  ("\textit{intertwining operator}") $V_k$ of
  $\mathcal P$ satisfying
  \begin{equation}\label{intertwining}
 V_k(\mathcal
  P_n) = \mathcal P_n\,,\>\> V_k\,\vert_{ \mathcal P_0}\, =\, id
  \quad\text{ and }\quad
  T_\xi V_k\,=\, V_k\partial_\xi \quad \forall \,\xi\in \b R^N.\end{equation}
\end{theorem}

The intertwining operator is defined recursively on the spaces $\mathcal P_n$, as follows:
Let $\mathbb K$ be an extension field of $\mathbb Q$ containing the multiplicities $k_\alpha, \alpha\in R_+$. Consider the group algebra $\mathbb KW = \{\sum_{w\in W} c_w w: c_w\in \mathbb K\}.$ For $w\in W$ and $0<t\leq 1$ define the coefficient $q_w(t)=q_w(k;t)\in \mathbb K$ by
\[ \sum_{w\in W} q_w(t)w \,=\, \text{exp}\Bigl(\ln t\!\sum_{\alpha\in R_+}k_\alpha(1-\sigma_\alpha)\Bigr),\]
which is a central element in $\mathbb KW.$ Let further
\[ c_n(w):=\int_0^1 q_w(t)t^ndt \quad (n\in \mathbb Z_+).\]
Then $V_k$ is defined recursively on the spaces $\mathcal P_n$ by $V_k1=1$ and
\[ V_kp(x):= \sum_{w\in W} c_n(w)\Bigl(\sum_{i=1}^N (wx)_i V_k\bigl((\partial_i p)(wx)\bigr)\Bigr)\quad \text{for }\, p\in \mathcal P_n.\]
For details on this construction,  we refer the interested reader to  \cite{Du2} or \cite{DX}.
In contrast to the definition of $V_k$ itself,  it is fairly easy to write down the inverse $U_k = V_k^{-1}$, namely
\[ U_kp(x):=\bigl(e^{\langle x,T\rangle}p\bigr)(0) = \sum_{\nu\in \mathbb Z_+^N}
 \frac{x^\nu}{\nu !}(T^\nu p)(0).\]
Note that $U_k$ is well-defined, because the Dunkl operators $T_i$ are homogeneous of degree $-1$, and therefore the series is always a finite sum.
Moreover, it is obvious that $U_k1=1$ and that $U_k$ preserves the degree of homogeneity, that is $U_k(\mathcal P_n) \subseteq  \mathcal P_n$. Finally, we have for $p\in \mathcal P_n$
\[ \partial_{x_i}\bigl(e^{\langle x,T\rangle} p\bigr)(0) = \sum_{|\nu|=n, \nu_i\geq 1} \frac{x^{\nu-e_i}}{(\nu-e_i)!}T^\nu p\, =\,
 \sum_{|\mu|=n-1} \frac{x^\mu}{\mu !} T^\mu T_ip \,=\, \bigl(e^{\langle x,T\rangle}T_ip\bigr)(0)
\]
where $\nu - e_i := (\nu_1, \ldots, \nu_i-1, \ldots, \nu_N).\,$ This shows that
$U_k$ satisfies the commutation relation
\[ \partial_i U_k \, = \, U_kT_i \quad (i=1, \ldots, N).\]
As a consequence of  the above representation of $V_k^{-1}$ one obtains  the following

\begin{corollary}[Taylor Formula]\label{C:Taylor}
 Let $f\in C^m(\mathbb R^N)$ for some $m\in \mathbb N$. Then
\[ f(x) = \sum_{\nu\in \mathbb Z_+^N, |\nu|\leq m } \frac{V_k(x^\nu)}{\nu!} (T^\nu f)(0) \, + \, o(|x|^m) \quad\text{for } x\to 0.\]
Further, if $f:\mathbb R^N\to \mathbb C$ is real-analytic in a ball $B$ around $0$,  then
\[ f(x) = \sum_{n=0}^\infty \sum_{|\nu|=n} \frac{V_k(x^\nu)}{\nu!} (T^\nu f)(0)\]
where the series $\sum_{n=0}^\infty$ converges locally uniformly in $B$.
\end{corollary}

\begin{proof}
 Assume first that $f$ is a polynomial. Then
$\, V_k^{-1}f(x) = \sum_{\nu}\frac{x^\nu}{\nu !}(T^\nu f) (0)\,$
and therefore
\[ f(x) = \sum_{\nu} \frac{V_k(x^\nu)}{\nu!} (T^\nu f) (0).\]
The assertions now follow from the corresponding results for the classical case.

\end{proof}

\begin{lemma}\label{L:intertwiner_group} The intertwining operator $V_k$ commutes with the action of $W$:
\[ w^{-1} V_k w\,=\, V_k \quad (w\in W).\]
\end{lemma}

\begin{proof} The operator $\, \widetilde V_k = w^{-1}V_kw$ satisfies all the characteristic properties of $V_k$.
\end{proof}

Though the intertwining operator plays an important role in Dunkl's theory, an
explicit representation  for it is known so far only in some special cases. In the rank-one case it is given  by
\begin{equation}\label{V-k-eindim}
 V_k(x^n) = \frac{\bigl(\frac{1}{2}\bigr)_m}{\bigl(k+\frac{1}{2}\bigr)_m}\,
x^{n}\quad \text{with }\, m = \big\lfloor \frac{n+1}{2}\big\rfloor.
\end{equation}
For $k>0$, this  can be written as an integral operator, namely
\begin{equation}\label{int_rep_rank1}
 V_k\, p(x)\,=\, \frac{\Gamma(k+1/2)}{\Gamma(1/2)\,\Gamma(k)}
  \int_{-1}^1 p(xt)\,(1-t)^{k-1}(1+t)^k\,dt,
\end{equation}
see \cite{Du2}, Theorem 5.1.
 Besides this case, there are explicit integral formulas for the $A_2$-case, treated in \cite{Du4},  and the $B_2$-case
under the additional requirement  $k_1=k_2$, given in \cite{Du8}. They are both based on geometric cases, namely Harish-Chandra integral representations
which are available for single values of the multiplicity only.

\smallskip
For a further development of the theory it is crucial to extend the domain of $V_k$ away from polynomials.
The basic extension is  to  certain normed algebras of homogeneous series. This goes back to \cite{Du2}.

\begin{definition}\label{D:ALg} For $r>0,$ let $B_r:=\{x\in \b R^N: |x|\leq r\}$ denote the closed
ball of
radius $r,$ and let $A_r$  be the closure of $\mathcal P$ with respect to the norm
\[ \|p\|_{A_r}:= \sum_{n=0}^\infty \|p_n\|_{\infty, B_r} \, \quad\text{for }\,
   p=\sum_{n=0}^\infty p_n\,, \,p_n\in \mathcal P_n.\]
\end{definition}

\noindent Clearly $A_r$ is a commutative Banach-$*$-algebra under the pointwise operations and with
complex conjugation as involution.
Each $f\in A_r$ has a unique representation
$f = \sum_{n=0}^\infty f_n\,$ with $f_n\in \mathcal P_n$, and is continuous on the ball $B_r$
and real-analytic in its interior.
The topology of $A_r$ is stronger than the topology induced by the uniform norm on $B_r$.
Notice also that
 $A_r\subseteq A_s$ with $\|\,.\,\|_{A_r} \geq \|\,.\,\|_{A_s}$
for $s\leq r$.

\begin{theorem}\label{T:Homog}
 $\, \|V_k p\|_{\infty, B_r} \leq \|p\|_{\infty, B_r}$ for each $p\in \mathcal P_n$.
\end{theorem}

\noindent
The proof of this result is given in \cite{Du2} and can also be found in
 \cite{DX}.

\begin{corollary}\label{C:Vbeschr}  $\|V_kp\|_{A_r} \,\leq \|p\|_{A_r}$ for every $p\in \mathcal P$, and
$V_k$ extends uniquely to a bounded linear operator on $A_r$
 by
\[ V_k f:= \sum_{n=0}^\infty V_k f_n \quad\text{for }\, f=\sum_{n=0}^\infty f_n\,.\]
\end{corollary}

We call a linear operator $L$ \textit{positive} on $\mathcal P$, if it preserves the positive cone $\, \mathcal P_+ = \{ p\in \mathcal P: p(x) \geq 0 \,\, \forall \, \, x\in \mathbb R^N\}.$
Formula \eqref{int_rep_rank1} shows that in the rank-one case,
the operator $V_k$ is positive on polynomials. This is true for general $R$ and non-negative $k$.
Indeed, the following was proven in \cite{R3}:

\begin{theorem} \label{T:Main2} \begin{enumerate}\itemsep=-1pt
                            \item[\rm{(1)}]
The operator $V_k$ is \textit{positive} on $\mathcal P$.
\item[\rm{(2)}]
For each $x\in \b R^N$ there exists a unique probability measure
$\mu_x^k$ on
the Borel-$\sigma$-algebra of $\b R^N$ such that
\begin{equation}\label{(4.11)}
 V_kf(x)\,=\, \int_{\b R^N} f(\xi)\,d\mu_x^k(\xi)
\end{equation}
for all $f\in A_{|x|}$.
The representing measures $\mu_x^k$ are compactly
supported with
$\,{\rm supp}\,\mu_x^k \subseteq \, {\rm co}(W.x),$ the convex
hull of the orbit of $x$ under $W$.
Moreover, they satisfy
\begin{equation}\label{(4.12a)}
\mu_{rx}^k(B)\,=\,\mu_x^k(r^{-1}B),\quad \mu_{wx}(B)\,=\,\mu_x^k(w^{-1}(B))
\end{equation}
for each $r>0, \> w\in W$  and each Borel set $B$ in $\mathbb R^N$.
\end{enumerate}
\end{theorem}

This result immediately allows to $V_k$  to larger function spaces such as $C(\mathbb R^N)$ or $L_{loc}^1(\mathbb R^N)$. We shall come back to  extensions later.

\medskip

In the following, we outline the proof of Theorem \ref{T:Main2}, for details see \cite{R3}.
We begin with two lemmata which are also of interest in their own.

\begin{lemma}\label{abeltrafo} For all $p,q\in \mathcal P,$ \begin{enumerate}
\item[\rm{(1)}] $[V_k p,q]_k = [p,q]_0$;
\item[\rm{(2)}]
$\displaystyle
c_k^{-1}\int_{\b R^N} (V_k p) q \,e^{-|x|^2/2}w_k(x) dx\,=\, c_0^{-1}\int_{\b R^N} p\, \bigl(e^{-\Delta/2}e^{\Delta_k/2}q\bigr)e^{-|x|^2/2} dx.$
             \end{enumerate}
\end{lemma}

\begin{proof} (1) Due to the orthogonality of the spaces $\mathcal P_n$ with respect to both pairings it
suffices to consider $p,q\in \mathcal P_n$ for some $n$. Then
\[ [V_k\, p,q]_k\,=\, [q,V_k\, p]_k\,=\,q(T)(V_k\, p)\,=\,
   V_k(q(\partial)p)\,=\, q(\partial)(p)\,=\,
[p,q]_0\,; \]
here the characterizing properties of $V_k$ and the fact that
   $q(\partial)(p)$ is a constant have been used.

(2) Combining part (1) with the Macdonald-type identity
\eqref{(1.10)}, one obtains
\[\int_{\b R^N} e^{-\Delta_k/2}(V_k p)e^{-\Delta_k/2} q\,
e^{-|x|^2/2}w_k(x) dx\, = \,\frac{c_k}{c_0}\int_{\b R^N} e^{-\Delta/2}
p \,e^{-\Delta/2} q \,e^{-|x|^2/2} dx.\]
As $\displaystyle\, e^{-\Delta_k/2}(V_k\, p)\,=\,
V_k\bigl(e^{-\Delta/2}p\bigr)\,$, and as we may replace $p$ by
$e^{\Delta/2}p\,$ and $q$ by $e^{\Delta_k/2}q,\,$ this implies the claimed identity.
\end{proof}

\begin{lemma}\label{T:Posmin}
Let $A$ be a degree-lowering linear operator on $\mathcal P$, that is\\ $deg(Ap) < deg(p)$ for all $p\in \mathcal P$. Then
$\,e^{tA}: p \mapsto  \sum_{n=0}^\infty \frac{t^n}{n!}A^n p \,$ is a linear isomorphism of $\mathcal P$ (the series always terminates), and  the
 following statements are
 equivalent: \parskip=-1pt
\begin{enumerate}\itemsep=-1pt
\item[\rm{(1)}] $\, e^{tA}$ is positive on $\mathcal P$ for all $t\geq 0$.
\item[\rm{(2)}] $A$ satisfies the ``positive minimum principle''
\parskip=-1pt
\begin{center}
 {\rm (M)} \hspace{2pt} For every $p\in \mathcal P_+$ and $x_0\in \b R^N,$
\hspace{2pt}
  $p(x_0)=0\>$ implies $\, Ap\,(x_0)\geq 0.$
\end{center}
\end{enumerate}
\end{lemma}

This principle is an adaption of a well-known criterion for generators
of Feller semigroups, see Section \ref{heat}. For its proof see \cite{R3}.

\begin{proof}[Proof of Theorem \ref{T:Main2} (Sketch)] From part (2) of the above lemma it follows by standard density arguments that  (1) of the Theorem is equivalent to the positivity of the operator
\[ e^{-\Delta/2}e^{\Delta_k/2} \]
on $\mathcal P$.
We consider $\Delta_k$ as a perturbation of the usual Laplacian $\Delta$,
\[ \Delta_k = \Delta + L_k \quad\text{with}\quad L_k = 2\sum_{\alpha\in R_+} k_\alpha \delta_\alpha\]
as written in  \eqref{(1.3)}. The operators $\Delta_k, \Delta$ and $L_k$ are degree-lowering, and an
 argument involving the Trotter product formula for $e^{\Delta_k/2} = e^{\Delta/2 + L_k/2}$  shows
that it
suffices to verify positivity of the operators
\[ e^{-\Delta} e^{tL_k} e^{\Delta}\,\quad (t\geq 0) \]
on $\mathcal P.$ But
\[ e^{-\Delta} e^{tL_k} e^{\Delta}\,=\, e^{tA} \quad\text{with }\,
A = e^{-\Delta}L_k e^{\Delta}.\]
Thus, it is enough to show that the operator semigroup $(e^{tA})_{t\geq 0}$ is positive on $\mathcal P$.
This can be achieved by verifying that its generator $A$ satisfies the positive minimum principle $(M)$ stated
in Lemma \ref{T:Posmin}.
Indeed, it is easily checked that $A$ is degree-lowering. Further, $A$ decomposes as
\[ A =  2\sum_{\alpha\in R_+} k_\alpha e^{-\partial_\alpha^2}
\delta_\alpha\, e^{\partial_\alpha^2}\]
(here it is used that $\delta_\alpha$ acts in direction $\alpha$ only).
Direct computation shows that the one-dimensional operators
$ e^{-\partial_\alpha^2}\delta_\alpha e^{\partial_\alpha^2}$ satisfy the minimum principle $(M).$ As the $k_\alpha$
are non-negative,  $A$ also satisfies $(M).$ By Lemma \ref{T:Posmin}, this finishes the proof.

\smallskip\noindent
(2) Part (1) implies that the mapping
\[ \Phi_x : f\mapsto V_kf(x)\]
is a positive linear functional on the commutative Banach-*-algebra $A_{|x|}$. The  Bochner representation theorem for
 positive functionals on commutative
Banach-$*$-algebras (see for instance Theorem 21.2 of \cite{FD}) then implies an integral representation for $\Phi_x$ with
representing measures supported in the ball $B_{|x|}$.
The sharper statement on the support is obtained by results of \cite{dJ1}.
 The remaining statements are easy.
\end{proof}

\subsection{The Dunkl kernel and the Dunkl transform}\label{S:Dunklkernel}

Again we assume that $k\geq 0$. As the Dunkl operators $T_\xi = T_\xi(k)$ commute, it is natural to consider their joint eigenvalue problem. More precisely,  for a fixed spectral parameter $y\in \mathbb C^N$ we search for a function $f$ solving
\[ (E)\quad  \begin {cases} T_\xi f = \langle \xi,y\rangle f \quad \forall \,\, \xi\in \mathbb R^N\\
        f(0) =1.
       \end {cases}\]
Here $\langle\,.\,,\,.\,\rangle$ denotes the \textit{bilinear} extension of the Euclidean inner product to
   $\b C^N\times\b C^N.$
If $k=0$, then a solution to this problem is of course given by the exponential $f(x) = e^{\langle x,y\rangle}.$ In the general case,
we apply the intertwining operator. Notice that for fixed $y$, the function  $x\mapsto e^{\langle x,y\rangle}$ belongs to each of the algebras $A_r\,,\, r>0.$ This justifies the following

\begin{definition} \cite{Du2} \label{D:Dunklkern} For $y\in \b C^N$, define
\[  E_k(x,y):= V_k\bigl(e^{\langle\,.\,,y\rangle})(x), \quad x\in \b R^N.\]
$E_k$ is called the \textit{Dunkl kernel} associated with $R$ and $k$.
\end{definition}

Let us check that  $f(x) = E_k(x,y)$ solves (E). We write
\begin{equation}\label{E_khomog} E_k(x,y) = \sum_{n=0}^\infty E_k^{(n)}(x,y) \quad \text{with }\,\, E_k^{(n)}(x,y) = \frac{1}{n!}V_k\langle\,.\,,y\rangle^n(x).\end{equation}
The homogeneity of $V_k$ immediately implies that $E_k(0,y) = 1.$  Further, by the intertwining property,
\[ T_\xi  E_k^{(n)}(\,.\,,y) \,=\, \frac{1}{n!} V_k\,\partial_\xi \langle\,.\,,y\rangle^n\,=\,
\langle \xi,y\rangle E_k^{(n-1)}(\,.\,,y).\]
This shows that (E) is satisfied.

\begin{remark} From the definition of $E_k$ one can see that
\begin{equation}\label{kernelsum} E_k(x,y) = \sum_{n=0}^\infty \sum_{|\nu|=n} \frac{V_k(x^\nu)y^\nu}{\nu!}
\end{equation}
where the series $\sum_{n=0}^\infty $ converges absolutely and locally uniformly on $\mathbb R^N\times \mathbb C^N$.
\end{remark}

\begin{theorem}\label{P:Uniq} Let  $y\in\b C^N.$
Then $f= E_k(\,.\,,y)$ is the unique solution of the
system
\begin{equation}\label{(1.106)}
  T_\xi\,f\,=\, \langle \xi,y\rangle f \quad \text{for all }\xi\in \b R^N
\end{equation}
which is real-analytic on $\b R^N$ and satisfies $\,f(0)=1.$  Moreover, $ E_k$
 extends to a holomorphic function on
$\b C^N\times \b C^N$.
\end{theorem}

This is a weakened version of a result of Opdam (\cite{Op1}, Prop. 6.7) which includes complex multiplicities as well as
meromorphic dependence of $E_k$ on $k$. The decisive part for the uniqueness proof is
the observation that the joint kernel of the $T_\xi(k)$, considered as linear operators on $\mathcal P$, consists of the constants only.
Details can be found in \cite{R7}, see also \cite{Op1}.

\begin{proposition}\label{P:Kernequiv} For $x,y\in \b C^N, \,\lambda\in \b C$ and $w\in W,$
\begin{enumerate}\itemsep=-1pt
\item[\rm{(1)}] $E_k(x,y) = E_k(y,x).$
\item[\rm{(2)}] $E_k(\lambda x,y) = E_k(x,\lambda y)\, $ and $\,E_k(wx,wy) = E_k(x,y).$
\item[\rm{(3)}] $\overline{E_k(x,y)} = E_k(\overline x,\overline y).$
\end{enumerate}
\end{proposition}

\begin{proof} Part (1)  is shown in \cite{Du2}. (2) is easily obtained from the definition
of $E_k$ together with the homogeneity and equivariance properties of $V_k$.
For (3), notice that $f:= \overline{E_k(\,.\,,y)},$ which is again real-analytic on $\b R^N,$
satisfies $\,T_\xi f =
\langle\xi,\overline y\rangle\,f, \, f(0) =1$. By the uniqueness part of the above Theorem,
 $\overline{E_k(x,y)} = E_k(x,\overline y)$ for all real $x$. Now both
$x\mapsto \overline{E_k(\overline x,y)}$ and $x\mapsto
E_k(x,\overline y)$ are holomorphic on $\b C^N$ and agree on $\b R^N$. Hence they coincide.
\end{proof}

\noindent
Just as with the intertwining operator, the kernel $E_k$ is explicitly
known  for some particular cases only.  An important example is again the rank-one
situation:

\begin{example}\label{E:1kern}  In the \textit{rank-one case}
with $\text{Re}\,k >0,$
the integral representation \eqref{int_rep_rank1}
   for $V_k$ implies that for all $x,y\in \b C$,
\[ E_k(x,y)\,=\, \frac{\Gamma(k+1/2)}{\Gamma(1/2)\,\Gamma(k)}\int_{-1}^1
e^{txy}(1-t)^{k-1}(1+t)^k\,dt\, =\, e^{xy}\cdot
\phantom{}_1F_1(k,2k+1,-2xy).\]
This can also be written as
\begin{equation}\label{kern_rank1}
E_k(x,y)\,=\, j_{k-1/2}(ixy)\,+\,\frac{xy}{2k+1}\,
j_{k+1/2}(ixy)
\end{equation}
where for $\alpha \geq -1/2$, $j_\alpha$ is the normalized spherical
Bessel function
\begin{equation}\label{(1.8)}
 j_\alpha(z)\,=\, _0F_1(\alpha+1;-z^2/4)=\,
\Gamma(\alpha+1)\cdot\sum_{n=0}^\infty
\frac{(-1)^n(z/2)^{2n}}{n!\,\Gamma(n+\alpha+1)}\,.\end{equation}
\end{example}

\noindent
Notice that $\, j_{k-1/2}(ixy) = E_k(x,y) + E_k(-x,y).\,$
This motivates the following
\begin{definition} \cite{Op1}\label{D:Dunklbessel}
The (Dunkl-type) \textit{Bessel function} associated with $R$ and $k$  is defined for $x,y\in \b C^N$ by
\begin{equation}\label{(5.0)}
 J_k(x,y) := \frac{1}{|W|} \sum_{w\in W} E_k(wx,y).
\end{equation}
\end{definition}

Thanks to Prop. \ref{P:Kernequiv} $J_k$ is $W$-invariant in both arguments and
therefore naturally considered on Weyl chambers of $W$
(or their complexifications).
In the rank-one case, we  have
\[ J_k(x,y) = j_{k-1/2}(ixy).\]

\noindent
It is a well-known fact from classical analysis that
for fixed $y\in \b C$, the function $f(x) = j_{k-1/2}(ixy)$ is the unique
analytic solution of the differential equation
\[ f^{\prime\prime} + \frac{2k}{x} f^\prime \,=\, y^2 f \]
which is even and normalized by $f(0)=1$. This fact generalizes
to the multivariable case, as follows: Recall the
algebra of $W$-invariant polynomials
\[ \mathcal P^W = \{ p\in \mathcal P: w\cdot p = p \quad\text{for all }\, w\in W\},\]
as well as the restriction  $\text{Res}\, p(T): \mathcal P^W \to \mathcal P^W$
 for $p\in \mathcal P^W$.
 For fixed spectral parameter $y\in \b C^N$,
$J_k(\,.\,,y)$ is a solution to the following
\textit{Bessel system:}
\[ p(T)f\,=\, p(y)f \quad \text{for all }\, p\in \mathcal P^W, \, \, f(0)=1.\]
Actually, Opdam went the converse direction in  \cite{Op1}. He first proved
 that the Bessel system has a unique  $W$-invariant analytic solution.
From this solution, the generalized Bessel function, he then constructed the
 Dunkl kernel.
 The Bessel system generalizes the so-called system of invariant differential
operators on a Riemannian symmetric space of Euclidean type, and the
 Dunkl-type Bessel
functions  $J_k$ generalize the associated spherical functions.
We shall explain  this connection in more detail in Section \ref{Bessel-root-systems}.

The following general Bochner-type integral representation of the Dunkl kernel is an immediate consequence of Theorem \ref{T:Main2}.

\begin{proposition}\label{C:Bochner}
For each $x\in \b R^N$, the Dunkl kernel $E_k(x,\,.\,)$ has the integral representation
\begin{equation}\label{(5.1)}
 E_k(x,y)\,=\,\int_{\b R^N} e^{\langle\xi,\,y\rangle} d\mu_x^k(\xi)
\end{equation}
where the $\mu_x^k$ are  the representing measures from Theorem \ref{T:Main2}.
A corresponding integral representation holds for the Bessel function $J_k$.
\end{proposition}

\begin{corollary}\label{C:growthbounds} The Dunkl kernel satisfies
\begin{enumerate}\itemsep=-1pt
\item[\rm{(1)}] $E_k(x,y)>0 \quad\text{for all }\, x,y\in \b R^N.$
\item[\rm{(2)}]
For all $x\in \b R^N, y\in \b C^N$ and $\alpha\in \b Z_+^N$,
\[\vert \partial_y^\alpha E_k(x,y)\vert\,\leq\, |x|^{|\alpha|}\max_{w\in W}
     e^{\text{Re}\langle wx,y\rangle}. \]
\item[\rm{(3)}]
$\vert E_k(-ix,y)\vert\,\leq\, 1 \quad \text{for all }\, x,y\in \b R^N$.
\end{enumerate}
\end{corollary}

\begin{remarks}\begin{enumerate}\itemsep=-1pt
 \item[\rm{(1)}] M. de Jeu had already an estimate on $E_k$ with slightly weaker bounds in \cite{dJ1},
 differing by an additional factor $\sqrt{|W|}.$
\item[\rm{(2)}] In \cite{RV2a}, a completely different proof of Proposition \ref{C:Bochner}
is given under the restriction that $R$ is crystallographic. It is based on an
asymptotic relationship
between the Opdam-Cherednik kernel (see \cite{Op}) and the Dunkl kernel which was observed  in \cite{dJ5} (see also \cite{BO1}), as well as on positivity results of S. Sahi for the Heckman-Opdam polynomials and their non-symmetric counterparts. In contrast to the original approach described here, in \cite{RV2a} the precise information on the support is
obtained without using the exponential bounds on $E_k$ from \cite{dJ1}. Theorem \ref{T:Main2} was then, in the converse way, obtained from the integral representation for $E_k$.
\end{enumerate}
\end{remarks}

We conclude this section with two important reproducing properties for the Dunkl kernel due to  \cite{Du3}.
The above estimate (3) on $E_k$ assures the
convergence of the involved integrals.

\begin{proposition}\label{P:Reprod} For all $p\in \mathcal P$ and $y,z\in \b C^N,$
\begin{enumerate}\itemsep=1pt
\item[\rm{(1)}]
$\, \displaystyle \frac{1}{c_k}
\int_{\b R^N} e^{-\Delta_k/2}p\,(x)\, E_k(x,y)\,
 e^{-|x|^2/2}w_k(x)dx\,=\,e^{\langle y,y\rangle/2} p(y).$
\item[\rm{(2)}] $\displaystyle \,
\frac{1}{c_k}\int_{\b R^N} E_k(x,y)\,E_k(x,z)\,e^{-|x|^2/2}w_k(x)dx\,=\,
e^{(\langle y,y\rangle + \langle z,z\rangle)/2} E_k(y,z)$
\end{enumerate}
\end{proposition}

\begin{proof} We use the Macdonald-type formula
\eqref{(1.10)}. First, we show that
\begin{equation}\label{(1.104)}
 [E_k^{(n)}(x,\,.\,), p]_k \,=\, p(x)
\quad\text{for all }\, p\in \mathcal P_n, \, x\in \b R^N.
\end{equation}
Indeed, if $p\in \mathcal P_n$ then
\[p(x) = \frac{\langle x, \partial_y\rangle^n}{n!}p(y) \,\quad\text{and }\,\,
V_k^x p(x) = E_k^{(n)}(x, \partial_y)p(y).\]
Here the uppercase index in $V_k^x$ denotes the relevant variable.
Application of $V_k^y$ to both sides
 gives $\,V_k^x p(x) = \, E_k^{(n)}(x,T^y) V_k^y p(y).$
As $V_k$ is bijective on $\mathcal P_n$, this implies \eqref{(1.104)}.
For fixed $y,$ let $L_n(x):= \sum_{j=0}^n E_k^{(j)}(x,y).$ If $n$ is larger than the degree
of $p,$ it follows from \eqref{(1.104)} that
$\,[L_n, p]_k \,=\, p(y).$ Thus in view of the Macdonald formula,
\[c_k^{-1}\int_{\b R^N} e^{-\Delta_k/2} L_n(x) e^{-\Delta_k/2}p(x)
 e^{-|x|^2/2}w_k(x)dx\,=\,p(y).\]
On the other hand, it is easily checked that
\[ \lim_{n\to\infty} e^{-\Delta_k/2} L_n(x) \,=\, e^{-\langle y,y\rangle/2}E_k(x,y).\]
This proves (1).
Identity (2) follows from (1) by homogeneous expansion of $E_k$.
\end{proof}

\bigskip

The Dunkl kernel gives rise to an integral transform, the Dunkl transform, which was introduced in
\cite{Du3} for non-negative multiplicity functions
and further studied in \cite{dJ1} in the more general case
$\text{Re}\,k\geq 0$. Here  we again restrict ourselves to $k\geq 0$.

\begin{definition}\label{D:Dunkltrafo}  The Dunkl transform associated
with $R$ and $k\geq 0$ is defined on $L^1(\mathbb R^N, w_k)$ by
\[ \widehat f^{\,k}(\xi):=  c_k^{-1}\int_{\b R^N} f(x) E_k(-i\xi,x)
\,w_k(x)dx, \quad \xi\in \b R^N.\]
Note that $\widehat f^{\,k}$ is continuous and bounded.
The inverse transform is defined by $\,f^{\vee k}(\xi):=\widehat f^{\,k}(-\xi)$.
\end{definition}

\begin{lemma}\label{Dunklsymm} For $f,g\in L^1(\mathbb R^N, w_k),$
\[ \int_{\mathbb R^N} \widehat f^{\,k}(x) g(x) w_k(x) dx  \,=\, \int_{\mathbb R^N} f(x) \widehat g^k(x) w_k(x) dx.\]
\end{lemma}

\begin{proof} This follows from Proposition \ref{P:Kernequiv} and Fubini's theorem. \end{proof}

The Dunkl transform maps
Dunkl operators to multiplication operators, and it therefore suggests itself to consider it
on the Schwartz space $\mathcal S(\mathbb R^N)$ of rapidly decreasing functions on $\mathbb R^N$:

\begin{lemma}\label{Dunkltrafovert} Let $f\in \mathcal S(\b R^N).$ Then
\begin{enumerate}\itemsep=1pt
\item[\rm{(1)}] $\displaystyle \widehat f^{\,k} \in C^\infty(\b R^N)$ and
$\displaystyle  T_j(\widehat f^{\,k})\,=\, -(ix_j f)^{\wedge k}\, $ for $j=1,\ldots, N.$
\item[\rm{(2)}]
 $\displaystyle (T_j f)^{\wedge k}(\xi) = i\xi_j \widehat f^{\,k}(\xi).$
\item[\rm{(3)}] The Dunkl transform leaves the Schwartz space $\mathcal S(\mathbb R^N)$ invariant.
\end{enumerate}
\end{lemma}

\begin{proof} (1) is obvious from \eqref{(1.106)}, and
(2) follows from the anti-symmetry relation  (Prop. \ref{P:Antisym}) for the Dunkl operators.
For (3), it suffices to prove that $\partial^\alpha_\xi (\xi^\beta \widehat f^{\,k}(\xi))$ is bounded for $f\in \mathcal S(\mathbb R^N)$ and arbitrary multi-indices
$\alpha,\,\beta.$ By part (2),  we have $\xi^\beta \widehat f^{\,k}(\xi)= \widehat g^{\,k}(\xi)$ for some $g\in \mathcal S(\b R^N).$ The assertion then follows from the definition of the Dunkl transform and the  growth
bounds of Corollary \ref{C:growthbounds}  on $E_k$.
\end{proof}

It is not hard to see that $\mathcal S(\mathbb R^N)$ is dense in $L^p(\mathbb R^N,w_k)$ for $1\leq p<\infty$. Indeed, the weighted case can be reduced to  the unweighted one, see \cite{dJ1}.
This immediately implies a Riemann-Lebesgue lemma for the Dunkl transform:

\begin{corollary}\label{C:RLebesgue} For $f\in L^1(\b R^N, w_k),$ the Dunkl transform $\widehat f^{\,k}$ belongs to $C_0(\mathbb R^N)$, the space of continuous functions on $\mathbb R^N$ which vanish at infinity.
\end{corollary}

\noindent
The following are the main results for the Dunkl transform; they are in complete analogy to
the corresponding results for the Fourier transform; for details the reader is referred to \cite{dJ1}.

\begin{theorem}\label{T:Dunkltrafo}
\begin{enumerate}\itemsep=-1pt
\item[\rm{(1)}]  ($L^1$-Inversion) If $f\in L^1(\b R^N, w_k)$ with\\
   $\widehat f^{\,k}\in L^1(\b R^N, w_k),$ then
\[ f\,=\, (\widehat f^{\,k}\,)^{\vee k} \quad \text{a.e.}\]
\item[\rm{(2)}] The Dunkl transform is injective on $L^1(\mathbb R^N, w_k)$.
\item[\rm{(3)}] The Dunkl transform is a
  homeomorphism of $\mathcal S(\b R^N)$ with period $4$.
\item[\rm{(4)}] (Plancherel Theorem) The Dunkl transform has a unique
  extension to an isometric
isomorphism of $L^2(\b R^N, w_k)$. The extension is also denoted by
$\,f\mapsto \widehat f^{\,k}$.
\end{enumerate}
\end{theorem}

\begin{proof}(Sketch)
The decisive part  is the $L^1$-inversion. It is first proved for functions of the form
$f(x) = p(x)e^{-|x|^2/2}, \,p\in \mathcal P$, which form a dense subalgebra of $C_0(\b R^N),$ and then extended
to arbitrary $f\in L^1(\mathbb R^N, w_k)$. Part (2) is immediate from (1). Together with Lemma \ref{Dunkltrafovert}(3), this easily implies that the Dunkl transform is a bijection of $\mathcal S(\mathbb R^N)$, and continuity in both directions follows from the closed graph theorem. Part (4) is obtained by a standard
procedure (using Lemma \ref{Dunklsymm}) from the density of $\mathcal S(\b R^N)$ in $L^2(\mathbb R^N, w_k)$.
\end{proof}

As a consequence, we mention the   spectral resolution of the Dunkl Laplacian in $L^2(\mathbb R^N, w_k).$ Via the Dunkl transform, $\Delta_k$ it is  unitarily equivalent to the multiplication operator $M_{-|\xi|^2}$ in $L^2(\mathbb R^N, w_k),$ and this gives

\begin{corollary}
 The Dunkl Laplacian $\Delta_k$ with domain $\mathcal S(\b R^N)$ is essentially self-adjoint in $L^2(\mathbb R^N, w_k).$ The spectrum of its closure is
$\sigma(\overline{\Delta_k}) = (-\infty,0 ].$
\end{corollary}

There have been various approaches to Paley-Wiener theorems for the Dunkl transform, see \cite{dJ3}, \cite{dJ5},  \cite{T2} and \cite{AdJ}. For an open subset $\Omega\subseteq \mathbb R^N$, we denote by $\mathcal D(\Omega)$ the space of
compactly supported smooth functions on $\Omega$ with the usual Fr\'echet-space topology. The most basic variant of the Paley-Wiener theorem is as follows.

\begin{theorem}\label{PW} For $R>0$ consider the ball $B_R=\{x\in \mathbb R^N: |x|\leq R\}$ and let $\mathcal H_{B_R}$ denote the Paley-Wiener space of all entire functions $f$ on $\mathbb C^N$ characterized by the property that for each $M\in \mathbb Z_+$ there exists a constant $\gamma_M >0$ such that \[|f(\lambda)|\leq \gamma_M(1+|\lambda|)^{-M}e^{R|\text{Im}\lambda|} \quad\text{for all }\, \lambda\in \mathbb C^N.\] Then
the Dunkl transform $f\mapsto \widehat f^{\,k}$ is an isomorphism from $\mathcal D(B_R)$ onto
$\mathcal H_{B_R}$.
\end{theorem}

For the  proof of this result see \cite{dJ3} or \cite{dJ5}.
One consequence is the following important extension of the intertwining operator:

\begin{theorem}\label{smooth_V} (\cite{dJ3}, \cite{T2}.)
The intertwining operator $V_k$ extends, via formula \eqref{(4.11)},
to a homeomorphism of $C^\infty(\mathbb R^N)$.
\end{theorem}

We conclude this section with an outlook on the dual of the intertwining operator introduced in \cite{T1}.
The usual Fourier transform on $\mathbb R^N$ will be denoted by $\,f\mapsto \widehat f,$ its inverse by $\,f\mapsto \,f^{\vee}$.

\begin{definition}\label{dual}  The dual of the intertwining operator $V_k$ is defined by
\[ ^tV_k: \mathcal S(\mathbb R^N)\to \mathcal S(\mathbb R^N), \quad f\mapsto \bigl(\widehat f^{\,k}\bigr)^\vee.\]
\end{definition}

By Theorem \ref{T:Dunkltrafo}, $^tV_k$ is a homeomorphism of $\mathcal S(\mathbb R^N)$. It can be considered  as an analogue of the Abel transform on Riemannian symmetric spaces, with the spherical transform being replaced
by the Dunkl transform. Here are some further properties of this operator, justifying also the terminology.

\begin{proposition}\label{dualintertwin}
\begin{enumerate}\itemsep=-1pt
\item[\rm{(1)}] $\displaystyle ^tV_k T_\xi \,=\,\partial_\xi\, ^tV_k \quad \text{on }\, \mathcal S(\mathbb R^N).$
\item[\rm{(2)}] For $f\in \mathcal S(\mathbb R^N)$ and polynomials $p\in \mathcal P$,
\[ \int_{\mathbb R^N} V_k p(x) f(x) w_k(x)dx \,=\, \int_{\mathbb R^N} p(x) \,^tV_k(f)(x)dx\]
\item[\rm{(3)}] Put $\psi(x) = e^{-|x|^2/2}.$ Then for all $q\in \mathcal P$,
\[ ^tV_k(q\psi) \,= \,\bigl(e^{-\Delta/2}e^{\Delta_k/2}q\bigr)\psi.\]
\end{enumerate}
\end{proposition}

\begin{proof} (1) This is immediate from Lemma \ref{Dunkltrafovert}.

\smallskip\noindent
(2) (c.f. \cite{T1}) By definition of $^tV_k$,
\begin{align*} \frac{1}{c_0}\int_{\mathbb R^N} p(x)\, ^tV_k(f)(x)dx\, =&\, \bigl(p(\widehat f^{\,k})^\vee\bigr)^\wedge (0) \,=\,
\bigl(p(i\partial) \widehat f^{\,k}\bigr)(0)\\
=&\,\frac{1}{c_k} p(i\partial_\xi)\Bigl(\int_{\b R^N}f(x) E_k(-ix,\xi)w_k(x)dx\Bigr)\big\vert_{\xi = 0} \\
=&\, \frac{1}{c_k}\int_{\b R^N}f(x) \bigl(p(i\partial_\xi)E_k(-ix,\xi)\bigr)\big\vert_{\xi=0}w_k(x)dx.
\end{align*}
For a monomial $p(x) = x^\alpha$, we obtain from formula \eqref{kernelsum} the identity
\[ p(i\partial_\xi)E_k(-ix,\xi)\big\vert_{\xi=0} \,= \, (i\partial_\xi)^\alpha \sum_{\nu} \frac{V(x^\nu)(-i\xi)^\nu}{\nu!}\big\vert_{\xi=0} \,=\, V_k(x^\alpha) = V_k p(x).\]
By linearity, this extends to all $p\in \mathcal P$, and the assertion follows.

\smallskip\noindent
(3) This results from part (2) with $f=q\psi$ and Lemma \ref{abeltrafo}(2).
\end{proof}


\subsection{Heat kernel and  heat semigroup}\label{heat}

For a fixed root system $R$ and multiplicity $k\geq 0$, the Dunkl-type heat operator is defined by
\[ \Delta_k - \partial_t \quad \text{on }\, \mathbb R^N\times \mathbb R.\]
We consider the following initial-value problem for the generalized heat equation:

\begin{equation*}
\text{(IVP)}\quad \begin{cases}
(\Delta_k - \partial_t)\, u\, =\, 0 & \,\text{on $\b R^N\times (0,\infty)$}\\
     u(\,.\,,0)\, =\, f & \text{}
   \end{cases}
\end{equation*}
with  initial data $f\in C_b(\mathbb R^N)$.
We are searching  for solutions $u\in C^2(\mathbb R^N\times (0,\infty)) \cap C(\mathbb R^N, [0,\infty)).$  We shall
start with the more abstract aspect of this problem, namely the associated operator semigroup, and then determine the heat kernel and the solution of (IVP) in an explicit form. This approach will establish positivity of the heat semigroup
and the heat kernel by a standard approach. As
 a consequence, it implies that
\[ E_k(x,y) >0 \quad \forall \,\, x,y\in \b R^N\]
which is already known from the positive integral representation \ref{C:Bochner} for $E_k$. This integral representation, however,  is a much deeper result.

In the following,
we consider the Dunkl Laplacian as
a densely defined linear operator on the Banach space $(C_0(\b R^N), \|\,.\,\|_\infty)$ with domain $\mathcal S(\b R^N)$.
In the classical case $k=0$, it is well known that $\Delta=\Delta_0$ (more precisely: its closure) generates a Feller semigroup on $C_0(\mathbb R^N)$,
namely the heat semigroup
\[ H_tf(x) =\frac{1}{(4\pi t)^{N/2}}\int_{\b R^N} f(y) e^{-|x-y|^2/4t} dy.\]
Recall that for a locally compact Hausdorff space $\Omega$, a strongly continuous semigroup $(T_t)_{t\geq 0}$ on $(C_0(\Omega), \|.\|_\infty)$
is called a \textit{Feller semigroup}, if
it is contractive and positive, that is $\,\|T_t f\|_\infty \leq \|f\|_\infty$ and $\, f\geq 0$ on $\Omega$ implies that $T_tf \geq 0$ on
$\Omega$ for all $t\geq 0.$

In order to extend the above fact to general multiplicities $k\geq 0$, we employ the following
useful variant of the Lumer-Phillips theorem, which  characterizes
 Feller semigroups
in terms of a ``positive maximum principle'', see e.g. \cite{Kal}, Thm. 17.11.
In fact, this Theorem motivated the positive minimum principle
\ref{T:Posmin}  in the positivity-proof for $V_k$.

\begin{theorem}\label{T:Feller}
Let $A$ be a
densely defined linear operator in $(C_0(\Omega), \|.\|_\infty)$ with domain $\mathcal D(A).$  Then $A$ is closable,
and its closure $\overline A$ generates a Feller semigroup on $C_0(\Omega),$ if and only
if the following conditions are satisfied:
\begin{enumerate}\itemsep=-1pt
\item[\rm{(i)}] If $f\in \mathcal D(A)$ then also $\overline f\in \mathcal D(A)$ and $A(\overline f)
= \overline{A(f)}.$
\item[\rm{(ii)}] The range of $\lambda id -A$ is dense in $C_0(\Omega)$ for some $\lambda >0$.
\item[\rm{(iii)}] If $f\in \mathcal D(A)$ is real-valued with a non-negative maximum in $x_0\in\Omega$, i.e. \\ $0\leq f(x_0) = \max_{x\in\Omega} f(x),$ then $Af(x_0) \leq 0.$ (Positive
maximum principle).
\end{enumerate}
\end{theorem}

The following Lemma implies that $(\Delta_k, \mathcal S(\b R^N))$  satisfies the positive maximum principle:

\begin{lemma}
Let $\Omega\subseteq\b R^N$ be open and $W$-invariant.
If a real-valued function $f\in C^2(\Omega)$ attains an absolute maximum at
$x_0\in\Omega$, i.e. $f(x_0)= \sup_{x\in\Omega} f(x)$, then
\[\Delta_k f(x_0)\,\leq\, 0\,.\]
\end{lemma}

\begin{proof} Recall the explicit expression \eqref{(1.3)} for $\Delta_k$.  We assume first that
$\langle\alpha ,x_0\rangle \not=0$ for all $\alpha\in R$. The fact that $f$ has a maximum at $x_0$ implies that  $\partial_\alpha f(x_0)=0$ for all $\alpha\in R$ and that $\Delta f(x_0)\leq 0.$ Moreover, $f(x_0) \geq f(\sigma_\alpha x_0)$ for all $\alpha\in R$. Hence $\Delta_kf(x_0) \leq 0.$
If $\langle\alpha ,x_0\rangle =0$ for some $\alpha\in R$, then one has to employ a second order Taylor expansion of $f$.  For details see \cite{R2}.
\end{proof}

\begin{theorem}\label{T:heatsemigroup}
 The operator $(\Delta_k, \mathcal S(\b R^N))$ is closable, and its closure $\overline{\Delta_k}$ generates a Feller semigroup
$(H_t)_{t\geq 0}$ on $C_0(\b R^N)$ which is called the  \textit{generalized heat semigroup}.
\end{theorem}

\begin{proof} We have to check the conditions of Theorem \ref{T:Feller}. Condition (i) is obvious and (iii) is an immediate consequence of the previous lemma. (ii) is
also satisfied, because for each $\lambda >0$, the operator $\lambda id  -\Delta_k$ leaves $\mathcal S(\b R^N)$ invariant;  this follows from the fact that the
Dunkl transform  is a  homeomorphism of $\mathcal S(\b R^N)$ and
$\,\bigl((\lambda I  -\Delta_k)f\bigr)^{\wedge k}(\xi) =
(\lambda +|\xi|^2)\widehat f^{\,k}(\xi)$.
Theorem \ref{T:Feller} now implies the assertion.
\end{proof}

We expect that the heat semigroup $(H_t)_{t\geq 0}$ can be written explicitly in terms of a generalized heat kernel.
In order to find it, we consider first a slight modification of the usual Gaussian kernel:
\[ g_k(x,t) := \frac{1}{(2t)^{\,\gamma+N/2}c_k} e^{-|x|^2/4t} \quad (x\in \mathbb R^N, \, t>0).\]

\begin{lemma}\label{L:Fundsol}
\begin{enumerate}\itemsep=-1pt
\item[\rm{(1)}] $g_k$ solves the Dunkl-type  heat equation
 $(\Delta_k  - \partial_t) u = 0$  on $\b R^N\times
(0,\infty).$
\item[\rm{(2)}]
$\displaystyle \int_{\b R^N} g_k(x,t)\, w_k(x) dx\,=\, 1 \quad \text{for all}\>\> t>0.$
\item[\rm{(3)}] $\displaystyle \widehat g_k^k(\xi,t)\,=\, c_k^{-1}e^{-t|\xi|^2}.$
\end{enumerate}
\end{lemma}

\begin{proof} For (1),  use the product rule \eqref{(1.2)} as well as the identity
$\sum_{i=1}^N T_i(x_i) = N+2\gamma.$ Part (2) is immediate, and (3) results from the second reproducing property in Proposition  \ref{P:Reprod}.
\end{proof}

\noindent
The Gaussian $g_k$  generalizes the fundamental solution for the classical heat equation. In the classical case $k=0$, the
heat  kernel
is  obtained by translation from
the  fundamental solution. In the Dunkl setting, it is indeed also possible to define a generalized translation
which matches the action of the Dunkl transform, i.e. makes it a homomorphism on suitable function spaces. For our present purposes, it will be sufficient to consider the Schwartz space. A more general setting will be treated in Section \ref{generalized translation}.

\begin{definition} [\cite{R2}] On the Schwartz space $\mathcal S(\mathbb R^N)$, the Dunkl-type \textit{generalized translation} is defined by
\[
 \tau_y f(x):= \frac{1}{c_k} \int_{\b R^N} \widehat f^{\,k}(\xi)
\,E_k(ix,\xi) E_k(iy, \xi)\, w_k(\xi) d\xi; \quad y\in \b R^N.
\]
\end{definition}

Note that for $k=0$ this definition reduces to  $\,\tau_y f(x) = f(x+y).$
In the rank one  case, our generalized
 translation determines the convolution of a so-called signed
hypergroup structure which was defined in \cite{R1}; see also \cite{Ros}. This will be discussed in more detail
in Section \ref{generalized translation}.
Similar structures are conjectured in higher-rank cases, but not established so far.
Note that
  $\tau_y f(x) = \tau_x f(y)$;
moreover, the inversion theorem for the Dunkl transform assures that $\,\tau_0 f = f$ and
\begin{equation}\label{transtrafo} (\tau_y f)^{\wedge k}(\xi)
= E_k(iy,\xi) \widehat f^{\,k}(\xi).\end{equation}
 From this it is easy to see  that
$\tau_y f$ belongs to $\mathcal S(\b R^N)$ again.

\smallskip
Let us return to the Gaussian kernel $g_k$.
From the definition of the generalized translation, part (3) of Lemma \ref{L:Fundsol} and a further application of
 the reproducing formula (2) of Proposition \ref{P:Reprod}, we obtain

\begin{equation}
 \tau_{-y}F_k(x,t)\,=\, \frac{1}{(2t)^{\gamma+N/2}c_k} e^{-(|x|^2 + |y|^2)/4t}\,
E_k\Bigl(\frac{x}{\sqrt{2t}},\frac{y}{\sqrt{2t}}\Bigr).\notag
\end{equation}

\noindent
This motivates the following

\begin{definition}
The generalized heat kernel $\Gamma_k$ is defined for $x,y\in \b R^N$ and $t>0$ by
\begin{equation}\label{heatkernel}\Gamma_k(t,x,y):=\, \frac{1}{(2t)^{\gamma +N/2}c_k}\,e^{-(|x|^2 + |y|^2)/4t}\,
E_k\Bigl(\frac{x}{\sqrt{2t}},\frac{y}{\sqrt{2t}}\Bigr).\end{equation}
\end{definition}

Notice that $\,x\mapsto \Gamma_k(t,x,y)$ belongs to
$\mathcal S(\b R^N)$ for fixed $y$ and $t$.
It follows from  Corollary \ref{C:Bochner} that $\Gamma_k$ is strictly positive. However, we shall obtain this from the positivity of the heat semigroup, which is a more basic result; see
Theorem \ref{T:heatexplicit} below.
Let us first collect some further fundamental properties of the heat kernel which are all
more or less straightforward.

\begin{lemma}\label{L:Heatprop} The heat kernel $\Gamma_k$ has the following properties:
\begin{enumerate}\itemsep=1pt
\item[\rm{(1)}] $\,\displaystyle  \Gamma_k(t,x,y) \,=\,
 c_k^{-2} \int_{\b R^N} e^{-t|\xi|^2}\, E_k(ix,\xi)\,
E_k(-iy,\xi) \, w_k(\xi) d\xi\,.$
\item[\rm{(2)}] $\,\displaystyle \int_{\b R^N} \Gamma_k(t,x,y)\,w_k(y)
dy\,=\,1.$
\item[\rm{(3)}] $\,\displaystyle \Gamma_k(t+s,x,y)\,=\, \int_{\b R^N}
  \Gamma_k(t,x,z)\,\Gamma_k(s,y,z)\,w_k(z)dz.$
\item[\rm{(4)}] For fixed $y\in \b R^N$, the function $u(x,t):=
\Gamma_k(t,x,y)$ solves the generalized heat equation $\, \Delta_k u =
\partial_t u\,$ on
$ \b R^N\times(0,\infty)$.
\end{enumerate}
\end{lemma}

\begin{proof} (1) is clear from the definition of generalized translations.
 For details concerning (2)
see \cite{R2}. (3)
is obtained by inserting (1) for one of the kernels in the integral, and
(4) is immediate from (1).
\end{proof}

\begin{theorem} \begin{enumerate}\label{T:heatexplicit}
                 \itemsep=-1pt
\item[\rm{(1)}]  The heat semigroup $(H_t)$ on $C_0(\mathbb R^N)$  is given explicitly by
 \[ H_tf(x) = \int_{\b R^N} \Gamma_k(t,x,y) f(y)\,w_k(y)dy \quad \text{ for }\, t>0.\]
\item[\rm{(2)}] The heat kernel $\Gamma_k$ is strictly positive on $\b R^N\times \b R^N\times (0,\infty)$.
\item[\rm{(3)}] For $f\in C_0(\mathbb R^N), $ the  function $u(x,t) = H_tf(x)$  solves  initial value problem (IVP).
\end{enumerate}
\end{theorem}

As a Corollary, one obtains in particular that $\,E_k(x,y) >0$ for all $x,y\in \mathbb R^N$.

\begin{proof}
Define $u_f$ on $\mathbb R^N\times [0,\infty)$ by
\[ u_f(x,t) :=\begin{cases}\displaystyle
\int_{\b R^N} \Gamma_k(t,x,y) f(y)\,w_k(y)dy &
\text{if $\,\,t>0$},\\
     f(x) & \text{if $\,\,t=0.$}
\end{cases}\]
From part (5) of Lemma \ref{L:Heatprop} it is immediate that $u_f$ solves the generalized heat equation on $\mathbb R^N\times (0,\infty)$. We proceed further in several steps.

\smallskip\noindent
\textbf{Step 1.}
Consider first initial data $f\in \mathcal S(\mathbb R^N)$.  By  Lemma \ref{L:Heatprop} (1)  and Fubini's theorem, we obtain that  for $t>0,$
\begin{equation}\label{(tocauchy)}
u_f(x,t)\,=\, c_k^{-1} \int_{\b R^N} e^{-t|\xi|^2}
\widehat f^{\,k}(\xi)\, E_k(ix,\xi)\,  w_k(\xi) d\xi.
\end{equation}
 In view of the inversion theorem for the Dunkl transform,
this identity also extends to $t=0$. It follows that $x\mapsto u_f(x,t)\in \mathcal S(\mathbb R^N)$ for all $t>0$ and
$\displaystyle \| u_f(\,.\,,t)  - f\|_{\infty} \, \to 0\> $ as
$t\downarrow 0$. Thus  $u_f$ solves (IVP) for $f\in \mathcal S(\b R^N).$
Next, we prove that $\, u_f(x,t) = H_tf(x)$ for all $x\in \mathbb R^N$ and $t\geq 0$. For this, recall
from semigroup theory that the function $t\mapsto H_tf$ is the unique solution of the abstract
Cauchy problem
\[
\begin{cases}
\displaystyle\frac{d}{dt}u(t)= \overline\Delta_k u(t) & \text{for $t>0$},\\
     u(0) = f & \text{}
   \end{cases}
\]
within the class of all (strongly) continuously differentiable functions $u$
on $[0,\infty)$ with values in  $(C_0(\b R^N), \|.\|_\infty).$
On the other hand, if $f\in \mathcal S(\b R^N)$, then also $u_f(\,.\, ,t) \in \mathcal S(\b R^N)$ and from  formula \eqref{(tocauchy)}
it is easily deduced that
$t\mapsto u_f(\,.\,,t)$ also solves the abstract Cauchy problem. This proves parts (1) and (3) for initial data from $\mathcal S(\b R^N).$

\smallskip\noindent
\textbf{Step 2.} Part (2), which is of course also  an immediate  consequence of the positive  integral representation  for the Dunkl kernel,
 can  be directly deduced from the positivity of the heat semigroup $(H_t)_{t\geq 0}$ on $\mathcal S(\b R^N),$ which gives $\Gamma_k\geq 0.$ Strict positivity then follows from formula (4) of
Lemma \ref{L:Heatprop}.

\smallskip \noindent
\textbf{Step 3.}
Consider now  initial data $f\in C_0(\b R^N).$ By the density of $\mathcal S(\b R^N)$ in $C_0(\b R^N)$ and the
positivity of $\Gamma_k$, a standard approximation argument yields that $\,H_t f(x)= u_f(x,t)$ for all $f\in C_0(\b R^N)$
(and  $x\in \b R^N, t\geq 0$.)
By the strong continuity of the semigroup $(H_t) $ it further follows that $u_f$ is continuous on $\mathbb R^N\times [0,\infty)$ and thus  $u_f$ solves (IVP).
This finishes the proof of the Theorem.

\end{proof}

Involving the estimates of Corollary \ref{C:growthbounds} on $E_k$, one further obtains
\[ \Gamma_k(t,x,y)\,\leq\,
  \frac{1}{(2t)^{\gamma+N/2}c_k}\,\max_{w\in W} e^{-|wx-y|^2/4t}.\]

The heat operators $(H_t)_{t\geq 0}$ naturally extend to functions $f\in C(\b R^N)$ which satisfy the subexponential growth condition
\[ \forall \epsilon > 0 \, \,\exists C_\epsilon >0: \, |f(x)|\leq\, C_\epsilon\cdot e^{\epsilon |x|^2},\]
by
\[ H_tf(x) := \int_{ \b R^N} \Gamma_k(t,x,y) f(y) w_k(y)dy \quad \text{for }\, t>0.\]
This growth condition is in particular satisfied by functions from $C_b(\b R^N)$ and by polynomials.
Based on the above results,  it is checked by standard arguments that for general initial data $f\in C_b(\b R^N)$,
(IVP)  is solved by
$\, u(x,t) = H_tf(x)$.
Uniqueness of the solution within classes of functions satisfying suitable exponential growth conditions
is established by means of a maximum principle, just as with the classical heat equation. For details
on this, the interested reader is refered to  \cite {R2}.

\smallskip
We conclude this section  with a useful observation which will be important later on:

\begin{proposition}\label{heat_polynomial}
Let  $p\in \mathcal P$ be a polynomial. Then
\[ H_t p = e^{t\Delta_k} p.\]
Moreover, the function $u(x,t) = e^{t\Delta_k}p(x)\,$ is a polynomial solution of the initial value problem (IVP) with
initial data $p$.
\end{proposition}
 The polynomial  $H_tp$ is of the  same degree as $p$. It is called the \textit{heat polynomial} associated with $p$.

\smallskip\noindent
For the proof of the proposition,  the following scaling lemma is needed.

\begin{lemma}\label{scalinglemma} Let $p\in \mathcal P_n$. Then for $c\in \b C$ and $a\in \b C\setminus \{0\},$
\[ \bigl(e^{c\Delta_k}p\bigr)(ax) = \,a^n\bigl(e^{a^{-2}c\Delta_k}\bigr)p(x).\]
\end{lemma}

\begin{proof}
This is easily checked in terms of the exponential series for $e^{c\Delta_k}$.
\end{proof}

\begin{proof}[Proof of Proposition \ref{heat_polynomial}.] Part (1) of Proposition \ref{P:Reprod} can be written as
\[ p(x) = \int_{\b R^N} \Gamma_k\bigl(1/2,x,y\bigr)e^{-\Delta_k/2}p(y)\, w_k(y)dy \quad \text{for }\,p\in \mathcal P.\]
Replacing $p$ by $\,e^{\Delta_k/2}p$, one obtains the assertion for $t=1/2.$ The general case $t>0$ follows
by rescaling.
\end{proof}

\subsection{Calogero-Moser-Sutherland models and generalized \\Hermite polynomials}

Quantum Calogero-Moser-Sutherland (CMS) models describe quantum mechanical
 systems of identical particles on a circle or line which interact pairwise
 through potentials of inverse square type. They
 have gained some interest in mathematical physics due to their quantum-integrability. Among the broad literature in this area, we refer to \cite{BF1}-\cite{BF3}, \cite{He2}, \cite{HS}, \cite{vD}, \cite{K}, \cite{LV}, \cite{Pa},  as well as the monograph \cite{vDV}.

The Schr\"odinger operator of a CMS model for $N$ particles on a line or circle is given by
\[ \mathcal H = -\Delta + g\sum_{1\leq i<j\leq N} \frac{1}{d(x_i,x_j)^2} \]
where $g\geq -1/2$ is a coupling constant, $x_i$ is the position of particle $i$ and $d(x,y) =|x-y|$ for the
linear model, while $\, d(x,y) = \frac{1}{\pi} \sin\bigl(\pi(x-y)\bigr)$ for a model on a circle with circumference $1$.
The models on a line, initially studied by Calogero
 in \cite{Ca}, are closely related to  rational Dunkl operators of type $A_{N-1}$,  while those one a circle, going back to Sutherland \cite{Su}, are related to the trigonometric Dunkl operators of Heckman and Opdam. In order to obtain a discrete spectrum in the linear case, one has to
add some external potential, the most common one being of the form $\omega^2|x|^2$ with $\omega>0$ (harmonic confinement).
Dealing with identical particles, one considers the linear CMS operator in the so-called bosonic state space
\[ B= \{ f\in L^2(\mathbb R^N): \sigma_{ij}f = f \quad \forall \,\,i,j\}\]
where $\sigma_{ij}$ permutes the coordinates $x_i$ and $x_j$.
After explicit spectral
resolutions of  CMS models had already been obtained by Calogero and
 Sutherland, Moser \cite{Mo} proved complete
integrability of  the associated
classical Hamiltonian systems. But the deeper  algebraic structure  of the quantum CMS models became clear only
in the nineteen-nineties by independent work of
 \cite{Po} and \cite{He2}. For the free linear model, the basic idea is to
consider the modification
\[ \widetilde{\mathcal H}:=\, -\Delta + 2k \sum_{i<j} \frac{1}{(x_i-x_j)^2}(k\cdot id-\sigma_{ij})\]
acting in $L^2(\mathbb R^N)$. When  $g=2k(k-1),$ then $\, \widetilde{\mathcal H}\vert_B = \mathcal H.$
A short calculation, using
results from \cite{Du1}, gives
\[ w_k^{-1/2} \widetilde{\mathcal H} w_k^{1/2} \,=\, -\Delta_k^S\,\]
where $\Delta_k^S$ denotes the Dunkl Laplacian associated with the symmetric group $S_N$,
c.f. Example \ref{E:dunklops}(2), and $w_k$ is the weight function of type $S_N$,
\[ w_k(x) =\prod_{i<j} |x_i-x_j|^{2k}.\]
Now consider the algebra $\mathcal P^{S_N}$ of $S_N$-invariant polynomials on $\b R^N$.
It is generated by the $N$ elementary symmetric polynomials
\[ s_j(x) = \sum_{1\leq i_1 < \ldots <i_j\leq N} x_{i_1} \cdot\ldots\cdot  x_{i_j}, \quad j=1, \ldots, N.\]
Let $T$ stand for the Dunkl operators of type $A_{N-1}$ with multiplicity $k$. Then
\[\mathcal A:= \{ \text{Res}\, p(T): p\in \mathcal P^{S_N}\} \]
is a commutative algebra of differential operators on $\mathcal P^{S_N}$ containing the operator
\[\text{Res}(\Delta_k^S) \,=\,-  w_k^{-1/2} \mathcal Hw_k^{1/2},\]
c.f. Section \ref{Dunkloperators}.
Up to conjugation with $w_k^{1/2}$,  $\mathcal A$ is the so-called algebra of \textit{quantum integrals} for the CMS operator $\mathcal H$.  It is generated by the $N$ algebraically independent elements
$\text{Res}(s_j(T)), \, j=1, \ldots, N.$

\smallskip

There exist obvious generalizations of the classical CMS
models in the context of
abstract root systems: Suppose
$R$ is an arbitrary (reduced) root system in $\b R^N$  and $k$ a nonnegative multiplicity function,
 then the  corresponding  abstract  Calogero
Hamiltonian  is  given by
\[ \mathcal H \,=\, -\Delta + \,2\sum_{\alpha\in R_+}
k_\alpha(k_\alpha-1) \frac{1}{\langle \alpha,x\rangle^2}.\]
For the classical root systems, Olshanetsky and Perelomov proved quantum integrability of this model, following the method of Moser via Lax pairs.
Again, we consider a modification involving
 reflection terms:
\begin{equation}\label{(3.71)}  \widetilde{\mathcal H}\,=\, -\Delta +2\sum_{\alpha\in R_+}
 \frac{k_\alpha}{\langle\alpha,x\rangle^2}\,(k_\alpha
 -\sigma_\alpha)\,.\end{equation}
In this case,
\[ w_k^{-1/2} \widetilde{\mathcal H}\,w_k^{1/2} \,=\, -\Delta_k\,,\]
see \cite{R4}.
The quantum integrals for $\mathcal H$ are constructed just as in the $S_N$ case. According to a  classical theorem of Chevalley (see e.g. \cite{Hu}),
 the algebra  $\mathcal P^W$ of $W$-invariant polynomials is again generated by $N$
homogeneous, algebraically independent elements, providing a basis of quantum integrals in this case.

\medskip
Let us now turn to the  spectral analysis of abstract linear CMS operators with harmonic potential $\omega^2|x|^2$. We follow  \cite{R2}, \cite{R7}  and use the normalization $\omega= 1/2.$ (Other normalizations lead to results which are equivalent up to  scaling).
We  work with the gauge-transformed version with reflection terms,
\[H_k:=
   -\Delta_k + \frac{1}{4}|x|^2.\]
Due to the anti-symmetry of the first order Dunkl operators (Prop. \ref{P:Antisym}),
   this operator is symmetric and densely defined
in $L^2(\b R^N, w_k)$ with domain $\mathcal S(\b R^N)$. Note that in case $k=0$, $H_k$  is just the
 Hamiltonian of the $N$-dimensional isotropic harmonic oscillator.

The next theorem contains a  complete description of the spectral properties
of $\mathcal H_k$  and
generalizes well-known
facts for the classical harmonic oscillator.

\begin{theorem} \label{T:Spectral}(Spectral Theorem for $H_k$)
$L^2(\b R^N, w_k)$ decomposes as an orthogonal Hilbert space sum
according to
\[ L^2(\b R^N, w_k)\,=\, \bigoplus_{n\in \b Z_+} V_n \]
where
\[V_n\,:= \{e^{-|x|^2/4}e^{-\Delta_k/2}\, p(x)\,: p\in \mathcal P_n\} \subset  \mathcal S(\b R^N)\]
 is the eigenspace of $H_k$
 corresponding to the eigenvalue $\,n + \gamma+N/2.$ In particular,
$H_k$ is essentially self-adjoint. The spectrum of its closure is purely discrete and given by
\[ \sigma(\overline{\mathcal H_k}) = \{ n+\gamma+N/2, n\in \b Z_+\}.\]
\end{theorem}

For details on the proof, the reader is referred to \cite{R2} or \cite{R7}.
It relies on the
$sl(2)$-commutation relations of the operators
\[ E:= \frac{1}{2}|x|^2,\>\> F:= -\frac{1}{2}\Delta_k \quad\text{and}\>\>
H:= \sum_{i=1}^N x_i\partial_i +\left(\gamma +N/2\right)\]
which were observed by Heckman \cite{He2}, namely
\[
\big[H,E\big] = 2E,\>\> \big[H,F\big] = -2F,\>\>\big[E,F\big] = H.
\]
The first two relations are consequences of the fact that the  Euler operator
\begin{equation}\label{(3.5a)}\rho:=\sum_{i=1}^N x_i\partial_i\,
\end{equation}
 satisfies
$\,\rho(p)=np\,$ for each homogeneous $p\in \mathcal P_n$.

\medskip

The eigenvalues of the CMS Hamiltonian $\mathcal H_k$ are
highly degenerate if $N>1$.
We are now going to construct natural orthogonal bases for
them. They are made up by generalizations of the
 classical $N$-variable Hermite polynomials and Hermite functions
to the
Dunkl setting.
The starting point for our construction
is the Macdonald-type  identity  \eqref{(1.10)}: for all $p, q\in \mathcal P$,
\[
 [p,q]_k\,=\, \frac{1}{c_k}\int_{\b R^N} e^{-\Delta_k/2}p(x) e^{-\Delta_k/2}q(x)e^{-|x|^2/2}w_k(x)dx.\]

Notice that $[\,.\,,\,.\,]_k$ is a scalar product on the $\b R$- vector space $\mathcal P_{\b R}$ of
 polynomials with real coefficients.
Let $\{\phi_\nu\,, \nu\in \b Z_+^N\}$ be an orthonormal basis of $\mathcal P_{\b R}$
with respect to  $[.\,,.]_k$ such that $\phi_\nu\in
\mathcal P_{|\nu|}$. Write $\mathcal P_{n,\b R} = \mathcal P_n\cap \mathcal P_{\b R}.$ As homogeneous polynomials of different (total) degrees are orthogonal,
 the $\phi_\nu$ with fixed $|\nu|=n$ can for example be
constructed by Gram-Schmidt orthogonalization within $\mathcal P_{n,\b R}$  from an
arbitrary ordered real-coefficient basis. If $k=0$, the
 canonical choice of the basis $\{\phi_\nu\}$ is just $\,\phi_\nu(x):=
(\nu!)^{-1/2} x^\nu.$

\begin{definition}\label{Hermitepolys}
The generalized Hermite polynomials $\{H_\nu\}$ and Hermite functions $\{h_\nu\}\,\,(\nu\in \b Z_+^N)$
associated with the basis $\{\phi_\nu\}$ of $\mathcal P_{\b R}$ are defined  by
\[H_\nu(x):= e^{-\Delta_k/2}\phi_\nu(x); \quad h_\nu(x):= e^{-|x|^2/4}H_\nu(x).\]
\end{definition}

Observe that $H_\nu$ is a polynomial of degree $|\nu|$. By the Macdonald identity,  the Hermite functions ${h_\nu}$  form an
orthogonal basis of $L^2(\b R^N,w_k).$

\smallskip

For $k=0$ and the choice $\phi_\nu(x) = (\nu!)^{-1/2}x^\nu$, one obtains the classical multivariable Hermite polynomials
 \[H_\nu(x)\,=\, \frac{1}{\sqrt{\nu!}}\prod_{i=1}^N
  e^{-\partial_i^2/2}(x_i^{\nu_i})\,=\,
  \frac{2^{-|\nu|/2}}{\sqrt{\nu!}}\,\prod_{i=1}^N \widehat
  H_{\nu_i}(x_i/\sqrt 2)\]
where the $\widehat H_n$ are the classical one-variable Hermite polynomials
\[ \widehat H_n(x)\,=\, (-1)^n\,e^{x^2}\,\frac{d^n}{dx^n}e^{-x^2}\,.\]
More interesting examples are the following:

\begin{examples}\label{E:Ex1} \begin{enumerate}\itemsep=1pt
\item[\rm{(1)}] \textit{The one-dimensional case.} Up to sign changes, there exists only one
    orthonormal basis
   with respect to $[\,.\,,\,.\,]_k$. The
  associated generalized Hermite polynomials $(H_n^k)_{n\in \mathbb Z_+}$ are orthogonal with respect to the weight $|x|^{2k} e^{-|x|^2}$ on $\mathbb R$.
They can be found in \cite{Chi} and were further studied in \cite{Ros}.
We mention that they can be
  written as
\[ \begin{cases}
     H_{2n}^k(x) = (-1)^n 2^{2n} n! \,L_n^{k-1/2}(x^2),\\
     H_{2n+1}^k(x) = (-1)^n 2^{2n+1} n! \,xL_n^{k+1/2}(x^2);
   \end{cases}\]
where  the $L_n^\alpha$ are the usual Laguerre polynomials
\[ L_n^\alpha(x) = \frac{1}{n!} x^{-\alpha} e^x\,\frac{d^n}{dx^n}
\Bigl( x^{n+\alpha} e^{-x}\Bigr).\]
\item[\rm{(2)}] \textit{The $A_{N-1}$-case.}
There exists a natural orthogonal system $\{\varphi_\nu\}$ made up by the so-called
\textit{non-symmetric Jack polynomials} $\{E_\nu = E_\nu^k, \,\nu\in \b Z_+^N\}$.
They were introduced  in \cite{Op} for arbitrary root systems and
 are characterized
by the following conditions:\parskip=-1pt
\begin{enumerate}
\item[\rm{(i)}] $E_\nu$ is homogeneous of degree $|\nu|$ and of the form
\[E_\nu(x)\,=\, x^\nu \,+\, \sum_{\mu <_P\,\nu}
c_{\nu,\,\mu} x^\mu \quad\text{ with }\,\, c_{\nu,\mu}\in \b R;\]
\item[\rm{(ii)}] For all $\mu <_P\,\nu$, $\displaystyle
\bigl(E_\nu(x), x^\mu\bigr)_k\,= 0\,$
\end{enumerate}
\smallskip\noindent
Here $\, <_P\, $ is a dominance order defined within multi-indices
of equal total length (see \cite{Op}), and
the inner product $(.,.)_k$ on $\mathcal P_{\b R}$ is given by
\[ (f,g)_k\,:=\, \int_{\b T^N} f(z) \overline{g( z)}\, \prod_{i<j} |z_i-z_j|^{2k} dz\]
with $\b T=\{z\in \b C: |z|=1\}$ and $dz$  the Haar measure on the torus $\b T^N$.
If $f, g$ are homogeneous of different degrees, then $(f,g)_k=0$. The set $\{E_\nu,\, |\nu|=n\}$ forms a vector space basis of
$\mathcal P_{n,\b R}.$ It can be shown (by use of $A_{N-1}$-type Cherednik
operators) that the Jack polynomials $E_\nu$  are also
orthogonal with respect to the Dunkl pairing
$[\,.\,,\,.\,]_k$; for details see \cite{R2}. The corresponding generalized Hermite polynomials and their
symmetric counterparts
have been
studied in \cite{L1}, \cite{L2}, \cite{vD},  and in \cite{BF1} - \cite{BF3}.
\end{enumerate}
\end{examples}

As an immediate consequence of Theorem \ref{T:Spectral} we obtain analogues of the classical second order differential equations for generalized Hermite polynomials and Hermite functions:

\begin{corollary}\label{C:diffgleichungen}  \begin{enumerate}\itemsep=-1pt
\item[\rm{(i)}] $\displaystyle \bigl(-\Delta_k + \sum_{i=1}^N x_i\partial_i\bigr) H_\nu \,=\, |\nu| H_\nu\,.$
\item[\rm{(ii)}] $\displaystyle \bigl(-\Delta_k +\frac{1}{4}|x|^2\bigr) h_\nu\,=\, (|\nu|+\gamma+N/2) h_\nu\,.$
\end{enumerate}
\end{corollary}

Various further useful properties of the classical Hermite polynomials and Hermite functions have
extensions to our general setting. We conclude this section with a list of them. The proofs
can be found  in  \cite{R2}. For further results on generalized Hermite polynomials, one
can also see for instance \cite{vD}.

\begin{theorem}\label{T:Hermiteproperties} Let $\{H_\nu\}$  and $\{h_\nu \}$ be the Hermite polynomials and Hermite
functions associated with the basis $\{\phi_\nu\}$ on $\b R^N$ and let  $x,y\in \b R^N$. Then
\begin{enumerate}\itemsep=1pt
\item[\rm{(1)}]  $\displaystyle H_\nu(x) \,=\, (-1)^{|\nu|} e^{|x|^2/2}
\varphi_{\nu}(T) e^{-|x|^2/2}\,$  (Rodrigues-Formula)
\item[\rm{(2)}] $\displaystyle e^{-|y|^2/2} E_k(x,y)\,=\, \sum_{\nu\in \b Z_+^N} H_\nu(x)\varphi_\nu(y)\, $ (Generating relation)
\item[\rm{(3)}] (Mehler formula) For $r\in \b C$ with $|r|<1,$
{\small \[\sum_{\nu\in \b Z_+^N} H_\nu(x) H_\nu(y) r^{|\nu|} =
\frac{1}{(1-r^2)^{\gamma +\frac{N}{2}}}\,
{\rm exp}\left\{-\frac{r^2(|x|^2 + |y|^2)}{2(1-r^2)}\right\}
E_k\!\Bigl(\frac{rx}{1-r^2},y\Bigr).\]}

\end{enumerate}
The sums in (2) and (3) are absolutely convergent.
\end{theorem}
The Dunkl kernel $E_k$ in (2) and (3) replaces the usual exponential function.
It comes in via the following relation with the (arbitrary!) basis $\{\varphi_\nu\}$:
\[ E_k(x,y) \,=\, \sum_{\nu\in \b Z_+^N} \varphi_\nu(x)\varphi_\nu(y).\]

\begin{proposition}\label{P:Hermiteortho}
The generalized Hermite functions $\{h_\nu\,,\, \nu\in \b Z_+^N\}$
are an orthogonal  basis of eigenfunctions for the Dunkl transform on $L^2(\b R^N, w_k)$ with
\[ h_\nu^{\wedge k} \,=\, (-i)^{|\nu|} h_\nu.\]
\end{proposition}

\subsection{Generalized translation and spherical means}\label{generalized translation}

We recall the definition of the generalized translation on the Schwartz space $\mathcal S(\b R^N)$ introduced in Section \ref{heat},
\[
 \tau_y f(x):= \frac{1}{c_k} \int_{\b R^N} \widehat f^{\,k}(\xi)
\,E_k(ix,\xi) E_k(iy, \xi)\, w_k(\xi) d\xi; \quad y\in \b R^N.
\]
In addition to the properties already mentioned in Section \ref{heat},  we state the following relation which follows from \eqref{transtrafo}   and the Plancherel theorem for the Dunkl transform: For all $f,g\in \mathcal S(\b R^N),$
\begin{equation}\label{adjointrel}
 \int_{\b R^N} f(x*_ky) g(y) w_k(y)dy \,=\, \int_{\b R^N} f(y) g(-x*_k y)w_k(y)dy.
\end{equation}
 In \cite{T2}, this translation was extended to $C^\infty(\b R^N)$ via
\[ \tau_yf(x) := V_k^x V_k^y(V_k^{-1}f)(x+y)\]
where the superscripts denote the relevant variable. Indeed, it  is shown in \cite{T2} that both definitions coincide on
$\mathcal S(\b R^N)$.
Note that by Theorem \ref{smooth_V} $\tau_y$ is continuous with respect to the usual Fr\'echet space topology, and that
\[ \tau_0f = f, \quad T_\xi \tau_yf = \tau_yT_\xi f \quad\text{and }\, \tau_yf(x) = \tau_xf(y) \quad \forall \,\, x,y\in \b R^N.\]
We shall frequently use the more suggestive notion
\[ f(x*_ky)= \tau_yf(x).\]
For $k=0$, one obtains the usual group translation: $\, f(x*_0y) = f(x+y).$ It is also immediate from the definition that
\begin{equation}\label{kernelprod}
 E_k(x*_ky,z) = E_k(x,z) E_k(y,z) \quad \forall z\in \b C^N.
\end{equation}
By the Plancherel theorem and the fact that $\,|E_k(ix,\xi)|\leq 1$ for all $x,\xi\in \b R^N$,
the translation operator $\tau_y$ extends to a continuous linear operator on $L^2(\b R^N,w_k)$ with $\|\tau_y\|\leq 1.$
It is, however, an open question in general whether $\tau_y$ is also bounded as a linear operator on the spaces $L^p(\b R^N,w_k)$ with $1\leq p <\infty$, $p\not=2.$
 Only in rank one, this is known to be true so far. Let us briefly describe the situation in
this case.

\smallskip
\textbf{The rank-one case.}
Recall the explicit formula \eqref{kern_rank1}   for $E_k$ in terms of one-variable Bessel functions $j_\alpha$ in this case. It is well-known (see e.g. [BH], 3.5.61)
that the
$j_\alpha$ with $\alpha\geq -1/2$ satisfy the product formula
\begin{equation}\label{besselprod_onevar} j_\alpha(xz) j_\alpha(yz) = \int_0^\infty j_\alpha(\xi z) m_\alpha(x,y,z) z^{2\alpha +1} dz\end{equation}
with the kernel
{\small \[m_\alpha(x,y,z) = \frac{2^{1-2\alpha}\Gamma(\alpha +1)}{\sqrt{\pi}\,\Gamma(\alpha +\frac{1}{2})}\cdot \frac{[(z^2-(x-y)^2)((x+y)^2-z^2)]^{\alpha-1/2}}{(xyz)^{2\alpha}}\cdot 1_{[|x-y|, x+y]}(z)\]}
We remark that formula \eqref{besselprod_onevar} induces a convolution of point measures on $[0,\infty)$ according to
\[ d(\delta_x * \delta_y)(z) := m_\alpha(x,y,z)z^{2\alpha +1} dz.\]
This definition naturally extends to a weakly continuous, bilinear and \\
probability-preserving convolution on the space $M_b([0,\infty))$ of regular bounded Borel measures on $[0,\infty).$ This convolution induces the
structure of a so-called commutative \textit{hypergroup} on $[0,\infty)$ which is called the
\textit{Bessel-Kingman hypergroup of index $\alpha$}. There will be more on hypergroups and this important example in Section
\ref{S1}.

The Dunkl kernel $E_k$ itself satisfies a similar product formula which was proven
in \cite{R1}, namely
\[E_k(x,z) E_k(y,z) = \int_{\b R} E_k(\xi,z) d\mu_{x,y}^k(\xi) \quad \forall z \in \b C\]
with the measures
\[ d\mu_{x,y}^k(z) = m_{k-1/2}(|x|,|y|,|z|)|z|^{2k}\cdot \frac{1-\sigma_{x,y,z} + \sigma_{z,x,y} +\sigma_{z,y,x}}{2}\,dz\]
where
\[ \sigma_{z,x,t} =
\begin{cases} \frac{z^2 + x^2 - t^2}{2zx} & \text{ if }\, z,x\not= 0,\\
                     0 & \text{else.}
                    \end{cases}\]
Therefore the generalized translation on $\mathcal S(\b  R)$ can be written as
\[ f(x*_k y) = \int_{\b R} f(z)d\mu_{x,y}^k(z).\]
The measures $\mu_{x,y}^k$ are not positive in the generic case, but uniformly bounded with
$\|\mu_{x,y}^k\|\leq 4$ for all $x,y\in \b R$. The convolution of point measures defined by
$\delta_x *_k\delta_y := \mu_{x,y}^k$ extends to a weakly continuous convolution on $M_b(\b R)$
which is however not positivity-preserving. It induces the structure of a \textit{signed} hypergroup on $\b R$,
see \cite{R1}.
Due to the uniform boundedness of the measures $\mu_{x,y}^k,$ the generalized translation operators $\tau_y$ on the Schwartz space extend to  bounded linear operators on the
spaces $L^p(\b R, |x|^{2k}dx)$ with $\,\|\tau_y f\|_p \leq 4\|f\|_p$ for all indices $p$ with $1\leq p<\infty$.
For details, the reader may see \cite{R1} and the references cited there.

\medskip
Even if no result of this kind is available in the general case as far, there is at least a useful partial
result which states that the Dunkl-type generalized translation is positivity-preserving and
$L^p$-bounded when restricted to radial functions. Moreover, there is a
weakened form of positive product formula for the Dunkl kernel $E_k$ available.
The key for this is the positivity of the
spherical mean operator in the Dunkl-setting. This operator was first considered in \cite{MT}. It is defined
for $f\in C^\infty(\b R^N)$ by
\[ M_f(x,t) := \frac{1}{d_k} \int_{S^{N-1}} f(x*_kty) w_k(y) d\sigma(y)\quad (x\in \b R^N, \, t\geq 0)\]
where $S^{N-1} = \{ x\in \b R^N: |x|=1\}$ is the unit sphere in $\b R^N$ with normalized Lebesgue surface measure $d\sigma$ and
\[ d_k = \int_{S^{N-1}} w_k(\xi)d\sigma(\xi).\]
It is easily seen from the continuity properties of the generalized translation that $M_f\in C^\infty(\b R^N\times[0,\infty)).$
The following is the first main result of \cite{R6}.

\begin{theorem}\label{T:spherical} \begin{enumerate}\itemsep=-1pt
\item[\rm{(1)}] The spherical mean operator $\, f\mapsto M_f$ is positive on $C^\infty(\b R^N)$, that is
$\,  f\geq 0$ on $\b R^N$ implies   $\,M_f \geq 0 \, $on $\b R^N\times [0,\infty).$
\item[\rm{(2)}] For each $x\in \b R^N$ and $t\geq 0,$ there exists a unique compactly supported probability measure
$\sigma^k_{x,t} \in M^1(\b R^N)$ such that for all $f\in C^\infty(\b R^N),$
\[ M_f(x,t) =
\int_{\b R^N} f(\xi) d\sigma^k_{x,t}(\xi).\]
The measure $\sigma^k_{x,t}$ satisfies
\[\text{supp}\,\sigma^k_{x,t}
\subseteq \bigcup_{w\in W}\{\xi\in \b R^N: |\xi-wx|\leq t\},\]
and the mapping $(x,t)\mapsto \sigma^k_{x,t}\,$ is weakly continuous. Moreover,\\
$\sigma^k_{wx,t}(A) = \sigma_{x,t}^k(w^{-1}(A))$ and $\sigma_{rx,rt}^k(A) = \sigma_{x,t}^k
(r^{-1}A)$ for all $w\in W, r>0,$ and all Borel sets $A\subset \b R^N$.
\end{enumerate}
\end{theorem}

The proof of this theorem involves, among other ingredients,  the theory of  $k$-spherical harmonics which was initiated in \cite{Du0}. A good introduction to this subject can be found in the monograph \cite{DX}. The space of $k$-spherical harmonics of degree $n\geq 0$ is defined by
\[ \mathcal H_n^k = \text{ker}\,\Delta_k \cap \mathcal P_n.\]
As in the theory of ordinary spherical harmonics, the  space $L^2(S^{N-1}, w_kd\sigma)$ decomposes as an orthogonal Hilbert space sum
\[ L^2(S^{N-1}, w_k d\sigma) = \bigoplus_{n=0}^\infty \mathcal H_n^k.\]
An important ingredient in the proof of Theorem \ref{T:spherical} is the following special case of
the Funk-Hecke formula for $k$-spherical harmonics which goes back to \cite{X2}.

\begin{proposition} Let $N\geq 2$ and put $\lambda:= \gamma +N/2-1.$  Then for all $Y\in \mathcal H_n^k$ and $x\in \b R^N$,
\begin{equation}\label{Funk-Hecke}\frac{1}{d_k} \int_{S^{N-1}} E_k(ix,\xi)Y(\xi)w_k(\xi)d\sigma(\xi) =
\frac{\Gamma(\lambda +1)}{2^n\Gamma(n+\lambda +1)} j_{n+\lambda}(|x|) Y(ix).
\end{equation}
In particular,
\begin{equation}\label{Dunklint}\frac{1}{d_k}\int_{S^{N-1}} E_k(ix,\xi)w_k(\xi)d\sigma(\xi) = j_{\lambda}(|x|).\end{equation}

\end{proposition}

\begin{proof}[Proof of Theorem \ref{T:spherical} (Sketch)]
The proof of part (1) is achieved by a reduction to initial data of the form $f(x) = \Gamma_k(s,x,y), $ where $\Gamma_k$ is the Dunkl-type heat kernel. Indeed, it is not hard to see by Dunkl transform methods
that for each $f\in \mathcal S(\b R^N)$ and $(x,t)\in \b R^n\times [0,\infty),$
\[ M_f(x,t) = \lim_{s\downarrow 0}\int_{\b R^N} M_{\Gamma_k(s,\,.\,,z)} (x,t) f(z)w_k(z)dz.\]
For the proof that
$\, M_{\Gamma_k(s,\,.\,,z)} (x,t) >0$ for all $ s, t\geq 0$ and $x,z\in \b R^N,$
the Funk-Hecke formula \eqref{Funk-Hecke}, the
positive integral representation of the intertwining operator and the positive product formula \eqref{besselprod_onevar} for the one-variable Bessel functions $j_\lambda$ are involved.

(2) The existence of representing measures $\sigma_{x,t}^k$ follows by standard methods from (1).
For the statement on their support, one observes that $u(x,t) = M_f(x,t)$ solves the initial value problem
\[ \begin{cases}\bigl(\Delta_k - A_\lambda^t\bigr) u = 0  \quad \text{ in }\, \b R^N\times (0,\infty)\\
u(x,0) = f(x), \,\, u_t(x,0) = 0
\end{cases}\]
with the Bessel operator
\[ A_\lambda^t = \partial_t^2 + \frac{2\lambda +1}{t}\,\partial_t\,.\]
The operator $\Delta_k - A_\lambda^t$ is of Darboux-type and generalizes the classical wave operator. A study of the domain of dependence of its solutions
leads to the claimed statement about the support of the $\sigma_{x,t}^k$.
Hereby the influence of the reflection parts needs some care.
\end{proof}

For the special choice  $f(x) = E_k(x,iz)$ with fixed $z\in \b R^N$ we obtain from
formulas \eqref{Dunklint}  and
\eqref{kernelprod}  the simple form
\[ M_f(x,t) = E_k(ix,z) j_\lambda(t|z|).\]
Therefore the Dunkl kernel satisfies the following
radial product formula:

\begin{corollary}\label{C:radialprod}
\[ E_k(x,iz)\,j_\lambda(t|z|) \,=\, \int_{\b R^N} E_k(\xi,iz)d\sigma_{x,t}^k(\xi) \quad (x,z\in \b R^N, t\geq 0).\]
\end{corollary}

A function or measure  on $\b R^N$ is called \textit{radial}, if it is
 invariant under the action of
 the orthogonal group $O(N).$ We shall denote subsets consisting of radial functions or measures by the subscript $rad$.
As already announced, the Dunkl-type generalized translation is positive when
restricted to radial functions.
In fact, it can be written in a fairly simple
explicit form by means of the intertwining operator. This is proven in
\cite{R6} (proof of Theorem 5.1 there)
and can also be found in \cite{DX} via a slightly different approach:

\begin{theorem}\label{T:postrans}
If $f\in C^\infty_{rad}(\b R^N)$ with $f(x) = F(|x|)\geq 0\,$ for all $x\in \b R^N$,  then also $\tau_yf(x)\geq 0$ for all $x,y\in \b R^N.$
The translate $\tau_yf$ is given explicitly by
\[ \tau_yf(x) = V_k\bigl(F(\sqrt{|x|^2 + |y|^2 + 2\langle x,\,.\,\rangle})\bigr)(y).\]
\end{theorem}

We mention that in \cite{DX}, this formula is established on the space
$A_k(\b R^N) = \{ f\in L^1(\b R^N,w_k): \widehat f^{\,k} \in L^1(\b R^N, w_k)\},$ which is a subspace of $L^2(\b R^N, w_k).$

\noindent
Let us now turn to the translation of radial measures.

\begin{definition} For a subset $\mathcal M \subset M_b(\b R^N)$ we call
\[ \mathcal M_{prad} :=\{ \mu\in \mathcal M: \frac{1}{w_k}\,\mu \,\text{ is radial}\},\]
the space of \textit{pseudo-radial} measures from $\mathcal M$.
\end{definition}

Important examples of bounded pseudo-radial measures, i.e. measures belonging to $M_{b, prad}(\b R^N),$
are those of the form $d\mu(x) = f(x)w_k(x)dx$ with
$f\in L^1_{rad}(\b R^N,w_k).$

\begin{proposition}\label{P:pseudorad}
\begin{enumerate}\itemsep=-1pt
\item For $\mu\in M_{prad}^1(\b R^N),$ the assignment
\[ \delta_x*_k\mu(\varphi):= \int_{\b R^N} \varphi(x*_k y) d\mu(y), \quad \varphi\in \mathcal S(\b R^N)\]
defines a probability measure $\delta_x*_k\mu\in M^1(\b R^N).$
Moreover, $*_k$ extends to a probability-preserving and bilinear convolution \begin{align*} &*_k:M_b(\b R^N)\times M_{b,prad}(\b R^N) \to M_b(\b R^N), \\
&\mu*_k\nu(\varphi) :=\int_{\b R^N}\int_{\b R^N} \varphi(x*_k y) d\mu(x)d\nu(y).\end{align*}
\item If $\mu,\,\nu\in M_{b,prad}(\b R^N)$ are pseudo-radial, then also $\mu*_k\nu\in M_{b,prad}(\b R^N),$ and the convolution $*_k$ on $M_{b,prad}(\b R^N)$ is commutative.
\end{enumerate}
\end{proposition}

For the proof of these results, see \cite{R6} (the extension away from probability measures to bounded measures is straightforward).

In \cite{TX}, the authors extend the generalized translation $\tau_y$
to a positivity-preserving and bounded linear operator
\[\tau_y: L_{rad}^p(\b R^N, w_k) \to L^p(\b R^N, w_k)\]
for $1\leq p \leq 2.$ This allows  to define the convolution
of functions from various  $L^p$-spaces, such as $f*_kg$ for bounded
$g\in L^1_{rad}(\b R^N, w_k)$ and $f\in L^p(\b R^N,w_k)$ with $1\leq p\leq 2.$ See \cite{TX} for details and applications like the maximal operator
in the Dunkl setting.

\vfill\newpage

\section[Gelfand pairs, Bessel functions, and hypergroups on matrix cones]{Gelfand pairs, Bessel functions,  and hypergroups on matrix cones}
\label{S1}

\subsection{Motivation}\label{motivation-hypergroups}

A basic motivation for the study of Dunkl operators is their close relation with
 the theory of Riemannian symmetric spaces and their spherical functions. Indeed, the algebra of invariant differential operators of a Riemannian symmetric space $G/K$ can be expressed in terms
of the Weyl-group invariant parts of Dunkl or Cherednik operator algebras. Hereby symmetric spaces of the non-compact and compact type  lead to Dunkl-Cherednik operators, while rational Dunkl operators correspond to symmetric spaces of Euclidean type. The latter are associated with  Cartan motion groups of non-compact symmetric spaces $G/K$. They are of the form $(K\ltimes V)/K$, where $V$ is a Euclidean space of finite dimension
which can be identified with the tangent space of $G/K$ in the trivial coset $eK$, and $K$ is a compact Lie group acting  on $V$ by
orthogonal transformations. If $G/K$ is a Riemannian symmetric space of arbitrary type, then $(G,K)$ is always a \textit{Gelfand pair}, which means that the subalgebra of $K$-biinvariant functions within the convolution algebra $L^1(G)$ is commutative.

\smallskip

Symmetric spaces provide  many infinite discrete series of Gelfand pairs $(G,K)$
whose double coset spaces $\, G//K:=K\setminus G/K$ can be identified with a common locally compact space $X$. For symmetric spaces of rank one (where $X$ is one-dimensional) it has been well-known since a long time that  their spherical functions
can be embedded into certain series of classical hypergeometric functions.
 These are well defined for continuous parameters, but only for certain
 discrete parameter values
 they have an interpretation as spherical functions. Nevertheless, many  nice properties known for spherical functions, like
integral representations and  positive product formulas, extend to general parameter values, and one still
obtains  associated commutative convolution algebras of bounded Borel
measures on $X$ which have properties similar to group convolutions
 and form so-called commutative
hypergroups $(X,*)$. In this way,   commutative hypergroups
extend the theory of spherical functions in some
respect, and the theory of commutative  hypergroups may help to
understand results for families of special functions with a continuous
parameter range, which admit
positive product formulas, in a systematic way.

\smallskip
For symmetric spaces of higher rank,
the theories of Dunkl- and Cherednik operators provide a framework of hypergeometric functions in
several variables which again includes the spherical functions as
particular cases. In the Euclidean case, these are just Bessel functions of
Dunkl type.
They can be identified with spherical functions for discrete parameter values
which are determined by the Cartan decomposition of the underlying Lie algebra.

\smallskip
One may conjecture that for arbitrary non-negative multiplicities always
positive integral representations
and product formulas leading to hypergroup convolutions exist. These questions
are in general much harder than in rank one and remain unsolved to a major extent.
There are, however, some classes of known examples  which allow the extension
of product formulas
 and commutative hypergroup algebras  beyond the geometric cases. They are
 related with Grassmann
 manifolds $U(p,q;\b F)/U_p(\b F)\times U_q(\b F)$
over one of the skew-fields $\b F= \b R, \, \b C$ or $\b H$ and will be
presented in Section \ref{section-matrix-cones}.
For these classes,  hypergroup analysis supplements the Weyl-group invariant Dunkl theory in a  satisfactory way.

The basic guiding example, which corresponds to the rank-one case of the above series of Grassmann manifolds, are  the so-called Bessel-Kingman hypergroups on $X=[0,\infty)$. Their structure is precisely that of rank-one Dunkl theory
in the Weyl-group invariant case.
 Here one takes the groups $K=O(d)$
acting on $\b R^d$, and  the spherical functions $\phi_\lambda$ can be expressed in terms of
the spherical Bessel functions $j_\alpha(x) = \,_0F_1(\alpha+1;-x^2/4)$, c.f.
\eqref{(1.8)}, with index $\alpha = d/2 -1,$ via $\phi_\lambda(x)=j_\alpha(\lambda x)$, $\lambda\in\b C$.
 The spherical product formula
reads
 \begin{equation}\label{besselprod} \phi_\lambda(x)\phi_\lambda(y)=   \frac{\Gamma(\alpha+1)}{\sqrt\pi
   \Gamma(\alpha+1/2)}
\int_{-1}^1 \phi_\lambda\bigl(\sqrt{r^2 + s^2 + 2rst}\,\bigr)\,
(1-t^2)^{\alpha-1/2}\, dt
\end{equation}
for $\alpha=d/2-1 > -1/2.$ The case $\alpha = -1/2$ is degenerate with
$$\phi_\lambda(x)\phi_\lambda(y)=  \frac{1}{2}(\phi_\lambda(|x-y|) +\phi_\lambda({x+y})).$$
We shall explain this in Examples \ref{E:Bessel_Kingman} and \ref{E:Bessel_Kingman2}.
In fact, product formula \eqref{besselprod}  extends by analytic continuation to Bessel functions of arbitrary
real index $\alpha > -1/2.$ It is known as  Gegenbauer's product formula for Bessel functions, and
  \eqref{besselprod_onevar} is just  an equivalent formulation.
This product formula  leads to a continuous family of commutative hypergroup
  structures on $[0,\infty)$ with parameter $\alpha\geq -1/2$,  called  \textit{Bessel-Kingman
  hypergroups}. The associated hypergroup Fourier transforms are
  Hankel transforms which describe for half-integers $\alpha=d/2-1$ just
  the radial Fourier transforms on  $\b R^d$. One  may thus for example
  investigate radial random walks on $\b R^d$ by
  considering their radial parts which  form random walks on
  Bessel-Kingman hypergroups, see ~\cite{Ki}.

This classical one-dimensional example was recently extended to Bessel-type hypergroups on the
cones $\Pi_q=\Pi_q(\b F)$ of positive semidefinite matrices over
 $\mathbb F= \mathbb R, \mathbb C$ or $\mathbb H$ in \cite{R8} as follows: For $p\geq q$, consider
the matrix space $M_{p,q} = M_{p,q}(\mathbb F)$ of $p\times q$-matrices over $\mathbb F$ as a real vector space
with inner product $\langle x,y\rangle = \text{Re tr}(x^*y)$, where $x^* = \overline x^t$.
The unitary group $U_p = U_p(\mathbb F)$ acts by multiplication from
the left, and these transformations are orthogonal with respect to the given inner product. The mapping
$\, U_p.x \,\mapsto \, \sqrt{x^*x}$ then establishes a homeomorphism between the
space of
 $U_p$-orbits in $M_{p,q}$ and the cone
$\Pi_q$. The
spherical functions of the Gelfand pair $(U_p\ltimes M_{p,q}, U_p)$
are given by Bessel functions of matrix argument
(see \cite{FK}, \cite{Her} as standard references), and the Fourier  transform of $U_p$-invariant functions on $M_{p,q}$
can  be expressed in terms of
a  Hankel transform on the cone $\Pi_q$. These
discrete series of examples can be embedded  into three continuous series
 of matrix Bessel hypergroups on $\Pi_q$. This completes the $L^2$-theory
for Hankel transforms on $\Pi_q$ as developed by \cite{Her} and \cite{FT}
in the general framework of symmetric cones. A connection between this structures and Dunkl theory is then established as follows:
The group $U_q$ acts as a compact group of hypergroup
automorphisms on $\Pi_q$ by conjugation, and the associated orbit space is
canonically parametrized by
\[ \Xi_q  = \{ \xi\in \mathbb R^q: \xi_1\geq \ldots \geq \xi_q\geq 0\},\]
the set of possible spectra of positive semidefinite matrices. This is  a Weyl chamber
of type $B_q$. We shall thus obtain three continuous series of commutative hypergroups
on $\Xi_q$ whose
 hypergroup characters are
Dunkl-type Bessel functions  of type  $B_q$. They interpolate the
convolutions of the Gelfand pairs $((U_p\times U_q)\ltimes M_{p,q},U_p\times
U_q$), and
one obtains explicit positive product formulas and hence commutative
hypergroup structures on $\Xi_q$
for these parameters.

In order to explain these connections and their stochastic implications
more precisely, we next give a short survey about Gelfand pairs of Euclidean
type and commutative hypergroups.

\subsection{Gelfand pairs,   Euclidean orbit spaces, and hypergroups}\label{Section-Gelfand-pairs}

In this section we give a quick introduction to certain classes of Gelfand
pairs and commutative hypergroups.
We do not include proofs  and refer for details on Gelfand pairs to the survey of
\cite{F}, and for hypergroups
 mainly to
\cite{J} and  \cite{BH}. We first recapitulate some notions on Gelfand pairs.
Let $G$ be a locally compact group and $M_b(G)$
 the Banach space  of all regular bounded Borel measures on $G$. $M_b(G)$ becomes a
Banach-$*$-algebra with convolution $$\,\mu*\nu(f):=\int_G f(xy)\> d\mu(x)\> d\nu(y)$$ for $f\in C_b(G)$ and the involution
 $\mu^*(A):=\overline{\mu(A^{-1})}$ ($A\subset G$ a Borel set). The identity
 is $\delta_e$, the point measure in the
 identity  $e$ of $G$.
Recall also that there is a (up to normalization) unique left Haar measure
 $\omega_G\in M^+(G)$, that is a measure which is invariant under left translations
$f\mapsto \tau_x f, \,\, x\in G.$

Now let $K$  be a compact subgroup of $G$. The normalized Haar measure
$\omega_K\in M^1(K)$ may be regarded as  measure on $G$ and
satisfies $\omega_K^*=\omega_K*\omega_K=\omega_K$. It is easily checked that the space of all $K$-biinvariant bounded measures
\begin{align}
M_b(G\|K):&=\{\mu\in M_b(G):\> \delta_x*\mu*\delta_y=\mu\>\>\text{for all}\>\>
x,y\in K\}
\notag\\
&=\{\omega_K*\mu*\omega_K:\> \mu\in M_b(G)\}
\end{align}
 is a Banach-$*$-subalgebra of  $M_b(G)$.

\begin{definition}\label{D:Gelfandpaar}
The pair $(G,K)$ is be called a \textit{Gelfand pair}, if
$M_b(G\|K)$ is commutative.
\end{definition}

We mention that for a Gelfand pair $(G,K),$  the group  $G$ is
automatically unimodular, i.e.  $\omega_G$ is also a right Haar measure.
Further, the space $L^1(G, \omega_G)$ is naturally identified with a subspace of $M_b(G)$ and becomes a closed subalgebra of $(M_b(G),*)$ with  the convolution
$$f*g(x):=\int_G f(xy) g(y^{-1}) \> d\omega_G(y)$$ and  involution
$f^*(x):=\overline{f(x^{-1})}.$ If $(G,K)$ is a Gelfand pair, then
the Banach-$*$-algebra
$$L^1(G\|K):=\{f\in L^1(G, \omega_G): \> f(kxh)=f(x), \> x\in G,\> k,h\in K\}$$
of all biinvariant $L^1$-functions on $G$ is commutative.

\smallskip
\noindent
There are several  equivalent descriptions (e.g., using representation theory) of Gelfand pairs among which we quote the following useful criterion,
see \cite{F}.

\begin{lemma}\label{criterion-Gelfand}
Let $K$ be a compact subgroup of a locally compact group $G$. Assume there
exists a continuous involutive automorphism $\theta$ on $G$ satisfying $x^{-1}\in K\theta(x)K$
for all $x\in G$. Then $(G,K)$ is a Gelfand pair.
\end{lemma}

We next introduce  spherical functions and the spherical Fourier transform.

\begin{definition}\label{def-spherical-function}
Let $(G,K)$ be a Gelfand pair. Then a function $\phi\in C(G)$, $\phi\not\equiv
0$,  is called a
\textit{spherical function} of $(G,K)$  if $\phi$ is $K$-biinvariant, i.e., $\phi(kxh)=\phi(x)$ for  $k,h\in
K$ and  $x\in G$, and if $\phi$ satisfies the product formula
$$ \int_K \phi(xky)\> d\omega_K(k)= \phi(x)\phi(y).$$
Spherical functions obviously satisfy $\phi(e)=1$.
Denote the set of all  spherical functions of $(G,K)$ by $\Sigma$.
\end{definition}

The set $\Sigma_b:=\Sigma\cap C_b(G)$ can be identified with the spectrum of the Banach-$*$-algebra
$L^1(G\|K)$ via $\phi\mapsto L_\phi\in \Delta(L^1(G\|K))$ with $L_\phi(f)=\int_G
\phi f\> d\omega_G$. It becomes a locally compact space with the topology of locally-uniform
convergence, which corresponds to the Gelfand topology
on  the spectrum.

The spherical Fourier transforms of functions $f\in  L^1(G\|K)$ and measures
$\mu\in M_b(G\|K)$ are defined on $\Sigma $ according to
$$\hat f(\phi)=\int_G f(x)\phi(x^{-1}) \> d\omega_G(x)
\quad\text{and} \quad
\hat \mu(\phi)=\int_G \phi(x^{-1}) \> d\mu(x). $$
They satisfy $\hat f\in C_0(\Sigma)$ (``Riemann-Lebesgue-Lemma'') and $\hat \mu\in C_b(\Sigma)$.

\begin{definition}\label{pos_def}
A function $\phi\in C_b(G)$  is called \textit{positive
definite} if for all $n\in\b N$, $x_1,\ldots,x_n\in G$ and
$c_1,\ldots,c_n\in \b C$,
\[\sum_{i,j=1,\ldots, n} c_i\overline{c_j} \cdot
\phi(x_ix_j^{-1}) \ge0.\]
We denote by  $\Omega\subseteq \Sigma_b$  the closed subspace of all positive definite spherical functions.
\end{definition}

\noindent
We now list some major facts about the spherical Fourier transform:

\begin{theorem}\label{properties-spherical-Fourier-transform}
\begin{enumerate}\itemsep=-1pt
\item[\rm{(1)}] The spherical Fourier transform is injective.
\item[\rm{(2)}] Bochner Theorem: $\Omega$ is the set of extreme
  points in the set of positive
definite, biinvariant functions with  $\phi(e)=1$, and for each biinvariant
  positive definite function $f$ on $G$ there is a unique measure $\mu\in
  M_b^+(\Omega)$ with
$f(x)= \check{\mu}(x):= \int_\Omega \phi(x)\> d\mu(\phi).$
\item[\rm{(3)}] Plancherel's Theorem: There is a unique measure
  $\pi\in M^+(\Omega)$ such that the spherical Fourier transform of
  biinvariant $L^1$-functions can be extended to an isometric isomorphism from
  the space $L^2(G\|K)$ of all biinvariant $L^2$-functions onto $L^2(\Omega,\pi)$.
\item[\rm{(4)}] Inversion formula: If $f\in L^1(G\|K)$ with $\hat f\in
  L^1(\Omega,\pi)$, then $f= (\hat f)^\vee$.
\item[\rm{(5)}]  Inverse Riemann-Lebesgue lemma: If $f\in  L^1(\Omega,\pi)$,
  then $\check{ f}\in C_0(G)$.
\end{enumerate}
\end{theorem}

We note that $supp\>\pi \subset \Omega\subset  \Sigma_b\subset\Sigma$, and that equality may
fail for certain Gelfand pairs.
\medskip

 The theory of Gelfand pairs and the spherical Fourier transform can be considered as the
 origin of Fourier analysis on commutative hypergroups:

\begin{definition}\label{def-hypergroup} A \textit{hypergroup} $(X,*)$
is a locally compact Hausdorff space $X$ with a bilinear associative  convolution $*$  on $M_b(X)$ with the following properties: \parskip=-1pt
\begin{enumerate}\itemsep=-1pt
\item[\rm{(1)}] The map $(\mu,\nu) \mapsto \mu*\nu$ is weakly continuous.
\item[\rm{(2)}] For all $x,y\in X$, the product $\delta_x*\delta_y$ of point measures
is a compactly supported probability measure on $X$.
\item[\rm{(3)}] The mapping $(x,y)\mapsto \text{supp}(\delta_x *\delta_y) $ from $X\times X$ into the space of nonempty compact subsets of $X$ is continuous with respect to the Michael topology, see \cite{J}.
\item[\rm{(4)}] There is a neutral element $e\in X$, satisfying $\delta_e*\delta_x = \delta_x * \delta_e
= \delta_x$ for all $x\in X$.
\item[\rm{(5)}] There is a continuous involutive automorphism $x\mapsto\overline x$ on $X$ such that $\delta_{\overline x} * \delta_{\overline y} =(\delta_y*\delta_x)^-$ and $\,x=\overline y \Longleftrightarrow e\in \text{supp}(\delta_x*\delta_y),$
where for $\mu\in M_b(X)$, the measure  $\mu^-$ is given by
$\mu^-(A)=\mu(\overline A)$ for Borel sets $A\subseteq X$.
\end{enumerate}

 A hypergroup $(X,*)$ is called commutative if $*$ is commutative.
  In this case,
$(M_b(X),*,.^-)$
is a commutative Banach-$*$-algebra with identity $\delta_e$.
\end{definition}

Note that by weakly continuous, bilinear extension,
 the convolution of a hypergroup is uniquely determined as soon as it is given for point measures.

\begin{examples}\label{ex_1}
\begin{enumerate}
\item[\rm{(1)}] If $G$ is a  locally compact group, then $(G,*)$ is a
  hypergroup with the  group convolution $*$.
\item[\rm{(2)}]  Let $K$ be a compact subgroup of a locally compact group $G$.
Consider the Banach-$*$-algebra
$M_b(G\|K)$  with identity $\omega_K\in M^1(G)$, and the
  double coset space $$G//K:=\{KxK:\> x\in G\}$$  which is   a locally compact
Hausdorff space w.r.t. the quotient topology. The canonical projection
$p:G\to G//K$ induces a  probability preserving, isometric isomorphism
$\,p: M_b(G\|K)\to M_b(G//K)$ of Banach spaces by taking
images of measures. The transfer of
the convolution on  $M_b(G\|K)$ to $M_b(G//K)$ via $p$ leads to a
 hypergroup
structure $(G//K, *)$ with identity $K\in G//K$ and involution
$(KxK)^-:=Kx^{-1}K$, and
$p$  becomes a  probability preserving,
 isometric isomorphism of Banach-$*$-algebras. The convolution of point
 measures on $G//K$ is given explicitly by
$$\delta_{KxK}*\delta_{KyK}=\int_K \delta_{KxkyK}\> d\omega_K(k) \quad (x,y\in G).$$
This \textit{double coset hypergroup} is clearly commutative if and only if $(G,K)$ is a Gelfand
pair.
\end{enumerate}
\end{examples}

\begin{example}\label{orbit-hypergroup-example}
Let $(V,+)$ be a locally compact abelian group on which a compact group $K$ of
automorphisms acts continuously via $(k,v)\mapsto k.v$. Consider the semidirect product $G:= K\ltimes V$, and
regard $K$ as a compact subgroup of  $G$ in the
obvious way. Then $(G,K)$ is a Gelfand pair (apply criterion
\ref{criterion-Gelfand} to the automorphism $\theta(k,v):=(k,-v)$). Moreover,
  the  locally compact space
\[V^K:= \{ K.v: \,v\in V\}\]
 of all $K$-orbits of $V$ may be identified
  with the  double coset space
$G//K$  via
$ K.v\in V^K\simeq  K (e,v)K\,\in G//K$. Denoting
 the image measure of $\mu\in M_b(V)$ under $k\in K$ by $k(\mu)$, we see that
  the Banach-$*$-algebra $M_b(G\|K)$ of all biinvariant measures on $G$ may be
 identified with the Banach-$*$-algebra
$$M_b^K(V):=\{\mu\in M_b(V):\> k(\mu)=\mu\>\>\text{for all}\>\>k\in K\}$$
of $K$-invariant measures on $V$. By lifting
this structure to $M_b(V^K)$,
 one obtains  a so-called commutative \textit{orbit hypergroup} structure on
$V^K$ with the following explicit convolution of point measures:
$$\delta_{K.v}*\delta_{K.w}=\int_K \delta_{K.(v+k.w)}\> d\omega_K(k) \quad
(v,w \in V).$$
The neutral element of this hypergroup is $K.0 = \{0\}$, the involution   is given by $\overline{K.v} = K.(-v)$.
\end {example}

The most frequent examples of this type arise from Gelfand pairs of Euclidean type:
Here $V$ is  a finite-dimensional Euclidean vector space  and $K$ is a compact subgroup of the orthogonal group $O(V).$

\begin{example}[\textbf{Bessel-Kingman hypergroups}] \label{E:Bessel_Kingman} Let $V=\b R^d$ and $K=O(d)$, acting by matrix
 multiplication from the left.
Then the orbit space consists of the spheres $\{x\in \b R^d: |x|=r\}$ with $r\geq 0$ and can be
topologically identified with $[0,\infty).$
The convolution of the orbit
hypergroup $(\b R^d)^{O(d)}\cong [0,\infty)$ becomes
\begin{equation}\label{Besselformel_1} ( \delta_r * \delta_s)(f) = \int_{O(d)} f(|re_1 + k.se_1|) dk \,=\, \int_{S^{d-1}}f(|re_1 + s\sigma|)d\sigma\end{equation}
where $e_1 = (1,0,\ldots,0),$ $dk$ denotes the normalized Haar measure of $O(d)$, and $d\sigma$
is  the normalized surface measure on the unit sphere
$S^{d-1}.$  For $d>1$ we use  spherical polar coordinates and the cosine theorem, thus arriving at
\[  (\delta_r * \delta_s)(f) = \, \frac{\Gamma(d/2)}{\sqrt{\pi}\,\Gamma((d-1)/2) }\int_{-1}^1
f\bigl(\sqrt{r^2 + s^2 + 2rst }\bigr)(1-t^2)^{(d-3)/2}dt.\]
For  $d=1$ the formula degenerates to
\[  ( \delta_r * \delta_s)(f) = \frac{1}{2}\bigl(f(r+s) + f(|r-s|)\bigr).\]
The neutral element of each of these hypergroups is $0$, and the involution is the identity.
\end{example}

\noindent
We next collect some general notations and facts on commutative hypergroups $(X,*).$
\begin{enumerate}
\item[\rm{(1)}] For a bounded Borel function $f:X\to\b C$ and $x\in X$ define
  the translate
\[ f_x(y):=f(x*y):= \int_X f\> d(\delta_x*\delta_y).\]
The spaces $C_c(X)$, $C_0(X)$ and $C_b(X)$ are preserved by such translations.
\item[\rm{(2)}] By a famous result of
R. Spector, there exists a (up to normalization)  unique
Haar measure $\omega\in M^+(X)$ which is characterized by
$\, \omega(f)=\omega(f_x)$ for all  $\,f\in C_c(X)$ and
$x\in X.$
\item[\rm{(3)}] The involution and convolution of measurable functions  $f, g$
  on $X$ are given by $f^*(x):=\overline{f(\bar x)}$ and, in case of convergence,
  $$f*g(x)=\int_X f(y)\> g(x*\bar y)\> d\omega(y).$$
In particular, for $f, g\in L^1(X,\omega)$ one has  $f^*, f*g\in
L^1(X,\omega)$, and $L^1(X,\omega)$ becomes a commutative Banach-$*$-algebra.
There are also convolutions of functions from other $L^p$-spaces, just as with groups. For example,
if $f\in L^1(X,\omega)$ and $g\in L^p(X,\omega)$, then $f*g\in L^p(X,\omega)$ with
$\, \|f*g\|_p \leq \|f\|_1\cdot \|g\|_p.$
\item[\rm{(4)}]  Similar to  locally compact abelian groups,
one defines the spaces
\begin{enumerate}\itemsep=-2pt
\item[\rm{(a)}] $\chi(X):=\{\alpha\in C(X):\> \alpha\not\equiv 0,\>\>  \alpha(x*y)=
 \alpha(x)\alpha(y)\>\>\text{for all}\>\> x,y\in X\}$;
\item[\rm{(b)}]  $\chi_b(X):=\chi(X)\cap C_b(B)$;
\item[\rm{(c)}]$\widehat X:= \{\alpha\in\chi_b(X):\> \alpha(\bar x)=
\overline{\alpha(x)}\>\>\text{for all}\>\> x\in X\}$.
\end{enumerate}
Elements of $\widehat X$ are called
 characters, and $\chi_b(X)$ and $\widehat X$ are locally compact
 Hausdorff spaces w.r.t.~the topology of
compact-uniform  convergence.
\item[\rm{(5)}] The Fourier(-Stieltjes) transform of a function  $ f\in L^1(X, \omega)$
 and a measure $\mu\in M_b(X)$ are defined by
 $$\widehat f(\alpha):= \int_X f(x)\overline{\alpha(x)}\>
d\omega(x)
\quad\text{and}\quad
 \widehat \mu(\alpha) := \int_X \overline{\alpha(x)} d\mu(x) \quad (\alpha\in \widehat X).$$
\item[\rm{(6)}] The hypergroup Fourier transform satisfies $(\mu*\nu)^\wedge = \hat\mu\cdot
\hat \nu$ and $(\mu^*)^\wedge = \overline{\hat\mu}$ with analogous formulas for
$L^1$-functions.
Moreover, $\hat \mu\in C_b(\hat X)$ and $\hat f\in C_0(\hat X)$ for $ f\in L^1(X, \omega)$
(``Riemann-Lebesgue-Lemma'').
\item[\rm{(7)}]  A function $\phi\in C_b(X)$  is called positive
definite if for all $n\in\b N$, $x_1,\ldots,x_n\in X$ and
$c_1,\ldots,c_n\in \b C$, $\sum_{i,j=1,\ldots, n} c_i\overline{c_j} \cdot
\phi(x_i*\bar x_j) \ge0$.
\end{enumerate}

The following results are  analog to those in Theorem
\ref{properties-spherical-Fourier-transform}
for the spherical Fourier transform. For proofs we refer to \cite{J},
\cite{BH}, and references therein.

\begin{theorem}\label{properties-hypergroup-Fourier-transform}
\begin{enumerate}\itemsep=-1pt
\item[\rm{(1)}] The hypergroup  Fourier transform is injective.
\item[\rm{(2)}] Bochner Theorem: $\hat X$ is the set of extreme
  points in the set of bounded positive
definite functions on $X$ with normalization $\phi(e)=1$, and for each bounded
  positive definite function $f$ on $X$ there is a unique  $\mu\in
  M_b^+(\hat X)$ such that
$f(x)= \check{\mu}(x):= \int_{\hat X} \phi(x)\> d\mu(\phi).$
\item[\rm{(3)}] Plancherel Theorem: There is a unique measure
  $\pi\in M^+(\hat X)$ such that the hypergroup Fourier transform extends
  to an isometric isomorphism from
  the space $L^2(X,\omega)$ onto $L^2(\hat X,\pi)$.
\item[\rm{(4)}] Inversion formula: If $f\in L^1(X)$ with $\hat f\in
  L^1(\hat X,\pi)$, then $f= (\hat f)^\vee$.
\item[\rm{(5)}]  Inverse Riemann-Lebesgue lemma: If $f\in  L^1(\Omega,\pi)$,
  then $\check{ f}\in C_0(G)$, and
the space
$\{\check f:\> f\in  L^1(\hat X,\pi)\}$ is $\|.\|_\infty$-dense in $C_0(X)$.
\end{enumerate}
\end{theorem}

Similar to Gelfand pairs we have   $\,supp\>\pi\subset \hat X\subset
\chi_b(X)\subset \chi(X)$  where equality may always fail.

\begin{examples}\label{examples-haar-hyper}
\begin{enumerate}\itemsep=-1pt
\item[{(1)}]
Let $(G,K)$ be a Gelfand pair and $(X=G//K, *)$ the associated
double coset hypergroup. Then the canonical
projection of the Haar measure $\omega_G$ of $G$  onto $X$ is a Haar measure of
$X$. Moreover, we may identify $\chi(X)$ and $\Sigma$ by identification of biinvariant functions on
$G$ with functions on $X$. When doing so,
the Plancherel measures of $(G,K)$ and $X$ are the same, and one has the inclusions
$$ supp\>\pi\subset \Omega \subset \hat X\subset
\chi_b(X)=\Sigma_b\subset \chi(X)=\Sigma.$$
\item[{(2)}] For the Bessel-Kingman hypergroup $X= (\b R^d)^{O(d)}$, the Haar measure (in standard
normalization) is the
image measure of the normalized Lebesgue measure $(2\pi)^{-d/2} dx$ on $\b R^d$ under the mapping $x\mapsto |x|$ and is given by
\[d\omega_\alpha(r) = \frac{2^{-d/2}}{\Gamma(d/2)}r^{d-1}dr, \quad\quad \alpha=d/2-1. \]
\end{enumerate}
\end{examples}

In practice, for a given Gelfand pair $(G,K)$ or commutative
hypergroup $(X,*)$ there  usually  exists a natural way of constructing examples
of spherical or multiplicative functions, respectively. However some additional particular
techniques then often have to be employed to prove that these examples really
form ALL of them. If there exists for example a sufficiently
large class of
underlying invariant  differential operators, then one has to  show that
spherical (or multiplicative) functions are  eigenfunctions of these
operators which then leads to a complete description of these objects.
This is for instance the
case for Gelfand pairs associated with Riemannian symmetric spaces (see
\cite{Hel2})
and also for
so-called one-dimensional Sturm-Liouville hypergroups (see Section 3.5.23 of \cite{BH}).
In some cases,  there exists  a further method to determine $\Omega$
and $\hat X$, which seems to
be not  well-known and which is as follows: Often it is
possible to write down the
Plancherel measure $\pi$ explicitly, i.e., $supp\>\pi$ is known.
On the other hand there exists a growth criterion  which ensures that in the Gelfand pair setting
$supp\>\pi=\Omega$  and for commutative hypergroups $supp\>\pi=\hat X$ holds.

\begin{definition}\label{subexponential}
A hypergroup $(X,*)$ with left Haar measure $\omega$ has subexponential growth, if for each compact subset
$K\subset X$ and its powers $K^n$  ($n\in\b N$), which are recursively defined by
$K^1=K$ and
$$K^{n+1}:= K*K^n :=\bigcup_{x\in K, \> y\in K^n} supp(\delta_x*\delta_y),$$
and for each $a>1$, the growth rate   $\omega(K^n)=o(a^n)$ holds.
\end{definition}

\begin{theorem}\label{theorem-subexponentionential}
\begin{enumerate}
\item[\rm{(1)}] Let $(X,*)$ be a commutative hypergroup with subexponential
  growth. Then
$supp\> \pi=\hat X=  \chi_b(X).$
\item[\rm{2)}] Let $(G,K)$ be a Gelfand pair where  $G$ has  subexponential
  growth. Then $supp\> \pi=\Omega = \Sigma_b$.
\end{enumerate}
\end{theorem}

\begin{proof} Part (1) is due to \cite{Vog} and \cite{V0};
see Theorem 2.5.12 of \cite{BH} as a standard reference.
In the case of a Gelfand pair $(G,H)$, we note that $G$ has subexponential
growth if and only if so has the double coset hypergroup $(G//K,*)$ such that
(2) is a consequence of (1).
\end{proof}

\begin{example}\label{ex-theorem-subexponentionential}
Consider the  Gelfand pair
$(K\ltimes V, K)$ for a  locally compact abelian group $V$ as above.
Then $V$ is polynomially growing  (see
\cite{Gr}), and thus $K\ltimes V$ as well. Theorem
\ref{theorem-subexponentionential}
therefore implies $supp\> \pi=\Omega = \Sigma_b$.
\end{example}

We now turn to Gelfand pairs of Euclidean type. We shall in particular determine the dual space of the Bessel Kingman orbit hypergroups.

Let $(V, \langle\,.\,,.\,\rangle)$ be  a finite-dimensional Euclidean vector
space and $K\subset O(V)$
 a compact subgroup. Then by Example \ref{orbit-hypergroup-example}, $(G=K\ltimes V,K)$ is a Gelfand pair,   $G/K$ may be identified with $V$, and $G//K$
with the orbit space $X:=V^K$ in the obvious way. We thus have canonical
projections
$$G \quad \longrightarrow \quad G/K\simeq V \quad \longrightarrow \quad  G//K\simeq
V^K=X.$$
A precise description of the spherical functions can be found in the recent paper of Wolf \cite{Wo}. We introduce the complexification $V_{\b
  C}$ of $V$
equipped with the extension of the scalar product as well as the
complexification $K_{\b C}:=K\cdot K_{\b C}^0\subset O(n,\b C)$ where  $K_{\b
  C}^0$ is the identity component of $K_{\b C}$ that corresponds to the complexified Lie
algebra  of the Lie algebra of the compact Lie group
$K$.

The relevant results about the spherical functions of $(G,K)$ or, in different
language, the characters of the double coset hypergroup $(X,*)$ are as
follows:

\begin{theorem}\label{properties-Euclidean-Fourier-transform}
\begin{enumerate}
\item[\rm{(1)}] For $\lambda\in V_{\b C}$, the means
$$\phi_\lambda(x):= \int_K e^{-i\langle k.x,\lambda\rangle } \> d\omega_K(k) =\int_K
e^{-i\langle x, k.\lambda\rangle}  d\omega_K(k) \quad (x\in V)$$
are continuous, $K$-invariant functions on $V$, and thus may be
regarded as biinvariant functions $\phi_\lambda\in C(G\|K)$ on $G$ as well as
functions $\phi_\lambda\in C(X)$ on the orbit space $X$. If doing so,
\begin{equation}\label{eq-euclidean-completeness}
\{\phi_\lambda:\> \lambda\in V_{\b C} \} =\Sigma =\chi(X).
\end{equation}
\item[\rm{(2)}] For $\lambda,\tilde \lambda\in V_{\b C}$,  the equality
  $\phi_\lambda=\phi_{\tilde \lambda}$ holds if and only if the orbit closures
  of $\lambda,\tilde\lambda$ under $K_{\b C}$ satisfy
 $cl \> K_{\b  C}(\lambda)\cap cl \> K_{\b  C}(\tilde
  \lambda)\ne\emptyset$.

Moreover,  for $\lambda,\tilde \lambda\in V$,
 $\phi_\lambda=\phi_{\tilde \lambda}$ holds if and only if
 $ K(\lambda)= K(\tilde \lambda)$.
\item[\rm{(3)}] $\{ \phi_\lambda:\> \lambda\in V\} = supp\>\pi= \Omega=\hat X=\Sigma_b=\chi_b(X)$.
\end{enumerate}
\end{theorem}

\begin{proof}
It can be easily checked by computation that the $\phi_\lambda$ are spherical functions; see e.g.~\cite{F} or \cite{Wo}.
For  equality in (\ref{eq-euclidean-completeness}) as well as for part (2) we
also refer to \cite{Wo}.
For a proof of (3) we also can refer to  \cite{Wo}, but we here prefer the
following argument: Note first that obviously $\{ \phi_\lambda:\> \lambda\in
V\} \subset\Omega$ holds, and that the quotient topology on $V^K$ agrees with
the topology of local-uniform convergence on $\Omega$ after
a suitable identification. Moreover, the $L^2$-isometry of the Fourier transform on
$V$ immediately implies that the projection $\pi\in M^+(V^K)$ of the suitably
normalized Lebesgue measure on $V$ is in fact the Plancherel measure on
$\Omega$. Theorem \ref{theorem-subexponentionential}
 and Example \ref{ex-theorem-subexponentionential} now yield the  equality.
\end{proof}

\begin{example}[\textbf{Bessel-Kingman hypergroups}] \label{E:Bessel_Kingman2}
 Let $V=\b R^d$ and $K=O(d)$ as in Example \ref{E:Bessel_Kingman}. Then we may
 realize the associated orbit hypergroup as $(X:=[0,\infty),*)$ with the
 convolution there. Moreover, we have the Haar
 measure $\omega_\alpha\in M^b([0,\infty))$ from Example \ref{examples-haar-hyper}(2).

By  taking spherical polar coordinates in the integrals of part
 (1) of the preceding theorem  and a well-known integral representation of the Bessel
 functions $j_\alpha$ defined in Eq.~(\ref{(1.8)}) (see e.g.~Eq.~(1.71.6) of
 \cite{Sz}),
we see that the spherical functions are given by
\begin{equation}\label{Besselformel_laplace}
\phi_\lambda(x)= \frac{\Gamma(\alpha+1)}{\Gamma(\alpha+1)/2)\sqrt \pi}
 \int_{-1}^{1} e^{i\lambda xt} (1-t^2)^{\alpha-1/2}\>  dt = j_\alpha(\lambda x)
\end{equation}
 for $\lambda\in \b C$, $x\ge0$, $\alpha=d/2-1$. Here,
 $\phi_\lambda=\phi_{\tilde\lambda}$ is equivalent to
 $\lambda^2=\tilde\lambda^2$, and we conclude that
$\hat X=\chi_b(X)=\{\phi_\lambda:\> \lambda\in[0,\infty)\}$ where
we may identify $\hat X$ with $[0,\infty)$ in this way   topologically.
The construction of the Plancherel measure $\pi$ in the proof above shows that
 $\pi\in M^+([0,\infty))$ is also the spherical projection of the Lebesgue
 measure up to normalization and agrees thus with the Haar measure above up to normalization.
This is no accident, as it might be derived easily from the symmetry of the functions
 $\phi_\lambda(x)$
in the variables $\lambda, x$ that the Bessel-Kingman hypergroups are
 self-dual like the underlying groups $\b R^d$.

We next observe that the Bessel functions $j_\alpha$ and the associated
function $\phi_\lambda(x)= \phi_\lambda^\alpha(x)$ depend analytically on
$\alpha$ and that by a kind of analytical extension (see Section
\ref{section-matrix-cones} where a more general case is treated), the product formula
\eqref{Besselformel_1} as well as the positive integral representation
\eqref{Besselformel_laplace}
remain valid for all $\alpha\in \b R$ with $\alpha>-1/2$ (with the degenerated case $\alpha=-1/2$).
These formulas then imply that for each $\alpha\ge1/2$, there exists a unique
so-called Bessel-Kingman hypergroup on  $[0,\infty)$ with the dual space $\{
\phi_\lambda^\alpha:\> \lambda \in [0,\infty)\}$, with $0$ as identity, and
the identity as involution. These hypergroups are also self dual,  the Haar
measure $\omega_\alpha$  is given as in  \ref{examples-haar-hyper}(2), and the Plancherel
measure is equal to $\omega_\alpha$ up to normalization.
\end{example}

We finally  recapitulate that if $K$ is a compact group of automorphisms acting
continuously on some locally compact abelian group $V$, then the orbit space $V^K$
becomes a commutative orbit hypergroup. This can be easily extended to
the case where $K$ is  a compact group of automorphisms acting
continuously on some commutative hypergroup $(X,*)$. This generalization  will be crucial
to explain  how the Bessel hypergroup structures on matrix cones lead to
continuous series of commutative hypergroup structures on Weyl chambers associated with Dunkl operators of
type $B_N$ in Section \ref{section-matrix-cones}.
We here collect some facts  without proofs.
For further details we refer to \cite{J}, where this concept is embedded in
the more general context of  orbital morphisms.

\begin{remark}\label{hypergroup-automorphisms}
Let $(X,*)$ be a commutative hypergroup with  dual space
$\hat X$.
\begin{enumerate}
\item[\rm{(1)}] A homeomorphism $k:X\to X$ is called a hypergroup automorphism
  if $\delta_{k(x)}*\delta_{k(y)} = k(\delta_{x}* \delta_{y})$
  for all $x,y\in X$.

It follows readily for such an automorphism, that $k(e)=e$ and  $k(\bar
x)=\overline{k(x)}$ for $x\in K$. Moreover, $k(\mu*\nu)=k(\mu)*k(\nu)$ for all $\mu,\nu\in M_b(X)$.
Furthermore,
if $k_1,k_2$ are hypergroup automorphisms of $(X,*)$, then
so are $k_1^{-1}$ and $k_1k_2$.
\item[\rm{(2)}]  Now assume that $K$ is a compact group of hypergroup automorphisms acting
continuously on $(X,*)$. Then we may form the orbit space $X^K:=\{K.x:\>
x\in X\}$ which is locally compact w.r.t.~the quotient topology. It can be
easily checked that
$$M_b(X|K):=\{\mu\in M_b(X):\> k(\mu)=\mu \quad \text{for all}\quad k\in
K\}$$
is a Banach-$*$-subalgebra of $M_b(X)$ which is isometrically isomorphic as a
Banach space with
the Banach space $M_b(X^K)$ by taking images of measures w.r.t.~the
canonical projection from $X$ onto $X^K$.  A transfer of the convolution on
$M_b(X|K)$ to $M_b(X^K)$ then leads to a probability preserving convolution
$*$ on $M_b(X^K)$, and it can be easily checked that with this convolution,
$(X^K,*)$ becomes a commutative hypergroup with identity $\{e\}$ and involution
$\overline{K.x}:= K.\bar x$ for $x\in X$. Furthermore, if we define for
$\phi\in \hat X$ the $K$-invariant function $\phi^K(x):=\int_K\phi(k.x)\>
d\omega_K(k)$ for $x\in X$, one obtains readily that
$\{\phi^K:\> \phi\in \hat X\} \subset   (X^K)^\wedge.$
Moreover, if $X$ has subexponential growth, then so also has $X^K$, and the
Plancherel measure $\pi$ of $X^K$ satisfies $supp\>\pi=  (X^K)^\wedge$. This implies
that similar as above in the group case  the equality
$$\{\phi^K:\> \phi\in \hat X\} =  (X^K)^\wedge $$
holds; see Section 13 of \cite{J} or the proof of Theorem 4.1 of \cite{R8}.
\end{enumerate}
\end{remark}

\subsection{Bessel functions associated with root systems and symmetric spaces of Euclidean type}\label{Bessel-root-systems}

The radial parts  of invariant differential operators on a Riemannian symmetric space can be expressed in terms of commuting algebras of Dunkl or Dunkl-Cherednik operators. The rational Dunkl operators are hereby related
to symmetric spaces of the Euclidean type, and their spherical functions
appear as Dunkl-type Bessel functions with certain discrete multiplicity values.

To explain this connection, let us start with the underlying concepts. For more background and detail the reader may consult the monographs \cite{Hel1}, \cite{Hel2}, or also  \cite{Ko}, Chapt. I, II, III.

Let $\frak  g$ be a real semisimple Lie algebra, i.e., a Lie algebra over $\b R$ on which the Killing form $B$ is nondegenerate. An involutive automorphism $\theta$ of $\frak g$ is called a \textit{Cartan involution} and the  corresponding eigenspace decomposition
$\, \frak  g = \frak  k \oplus \frak  p$ into $+1$ and $-1$ eigenspaces of $\theta$ is called a \textit{Cartan decomposition} of $\frak g$ if $B$ is negative definite on $\frak k$ and positive definite on $\frak p$.
In particular, $(\frak p, B)$ is a finite-dimensional Euclidean vector space.
Cartan involutions exist and  are all conjugate under inner automorphisms.
Let $G$ be a connected Lie group with real semisimple Lie algebra $\frak
g$. Such a Lie group is called semisimple.
 It can be shown that the Cartan involution on $\frak g$ is the
differential of a unique involutive automorphism of $G$, also denoted by
$\theta$, and that the fixed point
 subgroup $\, K:= \{ g\in G: \theta(g) = g\}$ is just the connected Lie
 subgroup of $G$ with  Lie algebra $\frak k.$
The group  $K$ is compact iff $G$ has finite center; such a choice of $G$ is always possible. If $G$
has finite center and $\frak p \not=\{0\}$ then $G/K$ is called a \textit{Riemannian symmetric space of the non-compact type}.

\smallskip

In the situation described above, $K$ acts on $\frak g$ via the adjoint
 representation $Ad$. We denote this action by
 $k.X:=Ad(k)X$.  Recall that
\[\text{exp}(k.X) \,=\, k\,\text{exp}(X) k^{-1} \quad \text {for }\, X\in \frak g, \, k\in K.\]
As $[\frak k, \frak p]\subseteq \frak p$, $K$ leaves $\frak p$ invariant, and  acts via orthogonal transformations (with respect to the Killing form) on $\frak p.$
The semidirect product $G_0 := K\ltimes \frak p$ is called the \textit{Cartan motion group} associated
with $G/K$, and $G_0/K\cong \frak p$ is called a \textit{symmetric space of Euclidean type}.

\smallskip

Choose now a maximal abelian subspace $\frak a $ of $\frak p$.
All such subspaces have the same dimension which is called the rank of $G/K$.
For each $\alpha$ in the dual space $\frak a^*$ of $\frak a$ let
\[ \frak g_\alpha = \{ X\in \frak g: [H,X] = \alpha(H)X \quad \forall H\in \frak a\}.\]
Those $\alpha$ with $\alpha\not=0$ and $\frak g_\alpha\not=\{0\}$ are called the (restricted) \textit{roots} of $\frak g$ w.r.t. $\frak a$.
The simultaneous diagonalization of the commuting operators
$ad\, H, \, H\in \frak a$ leads to the root space decomposition
\[ \frak g = \frak g_0 +  \sum _{\alpha\in \Sigma} \frak g_\alpha \]
where $\Sigma$ is the set of restricted roots. The spaces $\frak g_\alpha$ are
called \textit{root spaces}, and
$\, m_\alpha := \text{dim}_{\b R} \frak g_\alpha\,$ is called the \textit{multiplicity} of $\alpha$. We consider $\Sigma$ as a subset of $\frak a$, identifying $\frak a$ with its dual space via the Killing form. The set $\Sigma$ is
a crystallographic, but not necessarily reduced root system in the Euclidean
space $(\frak a,B) \cong (\b R^N, \, \langle\,.\,,\,.\,\rangle).$ Denote by $W$ the associated Weyl group. The action of $W$ on $\frak a$ is obtained  from the action of $K$ on $\frak p$ as follows:
Consider the centralizer and the normalizer of $\frak a$ in $K$,
\begin{align*}
 &Z_K(\frak a) := \{ k\in K: k.H = H \quad \forall H\in \frak a\}\\
 &N_K(\frak a) := \{ k\in K: k.\frak a = \frak a\}
\end{align*}
Then $W$ is realized as the quotient $N_K(\frak a)/Z_K(\frak a).$
Each $K$-orbit in $\frak p$ meets $\frak a$, and
the orbit space $\frak p^K$ is homeomorphic to the closure of
the Weyl chamber $\frak a_+$ corresponding to some fixed positive subsystem
of $\Sigma$,
\[ \frak p^K \cong \overline{\frak a_+}\]
(see for instance \cite{Hel2}, Prop. I.5.18).

\smallskip
The \textit{spherical functions} of the Euclidean symmetric space $G_0/K$
are
defined as the spherical functions of the Gelfand pair $(G_0 = K\ltimes \frak p, K)$. Equivalently, they can be characterized  as follows:
Denote by $S(\frak p)^K$ the algebra of $K$-invariant polynomials on $\frak p$, and by $p(\partial)$ the constant coefficient differential operator  corresponding to $p\in S(\frak p)$. Then a function $\phi:\frak p\to \b C$ is spherical
if and only if  it is smooth and satisfies $\varphi(0)=1$ as well as
\[ p(\partial)\phi\,=\, p(\lambda) \varphi \quad \forall p\in S(\frak p)^K\]
with some $\lambda\in \frak a_{\b C}.$ This just means that $\varphi$ is a joint eigenfunction of all $K$-invariant
constant coefficient differential operators on $\frak p$. Being $K$-invariant,
$\varphi$ can also be considered as  a $W$-invariant function on $\frak a$
which is a joint eigenfunction of the $K$-radial parts of the operators from  the commutative algebra $\,\{ p(\partial), \, p \in S(\frak p)^K\}$.  An important
member of this algebra  is the Laplacian $\Delta_\frak p$.
Its radial part can be calculated as
\begin{equation}\label{radpart}\text{Rad}\,\Delta_\frak p =\, \Delta_{\frak a} + \sum_{\alpha\in \Sigma_+}
m_\alpha \frac{1}{\langle\alpha, \,.\,\rangle} \partial_\alpha,\end{equation}
see \cite{Hel2}, Chapter II.

\begin{example} For $d\geq 2$ consider the Lorentz group
$G= SO_0(d,1)$ with maximal compact subgroup $K=SO(d)$.
The symmetric space $G/K$ is a  real hyperbolic space which can be topologically
identified with the hyperboloid
\[H^d = \{x\in \b R^{d+1}: x_1^2 +\ldots + x_d^2 - x_{d+1}^2 = 1, \, x_{d+1}>0\}.\]
In this case, $G_0=SO(d)\ltimes \b R^d$ is the Euclidean motion group,
where $K=SO(d)$ acts in the subspace perpendicular  to the $x_{d+1}$-axis.
The symmetric space $G_0/K$ is naturally identified with the vector space $\mathbb R^d$ which is just the tangent space of $G/K$ in the base point $eK\cong (0, \ldots, 0,1)\in  H^d$.
Moreover, $\frak a\cong \b R$ and the rank of $G/K$ is $1$.
The $K$-invariant constant coefficient differential operators on $\b R^d$  are in this case exactly those which
are algebraically generated by the Laplacian $\Delta$. When considered as functions on $\b R$,  the spherical functions are  characterized as the smooth and even eigenfunctions  of the radial part of $\Delta$. These are precisely  the Bessel functions $\varphi_\lambda(x) = j_{d/2-1}(\lambda x)$ with $\lambda \in [0,\infty),$ c.f. Section \ref{S:Dunklkernel}.
\end{example}

The spherical functions of a Euclidean symmetric space $G_0/K$
can be written down explicitly in terms of an integral representation of  
Harish-Chandra type (see \cite{Hel2}, Chap. IV). In fact, this is also an
immediate consequence of Theorem \ref{properties-Euclidean-Fourier-transform}:

\begin{theorem}\label{harish}
The spherical functions of the Euclidean symmetric space $G_0/K$
are (as functions on $\frak a$) given by
\[ \varphi_\lambda(x) = \int_{K} e^{iB(\lambda, k.x)}dk\]
where $\lambda \in \frak a_{\b C}$ (the complexification of $\frak a$) and
$\varphi_\lambda = \varphi_{\widetilde\lambda}$ iff $\lambda$ and $\widetilde\lambda$ are in the same $W$-orbit in $\frak a_\b C$.
 \end{theorem}

By a convexity theorem of Kostant (Theorem 10.5 in Chap. IV of \cite{Hel1}) this integral formula can be written
in the form
\[ \varphi_\lambda(x) = \int_{co(W.x)} e^{iB(\lambda, \xi)} d\nu_x(\xi)\]
with a probability measure $\nu_x$ on the convex hull $co(W.x)$ of the Weyl group orbit of $x$.
This generalizes  the Mehler integral representation \eqref{Besselformel_laplace}
in rank one.

\medskip
Let us now proceed to the interpretation of the spherical functions $\varphi_\lambda$ as
Bessel functions of Dunkl type. We follow the expositions in  \cite{dJ5}, \cite{R6}.

\begin{definition}
Let $R$ denote the reduced root system obtained from $\Sigma$ by normalizing
all roots to length $\sqrt{2}$. (So some roots in a component $BC_n$ of $\Sigma$ will coincide in $R$). Define the multiplicity $k$ on $R$ by
$k_\alpha := \frac{1}{2}\sum_{\beta\in \b R\alpha\cap \Sigma_+} m_\beta$.
Data $(R,k)$ obtained in this way from a symmetric space are called \textit{geometric}.
\end{definition}

Consider the Dunkl operators associated with $R$ and $k$.
Comparison of the radial part \eqref{radpart} with the $W$-invariant restriction  $\text{Res}\,\Delta_k$
of $\Delta_k$ shows that
\[ \text{Rad} \,\Delta_\frak p\,=\, \text{Res}\,\Delta_k.\]
More general, there is a restriction isomorphism (the generalized Harish-Chandra isomorphism)
\[ \Phi: S(\frak p)^K \,\longrightarrow \, S(\frak a)^W \]
where $S(\frak a)^W$ denotes the algebra of $W$-invariant polynomials on $\frak a$
 (called $\mathcal P^W$ in Section \ref{Dunkloperators}).
It can now be shown (see \cite{dJ5} for the details) that for each $p\in S(\frak p)^K,$
\[ \text{Rad}\,p(\partial)\,=\, \Phi(p)(T(k)).\]
From this and the uniqueness of the solution to both differential systems, it is
not hard to see that as functions on $\frak a\cong \b R^N$, the spherical functions of the flat symmetric space $G_0/K$ are given by
the Dunkl-type Bessel functions associated with $R$ and $k$. That is
\[ \{ \varphi:\b R^N\to \b C: \varphi \text{ spherical for } G_0/K\} =
 \{ J_k(\,.\,, \lambda), \lambda \in \b C^N\}. \]

Thus, the positive integral representation of Theorem \ref{C:Bochner} for Bessel functions associated with arbitrary  non-negative multiplicities generalizes the Harish-Chandra integral
for spherical functions of Euclidean symmetric spaces.

\smallskip

To obtain a generalized translation with nice properties in Dunkl theory,
product formulas for the Dunkl kernel and the Bessel function would be most desireable. Indeed, for the geometric cases, a positive  product formula is
guaranteed by the interpretation of Bessel functions as spherical functions.
Let us briefly describe these matters.

Consider as above the Euclidean symmetric space $G_0/K$ associated with $G/K$
and the Bessel functions $J_k$ with corresponding geometric data $(R,k)$.
As described in Section \ref{Section-Gelfand-pairs}, the convolution of $K$-biinvariant functions on $G_0$ can be interpreted as the convolution
of the commutative orbit hypergroup $\frak p^K $
 which is given by
\[ \delta_{K.x} * \delta_{K.y} = \int_K \delta_{K.(x+ k.y)} dk, \quad x,y\in \frak p.\]
Recall that each orbit $K.x$ contains a unique element $x_+\in \overline{\frak a_+}$ and that the mapping $\, \frak p^K\to \overline{\frak a_+}, \, x \mapsto x_+ $ is a homeomorphism.
The above convolution therefore transfers  to a hypergroup convolution on the closed Weyl chamber
$\overline{\frak a_+}$ which we consider as a subset of $\b R^N$. It is given by
\[ \delta_x *\delta_y = \int_{K} \delta_{(x+k.y)_+}dk.\]
The neutral element of the hypergroup $X= (\overline{\frak a_+},*)$  is $0$ and the involution is given by
$\overline{x_+} = (-x)_+$.
Further, the space $\chi(X)$ of continuous multiplicative functions and the dual space of $X$
are
\[ \chi(X) = \{ J_k(\,.\,, \lambda),\,\lambda\in \frak a_ {\b C}\}; \quad
\widehat X = \chi_b(X) = \{J_k(\,.\,, i\lambda),\,\lambda\in \frak a_+\},\]
c.f.  Theorem \ref{properties-Euclidean-Fourier-transform}.
A Haar measure is given by the Dunkl-type weight function $w_k$.
The restriction of the Dunkl transfrom to Weyl group invariant functions therefore coincides with the Fourier transform on the hypergroup $X$.

\begin{example} [$SL(n,\b C)/SU(n)$]
Consider the semisimple connected Lie group $G= SL(n,\b C)$ with
Lie algebra
$\, \frak g= sl(n,\b C) = \{ x\in M_n(\b C): \,tr x = 0\}.$
The mapping  $\theta(x) = -x^*$ with  $x^* = \overline x^t$ is a Cartan involution on $\frak g$, and
the corresponding Cartan decomposition is  $\frak g = \frak k \oplus \frak p$ with
\begin{align*}
\frak k = &\{ x\in \frak g: x^* = -x\} &\text{(skew-Hermitian matrices)}\\
\frak p = & \{ x\in \frak g: x^* = x\}=: H_n &\text{(Hermitian matrices).}
\end{align*}
A possible choice for a maximal abelian subalgebra are the diagonal matrices
\[ \frak a = \{ x= \text{diag}(\xi_1, \ldots, \xi_n): \, \xi_i\in \b R, \, \sum_{i=1}^n \xi_i = 0\}\]
which we identify with a subspace of $\b R^n$.
The adjoint action of the maximal compact subgroup  $K = exp \,\frak k = SU(n)$ on $\frak p = H_n$
is given  by conjugation, $(u,x) \mapsto uxu^{-1}.$
Each matrix $x\in H_n$ is unitarily equivalent to a unique diagonal matrix
with the eigenvalues of $x$ being ordered by size. Therefore the orbit
$H_n^K$ can be (actually topologically) identified with the closed set
\[ \overline C:= \{ \xi= (\xi_1,\ldots, \xi_n) \in \b R^n: \, \xi_1\geq \ldots \geq \xi_n, \, \sum_{i=1}^n \xi_i = 0\}.\]
Notice that $\overline C$ is a closed Weyl chamber of the symmetric group $S_n$,
which is the Weyl group of the Riemannian symmetric space $G/K$. Indeed, the  restriction of the
$K$-action on $H_n$ to $\frak a$ is realized by permutations of the diagonal entries.
According to our general results, the spherical functions of
the Euclidean symmetric space $(K\ltimes H_n)/K$ are given by the functions
\[ \{\xi\mapsto J_1^A(\xi, \lambda),\,\, \lambda\in \frak a_{\b C}\}\]
where $J_1^A$ denotes the Bessel function associated with root system $A_{n-1}$
and multiplicity $k=1$.
The value  $k=1$ results from the fact that $SL(n,\b C)$ is a complex Lie group and therefore the real dimensions of all root spaces are equal to $2$.
The Bessel function $J_1^A$ is given explicitly by
\[ J_1^A(\xi,\lambda) = \kappa\sum_{w\in S_n} \frac{\epsilon(w)}{\pi(\xi)\pi(\lambda)} \,e^{\langle \xi, w\lambda\rangle}\]
where $\kappa\in \b C$ is a constant, $\langle \,.\,,\,.\,\rangle$ denotes the usual
inner product in $\b R^n,$ $\epsilon(w)$ is the signature of $w$ and $\pi$ is the fundamental alternating polynomial
\[ \pi(\xi) = \prod_{i<j} (\xi_i-\xi_j).\]
Indeed, this alternating sum representation for $J_1^A$ follows from the explicit formula for the spherical functions of $G/K$  in \cite{Hel2}, but it is also a special case
of the following general result:

\begin{theorem} [\cite{Du4}, Propos.1.4.]  \label{alternate}  For fixed (not necessarily crystallographic)  root system $R\subset \b R^N$, the Bessel function $J_{k+1}$ associated with $R$
is obtained from the Dunkl kernel $E_k$ associated with $R$ by
\[ J_k(x,y) = \kappa_k\sum_{w\in W} \frac{\epsilon (w)}{\pi(x)\pi(y)} E_k(wx,y)\]
where $\kappa_k$ is a constant, $\epsilon(w)$ is the determinant of $w$ as an element of $O(N)$ and $\pi$ is the fundamental alternating polynomial
$\,\pi(x) = \prod_{\alpha\in R_+} \langle \alpha, x\rangle.$
\end{theorem}

Let return to our example and take a look at the convolution of the orbit hypergroup
$\,H_n^K \cong \overline C$. It is given by
\[ \delta_\xi * \delta_\eta = \int_{SU(n)} \delta_{\sigma(\xi + u\eta u^{-1})}du \quad (\xi, \eta \in \overline C)\]
where $\xi$ and $\eta$ are identified with the corresponding diagonal matrices,
and $\sigma(x)$ denotes the eigenvalues of $x$ ordered by size.
The convolution product describes the distribution of the spectra of the
possible sums $x+y$ of Hermitian matrices $x, y$ with given spectra $\sigma(x) = \xi, \, \sigma(y) = \eta$. The precise description of the support of the measure $\delta_\xi * \delta_\eta$ had been a long-standing problem, formulated as Horn's conjecture, until it was completely solved
only ten years ago by Klyachko, Knutson and Tao (\cite{Kl}, \cite{KT}).
\end{example}

The symmetric spaces $SL(n,\mathbb R)/SO(n)$ and $SL(n,\mathbb H)/Sp(n)$ can be treated in a similar way. They lead to Bessel functions of type $A_{n-1}$ associated with
the multiplicity parameter $k=\frac{1}{2}$ and $k=2$, respectively. The Bessel function $J_2^A$ can be obtained
from the Dunkl kernel $E_1^A$ by an alternating formula according to Theorem \ref{alternate}.\\
\\

The (irreducible) Riemannian symmetric spaces of non-compact type are completely classified, see Chapter X of
\cite{Hel1}. Here is a list of the infinite classes of connected spaces $G/K$, the type of $R$, which is $A_n, B_n $ or $ D_n$  (with renormalized roots of unit length $\sqrt{2}$),  as well as the
values of the multiplicity function $k$. For $R$ of type $B$, the multiplicity $k$ is of the from $k=(k_1, k_2)$ where $k_1$ is the value on $\pm\sqrt 2 e_i$ and $k_2$ the value on $\pm e_i\pm e_j$.

In those cases where $\frak g$ admits a complex structure, the multiplicity is $k=1$
(Classes A, B, C, D). The data are taken from Table C.2. in the appendix of \cite{Kn} and Table 9 in the appendix of \cite{OV}.
{\footnotesize
\begin{equation*}
\begin{tabular}{|c|c|c|c|c|}\hline
\text{Class}  & $G$& $K$& $ R $ & $k $ \\ \hline
A        &   $SL(n,\b C), \,n\geq 2$  &$SU(n)$ & $  A_{n-1} $  & $1 $   \\
AI    & $SL(n,\b R), \,n\geq 2$ &$ SO(n) $&$ A_{n-1}$ & $1/2$ \\
AII   & $SL(n,\b H), \,n\geq 2$ &$ Sp(n)$ & $A_{n-1}$ &$ 2$ \\ \hline
B        &  $SO(2n+1, \b C), \,n\geq 2$  & $SO(2n+1)$ &$B_n$ &  $ 1$\\
C        &  $Sp(n,\b C), \,n\geq 3$ &$ Sp(n) $ &$ B_n$ &$ 1$\\
D      &  $SO(2n,\b C), \,n\geq 4$ & $SO(2n)$ & $D_n$  & $1 $\\ \hline
AIII   &  $SU(p,q)$ & $ S(U(p)\times U(q))  $&  $B_q $ &  $k_1 = p-q + 1/2$\\
& $p\geq q\geq 1 $& & &  $k_2 = 1 $\\ \hline
BDI &  $SO_0(p,q),\, p\geq q \geq 1, $ &  $SO(p)\times SO(q) $ &  $ B_q/D_q $ &  $k_1=(p-q)/2$\\
  & $  p>1, p+q\not=4$         &&& $k_2  = 1/2 $ \\ \hline
CII &  $Sp(p,q)$  &  $Sp(p)\times Sp(q) $ &  $B_q $ &  $k_1 = 2(p-q) + 3/2$ \\
&$p\geq q\geq 1 $&&& $k_2 = 2 $\\ \hline
CI & $Sp(n,\b R), \, n\geq 1 $ &  $U(n) $ &  $B_n $ & $ k_1=k_2 = 1/2 $ \\
DIII &  $SO^*(2n), \, n>2 $ even &  $U(n) $ &  $B_n $ &  $k_1= 1/2, \, k_2 = 2 $\\
DIII &  $SO^*(2n), \, n>2 $  odd  &  $U(n) $ &  $B_ n $ &   $k_1 = 5/2, \, k_2 = 2 $\\ \hline
\end{tabular}
\end{equation*}

}
In all these geometric cases, the Bessel functions $J_k$ are hypergroup characters and satisfy the positive product formula
\[ J_k(x,\lambda)J_k(y,\lambda) = \int_K J_k\bigl((x+k.y)_+, \lambda \bigr) dk\quad \forall \lambda\in \frak a_\mathbb C.\]
The following is a natural conjecture, which is however open so far in general:

\begin{remark} Let $R$ be a reduced root system in $\b R^N$ and $k\geq 0$
an arbitrary non-negative multiplicity function. Choose a closed Weyl chamber $\overline{C}$
of $R$. Then it is conjectured that the Bessel function $J_k$ associated with
$R$ and $k$ satisfies a product formula of the form
\[ J_k(x,\lambda) J_k(y,\lambda) = \int_{\overline C} J_k(\xi, \lambda) d\mu_{x,y}^k(\xi)\quad \forall \lambda\in \b C^N\]
with probability measures $\mu_{x,y}^k\in M^1(\overline{C})$, and that
this product formula leads to a commutative hypergroup structure on $\overline C$
with the functions $x\mapsto J_k(x,i\lambda)$ as characters. For $k=1$, the existence of such a product formula has been
recently obtained also for non-crystallgraphic root systems (where there is no Lie-theoretic background)
in \cite{BBC2} by stochastic methods.
\end{remark}

In Section \ref{Dunklhyper} we shall present three continuous series
of Dunkl structures where this conjecture is also true, and where the convolution can be written in an explicit form. These examples are of type $B$ and are obtained
by interpolation of the three discrete series of orbit hypergroup convolutions
related to the Cartan motion groups  of noncompact Grassmann manifolds over $\b R, \b C$ and $\b H$, that is the classes BDI, AIII, CII. We shall derive these convolutions from hypergroup algebras on cones of positive semidefinite matrices which have matrix Bessel functions as characters.
These are the topic of the folowing section.

 \subsection{Bessel hypergroups on matrix cones}\label{section-matrix-cones}

In this section, we construct three series of Bessel hypergroups
on cones of positive semidefinite matrices over one of the skew-fields $\b F= \b R, \b C$ or $\b H$. Each of them is obtained from a discrete series of orbit
hypergroup convolutions derived from radial analysis on matrix spaces with varying
dimension. The main references for this section are \cite{R8} and  \cite{FT}. For an introduction into the analysis on  symmetric cones, we refer the reader to the
monograph \cite{FK}. We start with a geometric situation which generalizes
the action of the orthogonal group in $\b R^d$ to a higher rank setting:

For natural numbers $p\geq q$, consider the matrix space $M_{p,q} = M_{p,q}(\b F)$ of $p\times q$ matrices over $\b F$. We regard $M_{p,q}$ as a real vector space of dimension $dpq$ with $d= \text{dim}_\b R \b F$, equipped with
 the  inner product $\langle x,y\rangle = \text{Re}\, tr(x^*y)$ (where $x^* = \overline x^t$ denotes the conjugate transpose) and the norm $\|x\|= \sqrt{tr(x^*x)}.$ Let further
$H_q = H_q(\b F)= \{x\in M_{q,q}(\b F): x= x^*\}\,$ denote the space of
Hermitian $q\times q$-matrices and
\[\Pi_q= \Pi_q(\b F):= \{x^2: x\in H_q\}\subset H_q\]
the closed cone of positive semidefinite matrices over $\b F$.
Its interior $\Omega_q$, consisting of the strictly positive matrices, is
a symmetric cone in the sense of \cite{FK}.
The unitary group $U_p = U(p, \b F)$ over $\b F$ acts
on $M_{p,q}$ by multiplication from the left as a closed subgroup of the orthogonal group of $M_{p,q}$,
\[ U_p \times M_{p,q}\to M_{p,q}, \quad (u,x)\mapsto ux.\]
It is easy to  see that the orbit space $M_{p,q}^{U_p}$ for this action
can be topologically identified with the cone $\Pi_q$ via
$\,  U_p.x \mapsto \,\sqrt{x^*x}\, =:\, |x|.\,$
 Here for $r\in \Pi_q\,,\, \sqrt{r}$ denotes the unique positive semidefinite square root of $r$.

The additive group structure on $M_{p,q}$ induces an orbit hypergroup convolution
on $\Pi_q$. To obtain the convolution, Haar measure and dual space of this hypergroup, we apply the results of Section \ref{Section-Gelfand-pairs}.
The calculation of the convolution is similar as for $(\b R^d)^{O(d)}$ in Example
\ref{E:Bessel_Kingman}. For $r,s\in \Pi_q$ we obtain
\[ (\delta_r * \delta_s)(f) = \int_{U_p}
f\bigl(\vert \sigma_0 r + u \sigma_0 s\vert\bigr) du\]
with the block matrix
$ \sigma_0  :=  \begin{pmatrix} I_q\\ 0
\end{pmatrix} \in M_{p,q}.\,$
The orbit of $\sigma_0$ is the Stiefel manifold
\[ \Sigma_{p,q} = \{ x\in M_{p,q}: x^*x= I_q\}\]
and one obtains
\begin{equation}\label{convo}
(\delta_r * \delta_s)(f)= \int_{\Sigma_{p,q}} f\bigl(\sqrt{r^2 + s^2 + r\widetilde\sigma s +(r\widetilde\sigma s)^*}\bigr) d\sigma
\end{equation}
where $d\sigma$ is the normalized surface measure on $\Sigma_{p,q}$ and $\widetilde\sigma = \sigma_0^*\sigma$ is the $q\times q$-matrix whose rows are given by the first $q$ rows of $\sigma$.
 The neutral element of the orbit hypergroup $(\Pi_q, *)$ is $0$, and the involution is the identity mapping (because $x\in M_{p,q}$ and $-x$ are in the same $U_p$-orbit).
A Haar measure is provided by the image measure of the Lebesgue measure on $M_{p,q}$ under the mapping $x\mapsto |x|$. Calculation in polar coordinates gives
\[ \omega(f) = c \int_{\Pi_q}  f(\sqrt{r})\Delta(r)^\gamma dr\]
where $c>0$ is a constant, $\Delta(r)$ is the determinant of $r$
(in case $\b F=\b H$, one takes the Dieudonne determinant $\Delta(r) = (\text{det}_{\b C}r)^{1/2}, $ with $r$ considered as a complex
$2q\times 2q$-matrix in the usual way). Finally,
\[\gamma = \frac{d}{2}(p-q+1) -1.\]
 The dual space turns out to consist of Bessel functions associated with the symmetric cone $\Omega_q$. These are hypergeometric functions of matrix
argument defined in terms of the spherical polynomials of  $\Omega_q$, see \cite{FT}. Let us recall those.
The spherical polynomials of  $\Omega_q$ are parametrized by partitions $\lambda= (\lambda_1\geq \ldots \geq \lambda_q)\in \b Z_+^q$, for which we write $\lambda\geq 0$ for short. They are given, up to normalization, by
\[ Z_\lambda (x) = c_\lambda \int_{U_q} \Delta_\lambda(uxu^{-1})du, \quad x\in H_q\]
with the power functions
\[ \Delta_\lambda(x) = \Delta_1(x)^{\lambda_1-\lambda_2} \Delta_2(x)^{\lambda_2-\lambda_3} \cdot\ldots\cdot \Delta_q(x)^{\lambda_q}.\]
The $\Delta_i(x)$ are the principal minors of  $\Delta(x)$, and the constant
 $c_\lambda>0$ can be chosen such that
\[ (\text{tr} \,x)^k = \sum_{\lambda\geq 0, |\lambda|=k} Z_\lambda(x).\]
The polynomial $Z_\lambda$ is homogeneous of degree $|\lambda|$ and
invariant under conjugation by $U_q$. It therefore depends only on the eigenvalues of its argument.

\begin{definition} The Bessel functions $\mathcal J_\mu$ associated with $\Omega_q$
are defined on $H_q$ by
\[ \mathcal J_\mu(x) = \, _0F_1(\mu;-x):= \sum_{\lambda\geq 0} \frac{(-1)^{|\lambda|}}{(\mu)_\lambda|\lambda|!}\,Z_\lambda(x)\]
where $\mu\in \b C$ is an index with $\text{Re}\, \mu > \frac{d}{2}(q-1)$ and $(\mu)_\lambda$ denotes
the generalized Pochhammer symbol
\[(\mu)_\lambda := \prod_{j=1}^q \bigl(\mu - \frac{d}{2}(j-1)\bigr)_{\lambda_j}.\]
\end{definition}

If $q=1$ then $J_\mu$ does not depend on $d$ and is given by
\[ \mathcal J_\mu\bigl(\frac{x^2}{4}\bigr) = j_{\mu-1}(x) \quad (x\in \mathbb R).\]

\begin{lemma}\label{kegeldual}
The dual space of the orbit hypergroup $M_{p,q}^{U_p}\cong \Pi_q$ consists of the Bessel functions
\[ \varphi_s(r)  = \mathcal J_\mu\bigl(\frac{1}{4}sr^2s\bigr), \,\,s\in \Pi_q,\]
with index $\,\mu = pd/2.$
\end{lemma}

\begin{proof} According to  Theorem \ref{properties-Euclidean-Fourier-transform},
the dual space is given by the functions $\phi_s, \,s\in \Pi_q$ with
\begin{equation}\label{Bochnerrep}
 \phi_s(r) = \,\int_{U_p} e^{-i(u\sigma_0 r\vert \sigma_0 s)} du\, =\,
 \int_{\Sigma_{p,q}} e^{-i(\sigma\vert \sigma_0 sr)} d\sigma.
\end{equation}
By Propos. XVI.2.3 of \cite{FK} these coincide with the stated Bessel functions.
\end{proof}

The hypergroup Fourier transform on $M_{p,q}^{U_p}$ is a Hankel transform
involving the Bessel function $\mathcal J_\mu$ with $\mu=pd/2$. It coincides
with the (group) Fourier transform of $U_p$-radial functions on $M_{p,q}$.
An $L^2$-theory for the Hankel transform with a continuous range of  real indices $\mu$ was established in \cite{FT} in the general setting of symmetric cones. For matrix cones over $\b F=\b R$, this had been done earlier by Herz,  \cite{Her}.
In our setting, the relevant result is as follows:

\begin{theorem}\label{L2}
Let $\mu > \frac{d}{2}(q-1)$ and define the measure $\omega_\mu$ on $\Pi_q$ by
\[\omega_\mu(f) = \frac{2^{-\mu q}}{\Gamma_{\Omega_q}(\mu)}\int_{\Omega_q}f(\sqrt r) \Delta(r)^\gamma dr\]
where $\, \gamma = \mu -\frac{d}{2}(q-1)-1$ and $\Gamma_{\Omega_q}$ is the gamma function of the cone $\Omega_q$ (see \cite{FK}). Put
\[\varphi_s^\mu(r):= \mathcal J_\mu\bigl(\frac{1}{4}sr^2s) = \varphi_r^\mu(s).\]
Then the Hankel transform
\[ f \mapsto \widehat f^\mu,\quad \widehat f^{\mu}(s) = \int_{\Pi_q} f(r) \varphi_s^\mu(r) d\omega_\mu(r)\]
is an isometric and self-dual isomorphism of $L^2(\Pi_q, \omega_\mu)$.
\end{theorem}

We shall now complement this result by a hypergroup structure within a slightly
smaller range of indices. The decisive observation is
that the integrand in convolution formula \eqref{convo} does not depend on the complete matrix $\sigma$ but depends only on the truncation $\widetilde \sigma$. This matrix  is contained in the closure of the matrix ball
\[ D_q := \{v\in M_{q,q}: v^*v <I\}\]
($r<s$ means that $s-r$ is positive definite). By a corresponding splitting of coordinates on the Stiefel manifold, one obtains the convolution formula
in a different form with $D_q$ as domain of integration. The dimension parameter $p$ then occurs as an exponent in the
density of the integral. To become precise, let
\[ \varrho:= d\bigl(q-\frac{1}{2}\bigr) +1 \quad \text{and }\quad \kappa_\mu := \int_{D_q} \Delta(I-v^*v)^{\mu-\varrho} dv\]
for $\mu\in \mathbb R$ with $\mu>\varrho -1$. The decisive splitting lemma, which requires $p\geq 2q$, is as follows:

\begin{lemma}\label{red} Let $p\geq 2q$ and put $\mu:= pd/2.$ Then for $f\in C(\Sigma_{p,q})$ of the form
\[ f(\sigma) = F(\widetilde\sigma), \quad \widetilde\sigma = \sigma_0^*\sigma\]
one has
\[ \int_{\Sigma_{p,q}} fd\sigma \,=\, \frac{1}{\kappa_\mu}  \int_{D_q} F(v)\,\Delta(I-v^*v)^{\mu-\varrho} dv.\]
\end{lemma}

Rewriting the convolution formula of the orbit hypergroups on $\Pi_q$ in this way allows to
 extend the product formula of the
Bessel functions and the associated hypergroup structure to
a continuous range of indices $\mu$.  The basic technique for this  is analytic
continuation with respect to $\mu$ from the discrete values $\mu=pd/2$ into the full half-plane $\{ \mu\in \mathbb C: \text{Re}\,\mu >\rho -1\}$ by use of Carlson's Phragmen-Lindel\"of-type Theorem (see \cite{Ti}, p.186).
This gives three continuous series of commutative hypergroup structures
on $\Pi_q$ (corresponding to $d=1,2,4$) which interpolate those occuring as orbit hypergroups for the indices $\mu= pd/2$. The following
theorem contains the main results of \cite{R8}:

\begin{theorem}\label{main2}
Let $\mu\in\mathbb R$ with $\mu > \varrho-1.$
\parskip=-1pt
\begin{enumerate}\itemsep=-1pt
\item[\rm{(a)}] The assignment
\[(\delta_r *_\mu \delta_s)(f) := \frac{1}{\kappa_\mu}\int_{D_q} f\bigl(\sqrt{r^2 + s^2 + rvs + sv^*r}\,\bigr)\,
\Delta(I-v^*v)^{\mu-\varrho}\, dv\]
defines a commutative hypergroup  $X_\mu  = (\Pi_q\,, *_\mu)$ with neutral element $0$
and the identity mapping as involution. For $\mu=pd/2$ with $p\geq 2q$, $X_\mu$ is the orbit hypergroup $M_{p,q}^{U_p}.$ The support of $\delta_r*_\mu\delta_s$ satisfies
\[\text{supp}\,(\delta_r*_\mu\delta_s) \subseteq  \{t\in \Pi_q: \|t\|\leq \|r\|+\|s\|\}.\]
\item[\rm{(b)}] A Haar measure of  $X_\mu$ is given by
the measure $\omega_\mu$ from Theorem \ref{L2}.
\item[\rm{(c)}] The dual space is given by
$\displaystyle \widehat{X_\mu} = \{\varphi_s^\mu(r) = \mathcal J_\mu\bigl(\frac{1}{4}sr^2s\bigr) : \,\,s\in \Pi_q\}.$
\item[\rm{(d)}] When identifying $X_\mu$ with its dual via $s\mapsto \varphi_s^\mu$,
the Plancherel measure of $X_\mu$ coincides with $\omega_\mu$.
\end{enumerate}
\end{theorem}

The proof of part (c) is based on the analytic extension of the product formula for the Bessel functions, but one also has to
make sure that there are no further characters apart from the Bessel functions
$\varphi_s^\mu, s\in \Pi_q$. For this, the Plancherel Theorem \ref{L2} as well
as subexponential growth of the hypergroup are needed. For details, see
\cite{R8}. Notice that the group cases also follow from the results in Section 3.2.

 For $\mu=\varrho -1=\frac{d}{2}(2q-1)$, which corresponds to the orbit hypergroup $M_{p,q}^{U_p}$ with $p=2q-1$,  the integral defining the convolution $*_\mu$  becomes singular. The degenerate version of the convolution formula
can be calculated after a suitable change of coordinates; this is carried out in \cite{R8}.

\subsection{Hypergroups for Dunkl-type Bessel functions of type $B$}\label{Dunklhyper}

In the last section we saw that the cone $\Pi_q$ carries a continuously parametrized family of commutative hypergroup
structures $*_\mu$ with $\mu>\varrho-1$, as well as additional orbit hypergroup structures for $\mu = pd/2, \, p\geq q$ an integer.  Let
\[ \mathcal M_q:= \big\{\frac{pd}{2}, \, p = q, q+1, \ldots\big\} \cup \, (\rho -1,\infty)\,.\]
In this section we study structures depending only on the matrix spectra. Under the mapping $r\mapsto \text{spec}(r)$, the hypergroup
convolutions $*_\mu$ on the matrix cone $\Pi_q$ induce a series of
hypergroup convolutions
on a Weyl chamber of type $B_q$. The characters  turn out to be  Dunkl-type Bessel functions. These hypergroups extend the harmonic analysis for Cartan motion groups associated with the Grassmann manifolds $U(p,q)/U_p\times U_q$.

We start with the geometric setting. Let $G=U(p,q)$ denote the indefinite unitary group of index $(p,q)$ over $\mathbb F$ with $p\geq  q$.
Its maximal compact subgroup $K$ is naturally isomorphic with $U_p\times U_q$.
We may identify $M_{p,q}$ with the tangent space of the Riemannian symmetric space $G/K$  in the coset $eK$.
This identification induces an action of $U_p\times U_q$ on $M_{p,q}$ according to
\begin{equation}\label{act2} \bigl((u,v), x\bigr)\mapsto uxv^{-1}, \quad u\in U_p,\, v\in U_q.\end{equation}
The associated orbit space is canonically parametrized by the possible singular spectra of matrices from $M_{p,q}$ and is homeomorphic to
\[ \Xi_q  = \{ \xi\in \mathbb R^q: \xi_1\geq \ldots \geq \xi_q\geq 0\}\]
which is a closed Weyl chamber of type $B_q$.
The action \eqref{act2} induces an action of $U_q$ on the cone $\Pi_q = \{ x^*x: x\in M_{p,q}\}$
by conjugation $(v,r) \mapsto vrv^{-1}$, which again leads to the chamber $\Xi_q$ as
orbit space.
From the explicit formula for the hypergroup convolution $*_\mu$ on $\Pi_q$ one
readily sees that the mapping $\, r\mapsto vrv^{-1}$ is a hypergroup automorphism in the sense of Remark \ref{hypergroup-automorphisms}. Therefore
each convolution $*_\mu$  with $\mu\in \mathcal M$  induces
a commutative hypergroup convolution $\circ_\mu$ on $\Xi_q$ by taking image
measures under the canonical mapping
\[\sigma: \Pi_q\to \Xi_q, \,r \mapsto \sigma(r),\] where $\sigma(r)$ denotes the set of eigenvalues of $r$ ordered by size. In the next theorem,
we summarize the results obtained in \cite{R8}; for convenience of notation, $\xi\in \Xi_q$ will be identified with the associated diagonal matrix from $\Pi_q$.

\begin{theorem}\label{main3} \parskip=-1pt\begin{enumerate} \itemsep=-1pt
                              \item[\rm{(1)}]
For fixed  $d\in\{1,2,4\}$ and each $\mu \in \mathcal M_q$
the chamber $\Xi_q$ carries a commutative hypergroup structure $Y_{\mu}=(\Xi_q,\circ_\mu)$ with convolution
\[ (\delta_\xi \circ_\mu \delta_\eta)(f) = \,\int_{U_q} (f\circ\sigma)(\xi *_\mu v\eta v^{-1}) dv.\]
The neutral element is  $0$ and the involution is given by the identity mapping.
 \item[\rm{(2)}]
 A Haar measure on  $Y_{\mu}$ is given by the image measure $\widetilde\omega_\mu$ of $\omega_\mu$ under the mapping $\sigma: r\mapsto \sigma(r),$ namely
\[ \widetilde\omega_\mu \,= \,d_\mu h_\mu(\xi)d\xi \quad\text{with}\quad
 h_\mu(\xi) =\,\prod_{i=1}^q \xi_i^{2\gamma + 1} \prod_{i<j} (\xi_i^2-\xi_j^2)^d\,\]
and $\, d_\mu = \bigl(\int_{\Xi_q} h_\mu(\xi) e^{-|\xi|^2/2}d\xi\bigr)^{-1}.$
  \item[\rm{(3)}] The dual space  $\widehat{Y_{\mu}}$ is parametrized by the chamber $\Xi_q$ and consists of the functions
\[\psi_\xi^\mu(\eta) = \int_{U_q}\varphi_\xi^\mu(v\eta v^{-1})dv = \,J_k^B(\xi, i\eta), \quad \xi \in \Xi_q\]
where $J_k^B$ is the Dunkl-type Bessel function associated with $R=B_q$ and the  multiplicity $k$ is given by $ k=(k_1,k_2)=\bigl(\mu-\frac{d}{2}(q-1)-\frac{1}{2}, \frac{d}{2}\bigr).$ Here  $k_1$ and $k_2$ are the parameters on the roots $\pm e_i$ and $\pm e_i\pm e_j$, respectively.
 \item[\rm{(4)}] Under the identification of $Y_\mu$ with its dual via $\,\xi \mapsto \psi_\xi^\mu,$ the Plancherel measure of $Y_{\mu}$ coincides with the Haar measure $\widetilde\omega_\mu$.
\end{enumerate}
\end{theorem}

In particular,
 the
Bessel function $J_k^B$ in this case satisfies the positive product formula
\[ J_k^B(\xi,z) J_k^B(\eta,z) \,=\, \int_{\Xi_q} J_k^B(\zeta, z) \,d(\delta_\xi \circ_\mu\delta_\eta)(\zeta) \quad \forall \, \xi,\eta\in \Xi_q,
\,z\in \mathbb C^q.\]

\begin{proof}(Sketch) Everything apart from (3) is done by hypergroup methods as indicated above.
Let us therefore only give some detail on the proof of part (3). A decisive ingredient is the
fact that the spherical polynomials $Z_\lambda$  satisfy the product formula
\[ \frac{Z_\lambda(r)Z_\lambda(s)}{Z_\lambda(I)}  = \int_{U_q} Z_\lambda(\sqrt{r}usu^{-1}\sqrt{r})du \quad \forall \,r,s\in \Pi_q\,\]
see \cite{FK}, Cor. XI.3.2. Define a hypergeometric function of two matrix arguments,
\[ _0F_1(\mu;x,y ) := \sum_{\lambda\geq 0} \frac{1}{(\mu)_\lambda\,|\lambda|!} \frac{Z_\lambda(x)Z_\lambda(y)}{Z_\lambda(I)}. \]
Then by the definition of the Bessel function $J_\mu$ and the above product formula,
\begin{equation}\label{hypertwo}
 \Psi_\xi^\mu(\eta) = \, _0F_1\bigl(\mu; \frac{\xi^2}{2}, \frac{(i\eta)^2}{2}\bigr),\end{equation}
where again $\xi$ and $\eta$ are identified with the corresponding diagonal matrices.
It is further known that  $Z_\lambda$ is given in terms of (suitably normalized)  Jack Polynomial $C_\lambda^\alpha$ of index $\alpha=2/d$, namely
$\, Z_\lambda(r) = C_\lambda ^{2/d}(\sigma(r))$ and in particular $ Z_\lambda(\xi) = C_\lambda ^{2/d}(\xi)$.
By a result of \cite{BF1} (see also \cite{R8}), the right side of formula \eqref{hypertwo} can then be identified with the
Bessel function $\, J_k^B(\xi,i\eta)$ with multiplicity $k$ as stated.
\end{proof}

In the geometric cases
 $\mu = pd/2$,  the
convolution of the hypergroup $Y_{\mu}$ coincides with the convolution of the Gelfand pair $(U_p\times U_q)\ltimes M_{p,q}/(U_p\times U_q)$, and the
support of the probability measure
$\delta_\xi \circ_\mu\delta_\eta$  on $\Xi_q$
describes the set of possible singular spectra of sums $x+y$ with matrices $x,y\in M_{p,q}$ having fixed singular spectra $\xi$ and $\eta$.

\vfill\newpage

\section{Markov processes}\label{S3}

The main object of this chapter are Markov processes on $\b R^N$ as well as on closed
Weyl chambers $\overline C\subset \b R^N$ whose
transition probabilities are related with Dunkl theory. The most prominent
examples are the so-called Dunkl processes on $\b R^N$ whose transition
semigroups are  generated by
the Dunkl Laplacians $\Delta_k$ as well as their symmetrized counterparts on $\overline C$ which
are diffusions and often called  Dunkl-Bessel processes.

Our concept of general Markov processes on $\b R^N$  related with Dunkl theory
 is motivated by  random walks
on groups (which can be studied via group representations, characters  and the
Fourier transform) and, more generally, on
  hypergroups. These topics are
connected in several respects.

First,  under suitable symmetry assumptions and up to suitable projections, the same
processes appear as random walks on groups, as random walks on hypergroups, as
well as as processes
on Weyl chambers  $\overline C$ on different levels after taking projections. These
relations on different levels are caused  by the underlying algebraic relations  between the
underlying state spaces explained in the preceding Chapter.

Furthermore, the very concept of  Markov processes on $\b R^N$ related with Dunkl
theory as well as the concept of random walks on hypergroups  may be regarded
within a common frame, namely Markov processes on a state space which admits
an integral transform. In Dunkl theory, this transform will be the Dunkl
transform, and in case of commutative hypergroups,  the hypergroup
Fourier transform.
This concept of an integral transform  allows somehow a common diagonalization of
all associated transition operators. We  develop this concept in Section
\ref{section-general-setup} and show in the further sections how this concept
leads to interesting  martingales and martingale
characterizations. The main emphasis will be on  Markov processes on
$\b R^N$ related with Dunkl theory and Dunkl processes.

\subsection{Random walks on groups and hypergroups}

In this section we briefly recall the concepts of random walks on groups and
hypergroups.

\begin{definition}
Let $G$ be a locally compact group with identity $e$.

Let $(Y_n)_{n\ge1}$ be a sequence of $G$-valued independent random variables
with laws $\mu_n\in M^1(G)$. Then we form the right random walk
$(X_n:=Y_1\ldots Y_n)_{n\ge1}$ (with the convention $X_0=e$).  This process is
a Markov process on $G$ starting in $e$ with transition probabilities
$$P(X_{n+1}\in A|\> X_n=x)= (\delta_x*\mu_{n+1})(A) \quad(n\ge 0,\> x\in G, \>
A\in \scr B(G)).$$

In continuous time, we proceed as follows: A  $G$-valued process $(X_t)_{t\ge0}$ is
called a random walk in continuous time (or a process with right-independent
increments), if for all $n\in\b N$, $0=t_0<t_1<\ldots<t_n$, the random
variables
$X_{t_0}, X_{t_0}^{-1}X_{t_1}, \ldots, X_{t_{n-1}}^{-1}X_{t_n}$ are
independent. Denoting the laws of
 $X_{s}^{-1}X_{t}$ by $\mu_{s,t}$ for $0\le s\le t$, we obtain that
 $(X_t)_{t\ge0}$
is a Markov process on $G$ with initial distribution $P_{X_0}$ and
transition probabilities
$$P(X_{t}\in A|\> X_s=x)= (\delta_x*\mu_{s,t})(A) \quad(0\le s\le t,\> x\in G, \>
A\in \scr B(G)).$$
\end{definition}

 This characterization of random walks on groups in discrete or continuous
 time as Markov process with translation invariant transition kernels motivates:

\begin{definition}
Let $(X,*)$ be a (always second countable) hypergroup, and let  $I=\b Z_+$ or
$I=[0,\infty)$.
A $X$-valued Markov process $(X_t)_{t\in I}$ is called a random walk on $X$,
if for all $s\le t\in I$,$x\in X$, and $A\in \scr B(X)$,
$$P(X_{t}\in A|\> X_s=x)= (\delta_x*\mu_{s,t})(A)$$
for a suitable family $(\mu_{s,t})_{s\le t\in I}$ of probability measures on $X$.
\end{definition}

It can be easily checked that the $\mu_{s,t}$ form a so-called hemigroup,
i.e., for $s\le t\le u\in I$, we have
$\mu_{s,t}*\mu_{t,u}=\mu_{s,u}$. Moreover, for $t\in I$, the law of $X_t$ is
given by $P_{X_0}*\mu_{0,t}$ as in the group case.
Moreover, standard arguments on the construction of Markov processes ensure
that
for a given initial law and a given hemigroup  $(\mu_{s,t})_{s\le t\in
  I}\subset M^1(X)$ there always exists an associated random walk on a suitable
probability space.

We next turn to relations between random
walks on different groups or hypergroups.

\begin{example}
Let $(G,K)$ be a Gelfand pair and $(X_t)_{t\in I}$ a random walk on $G$
starting in $e$ such that all transition probabilities $\mu_{s,t}$ are
$K$-biinvariant, i.e., contained in $M^1(G||K)$. We then say that $(X_t)_{t\in
  I}$
is a $K$-biinvariant random walk on $G$.
It is well-known that  $K$-biinvariant random walks on $G$ may be analyzed via
spherical Fourier transform. In particular, central limit theorems are often
derived in this way similar to the classical approach on the group $\b R^N$.

We may regard this approach also as follows: Consider the double coset
hypergroup $(X=G//K,*)$ and the canonical projection $p:G\to G//K$. It can be
easily checked (see below) that the process $(p(X_t))_{t\in I}$  then is a random
walk on the commutative hypergroup $(X=G//K,*)$. In this way,  an
investigation of this random walk on $(X=G//K,*)$ via the hypergroup
Fourier transform  corresponds precisely to the study of $K$-biinvariant
random walks on $G$ via the
spherical Fourier transform.
\end{example}

\begin{lemma}
In the preceding setting,  $(p(X_t))_{t\in I}$  is a  random
walk on the commutative hypergroup $(X=G//K,*)$ associated with the hemigroup\\
$(p(\mu_{s,t}))_{s\le t\in I}$.
\end{lemma}

\begin{proof} Let $(\cal F)_{t\in I}$ be the canonical filtration of the Markov process
  $(X_t)_{t\in I}$ on $G$ defined on some probability space
  $(\Omega,\cal A, P)$, and let  $(\tilde{\cal F})_{t\in I}$ be the canonical filtration
  of  $(p(X_t))_{t\in I}$.
Fix $s\le t\in I$ and $A\in\scr B(X)$. We first note that the function
$x\longmapsto (\delta_x*\mu_{s,t})(p^{-1}(A))$ on $G$ is $K$-biinvariant.
This implies
$$(\delta_x*\mu_{s,t})(p^{-1}(A))=(\delta_{p(x)}* p(\mu_{s,t}))(A)
\quad\text{for all}\quad x\in G.$$
Therefore, by the Markov property of $(X_t)_{t\in I}$,
\begin{align}
P(p(X_t)\in A|\> \cal F_s)&= P(X_t\in p^{-1}(A)|\> X_s)=
(\delta_{X_s}*\mu_{s,t})(p^{-1}(A)) \notag\\
&=(\delta_{p(X_s)}* p(\mu_{s,t}))(A)  \notag
\end{align}
a.s.. This implies that
$$P(p(X_t)\in A|\> \tilde{\cal F}_s)=P(p(X_t)\in A|\> p(X_s))=(\delta_{p(X_s)}*
p(\mu_{s,t}))(A)$$
a.s.~as claimed.
\end{proof}

There exists the following variant of the preceding example:

\begin{example}\label{random-orbits}
Let $V$ be a locally compact abelian group on which a compact group $K$ of
automorphisms acts continuously. Let $(X_t)_{t\in I}$ be a random walk on $V$
starting in the identity $e$ such that its transition probabilities are
$K$-invariant, i.e., for all
$0\le s\le t$ and $k\in K$, $k(\mu_{s,t})=\mu_{s,t}$.  We then say that
$(X_t)_{t\in I}$ is $K$-invariant.

Consider the commutative orbit hypergroup $(V^K,*)$ and the associated canonical
projection $p:V\to V^K$. In the same way as in the preceding lemma it may be
checked that then  $(p(X_t))_{t\in I}$  is a  random
walk on  $(V^K,*)$ starting in the identity with transition probabilities
$(p(\mu_{s,t}))_{s\le t\in I}\subset M^1(V^K)$.
\end{example}

\begin{examples}
\begin{enumerate}
\item[\rm{(1)}] Let $V=\b R^N$ and $K=O(N)$. Then we may identify $V^K\simeq [0,\infty)$, and
  the radial parts of $O(N)$-invariant random walks on $\b R^N$ are random
  walks on the Bessel-Kingman hypergroup of index $\alpha=N/2-1$. In
  particular, as the radial parts of $N$-dimensional Brownian motions are
  Bessel processes with this dimension parameter, these Bessel processes may
  be regarded as random walks on the corresponding Bessel-Kingman hypergroups.
\item[\rm{(2)}] Let $V=M_{p,q}(\b F)$ with $p,q\in\b N$ and $\b F= \b R,\b
  C,\b H$. Let $K:=U_p(\b F)$ act on $V$ from the left. Then, by  Section
  \ref{section-matrix-cones}, $V^K\simeq \Pi_q$, where the canonical projection $p:V\to V^K$ is given by
  $p(x)=(x^*x)^{1/2}$. Therefore, if $(X_t)_{t\in I}$ is an $K$-invariant
  random walk on $M_{p,q}(\b F)$, then $((X_t^*X_t)^{1/2})_{t\in I}$ is a
  random walk on the matrix Bessel hypergroup $(M_{p,q}(\b F),*)$ with index
  $\mu=pd/2$.
In particular, as the ``$U_p$-radial'' parts of $N$-dimensional Brownian
  motions on $V$ are Wishart processes with a shape parameter corresponding
  to $p$, we see that Wishart processes may be  regarded as random walks on
  these Bessel hypergroups on matrix cones.
 The reader
  should be careful with a different norming of these processes in the
  literature; see  \cite{B} and \cite{V2}.
\item[\rm{(3)}] Let $V=H_n(\b F)$ the space of all Hermitian matrices on which
  $K=U_n(\b F)$ acts by conjugation. Then $V^K$ may be identified with the Weyl
  chamber $\overline C_n$ of type $A_{n-1}$ by taking the ordered
  spectrum. If  $(X_t)_{t\in I}$ is an $U_n(\b F)$-conjugation invariant
  random walk on $H_n(\b F)$, then its spectral part forms a random walk on
  the associated orbit hypergroup $(\overline C_n,*)$ which is associated with
  Dunkl theory of type  $A_{n-1}$ with multiplicity $k=d/2$ with $d=dim_{\b
  R}\b F$.

Similar interpretations exist for the root systems $B_n, C_n,D_n$.
\end{enumerate}
\end{examples}

The concept of Example \ref{random-orbits} may be
easily transfered to the case where $V$ is a commutative hypergroup:
Let $(X,*)$ be a commutative hypergroup on which a compact group of hypergroup
automorphisms acts continuously.  Let $(X_t)_{t\in I}$ be a random walk on $(X,*)$
starting in the identity $e$ such that its transition probabilities are
$K$-invariant, i.e., for all
$0\le s\le t$ and $k\in K$, $k(\mu_{s,t})=\mu_{s,t}$.  We then say that
$(X_t)_{t\in I}$ is $K$-invariant.
Consider the commutative orbit hypergroup $(X^K,*)$ and the associated canonical
projection $p:X\to X^K$. As above,   $(p(X_t))_{t\in I}$  then is a  random
walk on  $(V^K,*)$ starting in the identity with transition probabilities
$(p(\mu_{s,t}))_{s\le t\in I}\subset M^1(V^K)$.

\begin{example}
Consider the matrix Bessel hypergroup $(\Pi_q(\b F),*)$ of some admissible
index $\mu$ as in Section  \ref{section-matrix-cones}. As explained there,
 the group $U_q(\b F)$
acts on $\Pi_q(\b F)$ by conjugation as a group of hypergroup automorphisms,
and the associated orbit space can be identified with the Weyl chamber $\overline
C_q:=\Xi_q$ of type $B_q$. Therefore, the spectrum of $U_q(\b F)$-conjugation invariant
random walks on the Bessel hypergroup $(\Pi_q(\b F),*)$ are random walks on
the associated orbit hypergroup structures on $\overline C_q$ which belong, by
Section  \ref{section-matrix-cones}, to Dunkl theory with multiplicities
$k_1=\mu-d(q-1)/2-1/2$ and $k_2=d/2$ with $d=dim_{\b R}\b F$.

In summary, projections of sufficiently symmetric classical random walks on the
Euclidean space $M_{p,q}(\b F)$ appear as random walks on different levels:
As random walks on the matrix Bessel hypergroup $(\Pi_q(\b F),*)$ of index $\mu=pd/2$, as
 random walks on the commutative hypergroup on  the Weyl chamber $\overline
C_q$ of type $B_q$ which is associated with Dunkl theory with the
multiplicities above, and finally  as random walks on the one-dimensional
Bessel-Kingman hypergroup $[0,\infty)$ of index $\alpha=dpq/2-1$.
\end{example}

\subsection{Markov processes related with integral transforms}\label{section-general-setup}

In this section we introduce a concept which allows to study many features
of random walks on commutative hypergroups as well as of Markov processes on
$\b R^N$, which are related with  the Dunkl transform, within of a common
frame. The concept is that of processes related with an abstract integral
transform on the underlying state space.
This transform will
be either the Dunkl transform or the Fourier transform on a commutative
hypergroup.  Further integral transforms related with
suitable families of special functions  are also possible.

The following definition is suitable for our purposes:

\begin{definition}\label{absintegraltrafo}
 Let $X$ be a second countable, locally compact space, and
  $(\phi_\lambda)_{\lambda\in\hat X}$ a  family of
  functions in $C_b(X)$ labeled by some further locally compact space $\hat
  X$ with $|\phi_\lambda(x)|\le 1$ for $x\in X,\lambda\in \hat X$.
The triple\\ $T=(X,\hat X, (\phi_\lambda)_{\lambda\in\hat X})$ will be called an
  abstract integral transform, if the following  holds:
\begin{enumerate}
\item[\rm{(1)}] The mapping $\lambda\mapsto \phi_\lambda$ is continuous
  w.r.t.  compact-uniform topology on $C_b(X)$.
\item[\rm{(2)}] There exists $e\in X$ with $\phi_\lambda(e)=1$ for all
  $\lambda\in \hat X$, and $\hat e\in\hat X$ with  $\phi_{\hat e}(x)=1$ for all
  $x\in X$.
\item[\rm{(3)}] Riemann-Lebesgue Lemma: There exists $\pi\in M^+(\hat X)$ such
  that for all $f\in L^1(\hat X,\pi)$, the function $\check f(x):=\int_{\hat
  X} f(\lambda)\phi_\lambda(x)\> d\pi(\lambda)$ satisfies $\check f\in
  C_0(X)$, and $\{\check f:\> f\in L^1(\hat X,\pi)\}$ is $\|.\|_\infty$-dense
  in  $C_0(X)$.
\end{enumerate}
\end{definition}

\begin{examples}
\begin{enumerate}
\item[\rm{(1)}] Let $X=\hat X=\b R^N$, $e=\hat e=0$,  $d\pi(x)=w_k(x)dx$, and
  $\phi_\lambda(x):=E_k(-ix,\lambda)$
the Dunkl kernel for some multiplicity $k\ge0$. Then all axioms above are
  satisfied by Section \ref{S:Dunklkernel}.
\item[\rm{(2)}] Let $X$ be a (from now on always second countable) commutative hypergroup with neutral element $e$,
   dual space $\hat X$, and Plancherel measure $\pi$. Taking
  $\phi_\lambda(x):=\overline{\lambda(x)}$
for $\lambda\in\hat X$, $x\in X$ and $\hat e\equiv 1$, we  obtain all axioms
   above from Section  \ref{Section-Gelfand-pairs}.
\item[\rm{(3)}] By interchanging the roles of the commutative  hypergroup $X$
  and its dual $\hat X$ and taking the Haar measure instead of $\pi$, one also
  obtains also all axioms above from Section  \ref{Section-Gelfand-pairs}.
\end{enumerate}
\end{examples}

We collect a few obvious properties of abstract integral transforms:

\begin{proposition}\label{properties-absintegral-trafo}
\begin{enumerate}
\item[\rm{(1)}] The integral transform $M_b(X)\to C_b(\hat X)$, $\mu \mapsto \hat\mu$
 with $\hat\mu(\lambda)=\int_X \phi_\lambda(x)\> d\mu(x)$ is injective.
\item[\rm{(2)}] For $\mu\in M_b(X)$ and $f\in L^1(\hat X,\pi)$,
$\int_X \check f\> d\mu =\int_{\hat X} f\hat\mu\> d\pi$.
\item[\rm{(3)}] The $\phi_\lambda$ separate points of $X$, and $e$ is unique.
\end{enumerate}
\end{proposition}

\begin{proof} (2) is clear and leads together with   \ref{absintegraltrafo}(1) and
 \ref{absintegraltrafo}(3)
 to (1). Part (3) is then clear.
\end{proof}

Up to a further technical restriction, the  axioms are also strong enough in order to yield L\'evy's continuity
  theorem in a strong version:

\begin{theorem}\label{levy-cont}
Let $(\mu_n)_{n\in{\b N}}\subset M_b^+(X)$.
\begin{enumerate}
\item[\rm{(1)}] If $(\mu_n)_{n\in{\b N}}$ converges weakly to $\mu\in M_b^+(X)$,
then $(\hat \mu_n)_{n\in{\b N}}$ converges to $\hat\mu$ pointwise.
\item[\rm{(2)}] If  $(\hat \mu_n)_{n\in{\b N}}$ tends pointwise
 to a $\b C$-valued function $\phi$ on $\hat X$
continuous at ${\hat e}$ and if $\hat e\in supp\>\pi$ holds, then there is a unique $\mu\in M_b^+(X)$
 with $\hat\mu=\phi$, and
 $(\mu_n)_{n\in{\b N}}$ tends  weakly to $\mu$.
\end{enumerate}
\end{theorem}

Notice that the additional condition in (2) holds in the Dunkl setting and
precisely for the
commutative hypergroups where the identity character is contained in
$supp\>\pi$. This is  the case for all commutative hypergroups
with subexponential growth and in particular for all double coset hypergroups
of Euclidean type in Chapter \ref{Section-Gelfand-pairs} as well as for all
Bessel hypergroups on matrix cones.

\begin{proof} Part (1) is obvious. For the proof of (2) we use the well-known
 approach of Siebert which works for commutative groups and hypergroups; see
 e.g.~Section 4.2 of \cite{BH}: By the assumptions, $\lim_n \|\mu_n\|=
 \lim_n\hat\mu_n(\hat e) =\phi(\hat e)$. Therefore, by a compactness argument,
 there is a $\sigma(M_b(X), C_0(X))$-convergent subsequence $(\mu_{n_k})_k$
 with some limit $\mu\in M_b^+(X)$. Therefore, by Proposition \ref{properties-absintegral-trafo}(2), we
 obtain for all $g\in  L^1(\hat X,\pi)$ that $\lim_k \int_{\hat X} g \hat
 \mu_{n_k}\> d\pi = \int_{\hat X} g \hat \mu \> d\pi $. Moreover, as $\|\hat
 \mu_n\|_\infty$ remains bounded, this limit is also equal to $\int_{\hat X}
 g\phi \> d\pi$. As $\phi$ and $\hat\mu$ are continuous at $\hat e\in
 supp \> \pi$, we conclude that $\phi(\hat e)=\hat\mu (\hat e)$. This shows
 that $\lim_k\|\mu_{n_k}\|=\|\mu\|$, i.e., $(\mu_{n_k})_k$ tends weakly to
 $\mu$.
Therefore, $\hat \mu=\phi$, i.e., $\mu$ is determined uniquely and
 independent of the subsequence. As the closed unit ball in $M_b^+(X)$ is
 compact and metrizable w.r.t.~$\sigma(M_b^+(X), C_0(X))$, we readily obtain
that  $(\mu_{n})_n$ converges to $\mu$.
\end{proof}

We notice that for commutative hypergroups $X$  the following weaker
continuity result holds (see Section 4.2 of \cite{BH}):
If for measures $\mu_n,\mu\in M_b^+(X)$, $\hat \mu_n$ tends pointwise to $\hat
\mu$ on $\hat X$, then $\mu_n$ tends weakly to  $\mu$.

In any case, all versions of L\'evy's continuity theorem are strong enough for
application in probability, e.g., in order to derive CLTs. We do not go into
details  and refer to  \cite{BH} for applications to
random walks on hypergroups.

We next turn to Markov kernels associated with abstract integral
transforms. For this we fix  an
  abstract integral transform $T=(X,\hat X, (\phi_\lambda)_{\lambda\in\hat
    X})$.

\begin{definition}
A Markov kernel $P: X\times {\cal B}(X)\to [0,1]$ is called
related with $T$ if
\begin{equation}\label{Markov-rel-I}
P(x,.)^\wedge(\lambda)=P(e,.)^\wedge(\lambda)\cdot \phi_\lambda(x) \quad\quad {\rm
 for\>\> all\>\>} x\in X, \lambda\in\hat X.
\end{equation}
\end{definition}

\begin{examples}
\begin{enumerate}
\item[\rm{(1)}] Let $X$ be a  commutative hypergroup  with dual  $\hat X$. Let
  $\mu\in M^1(X)$. It can be easily checked  that $P_\mu(x,A):=\mu*\delta_x(A)$
  establishes a ``translationinvariant'' Markov kernel on $X$. This kernel is
  obviously related with the hypergroup Fourier transform.
\item[\rm{(2)}] For $X=\hat X=\b R^N$ and the Dunkl transform, condition
  (\ref{Markov-rel-I}) can be translated easily. We then obtain the notion of a kernel
  $P$  related  with the Dunkl transform.  For multiplicity $k=0$, i.e., the
  group case, this notion agrees by the injectivity of the classical Fourier
  transform
with the usual notion of translation
  invariant kernels on  $\b R^N$. We  notice that it is an open problem,
 for which probability measures $\mu\in M^1(\b R^N)$ precisely there exists a
  Markov kernel $P$ on  $\b R^N$ with $\mu=P(0,.)$.
On the other hand there exist many examples like all pseudo-radial probability
  measures on $\b R^N$ in the sense of  Section \ref{generalized translation},
  which in particular includes the Dunkl
  Gaussians.
\end{enumerate}
\end{examples}

We next collect  basic properties of such   Markov kernels.

\begin{lemma}\label{prop-kernels} Let $P$ and $Q$ be   Markov kernels related
  with the abstract integral transform as above.
\begin{enumerate}
\item[\rm{(1)}] $Pf(x):=\int_{X} f(y)\> P(x,dy)$ defines a
 bounded linear operator on $C_0(X)$.
\item[\rm{(2)}] The composition $P\circ Q$ with
 $P\circ Q(x,A)=\int_{X} Q(z,A)\> P(x,dz)$ is a
 Markov kernel on $X$ related with $T$, and
\begin{equation}
 ((P\circ Q)(x,.))^\wedge(\lambda)=  Q(e,.)^\wedge(\lambda)\cdot  P(e,.)^\wedge(\lambda)\cdot
 \phi_\lambda(x)\>\>{ for}\>\> x\in X, \lambda\in\hat X.
\end{equation}
In particular, $P\circ Q= Q\circ P$.
\end{enumerate}
\end{lemma}

\begin{proof}
For (1), it suffices to check  $Pf\in C_0(X)$ for
 $f\in C_0(X)$. As $\{\check g:\> g\in L^1(\hat X,\pi)\}$
is  $\|.\|_\infty$-dense  in  $ C_0(X)$, it suffices by an $\epsilon$-estimate, to
 do this for  $f=\check g$, $g\in L^1(\hat X,\pi)$. In this case  we obtain
 from Proposition \ref{properties-absintegral-trafo}(2)
\begin{align}
 Pf(x) &= \int_{X} \check g(y) \> P(x,dy)
 =\int_{\hat X} g(\lambda) \cdot P(x,.)^\wedge(\lambda) \> d\pi(\lambda) \notag\\
       &=\int_{\hat X} g(y)\cdot \phi_\lambda(x)  P(e,.)^\wedge(\lambda)\> d\pi(\lambda)
       = (g\cdot P(e,.)^\wedge)^\vee(x)
\notag  \end{align}
which is a function in  $ C_0(X)$ by our axioms. This implies
 (1). (2) is easy to check; compare e.g.~with Lemma 4.2 of \cite{RV1}.
\end{proof}

We now turn to families of  Markov kernels on $X$ and associated Markov
processes related to an abstract integral transform $T$. For this fix
a time area $I= [0,\infty)$ or $I=\b Z_+$, and put $S:=\{(s,t): s,t\in I,\> s\le t\}$.

\begin{definition}\label{hemigroup}
A  family $(P_{s,t})_{s\le t\in I}$ of  Markov
 kernels on $X$ related to $T$ is called a (continuous) hemigroup
of  Markov kernels related with $T$, if
\begin{enumerate}
\item[\rm{(1)}] for all $s\le t\le u\in I$,  $P_{s,t}\circ P_{t,u}= P_{s,u}$,
\item[\rm{(2)}] $P_{t,t}$ is the trivial kernel for  $t\in I$, and
\item[\rm{(3)}] the mapping $S\to M^1(X)$,
 $(s,t)\mapsto P_{s,t}(0,.)$, is weakly continuous.
\end{enumerate}
In particular, if in addition $P_{s,t}=P_{s+u,t+u}$ for all $(s,t)\in S$ and
$u\in I$, then we may put $P_t:=P_{0,t}$ for $t\in I$ and obtain a
continuous semigroup $(P_t)_{t\in I}$
of  Markov kernels related with $T$.

Assume that a hemigroup $(P_{s,t})_{s\le t\in I}$ of  Markov
 kernels on $X$ related to $T$ as well as a starting probability $\mu\in
 M^1(X)$ are given.
Then we can construct in a canonical way an associated Markov process on
 $X$. Such processes will be called Markov processes on $X$ related with the
 integral transform $T$. In the case of a semigroup we obtain
 time-homogeneous Markov processes on $X$ related with $T$; they are analogs of
  L\'evy processes.
\end{definition}

\begin{examples}
\begin{enumerate}
\item[\rm{(1)}] Let $(\mu_t)_{t\ge 0}
  \subset M^1(X)$ be a convolution semigroup on a commutative hypergroup $X$,
  i.e., $\mu_s*\mu_t=\mu_{s+t}$
 for  $s,t\ge 0$ with $\mu_0=\delta_e$, and   $[0,\infty)\to M^1(X)$, $t\mapsto \mu_t$
 is weakly continuous. Then the kernels $P_t(x,A):=\mu_t*\delta_x(A)$ ($x\in
  X$, $A\in\cal B(X)$) form a semigroup of Markov kernels related to
  the hypergroup Fourier transform.
\item[\rm{(2)}]  For any root system and any multiplicity $k\ge0$,
the Dunkl heat kernels $P_t(x,dy)=  \Gamma_k(t,x,y) w_k(y)dy $ on $\b R^N$ of
Chapter \ref{heat}
with
\[ \Gamma_k(t,x,y):=\, \frac{1}{(2t)^{\gamma +N/2}c_k}\,e^{-(|x|^2 + |y|^2)/4t}\,
E_k\Bigl(\frac{x}{\sqrt{2t}},\frac{y}{\sqrt{2t}}\Bigr),\quad x,y\in \b R^N,\]
and $t>0$ form (together with the trivial kernel $P_0$) a  semigroup of Markov kernels related to
  the Dunkl transform.
\end{enumerate}
\end{examples}

\begin{remark}
If $\hat X$ is connected (which is for instance the case in Dunkl theory),
then condition (2) of Definition \ref{hemigroup} holds automatically. In fact,
$P_{t,t}=P_{t,t}\circ P_{t,t}$ and Lemma \ref{prop-kernels}(2) ensure that
$P_{t,t}(e,.)^\wedge$
is a $\{0,1\}$-valued function with $P_{t,t}(e,.)^\wedge(\hat e)=1$, and hence
$P_{t,t}(e,.)^\wedge\equiv 1$ which implies together with the injectivity of
the integral transform the assertion.
\end{remark}

\begin{lemma}\label{Fourier-nicht-null}
Let  $(P_{s,t})_{s\le t\in [0,\infty)}$  be a hemigroup of  Markov
 kernels on $X$  related with $T$. Then for all $(s,t)\in S$ and
 $\lambda\in\hat X$, $P_{s,t}(e,.)^\wedge(\lambda)\ne 0$.
\end{lemma}

\begin{proof}
For $\lambda\in\hat X$, the mapping $(s,t)\mapsto
P_{s,t}(e,.)^\wedge(\lambda)\ne 0$ is continuous with
$P_{t,t}(e,.)^\wedge(\lambda)=1$ and $P_{s,t}(e,.)^\wedge(\lambda)\cdot P_{t,u}(e,.)^\wedge(\lambda)
=P_{s,u}(e,.)^\wedge(\lambda)$. Therefore, $P_{s,t}(e,.)^\wedge(\lambda)= 0$
immediately would lead to a contradiction.
\end{proof}

We now restrict our attention to the time-homogeneous case.

\begin{remark}\label{neg-definite}
The arguments of the proof of \ref{Fourier-nicht-null} imply for a semigroup
 $(P_t)_{t\ge0}$ of Markov kernels on $X$ that there
 exists  a function $\psi:\hat X\to\b C$ with
 $P_t(e,.)^\wedge(\lambda)=e^{-t\psi(\lambda)}$ for $t\ge 0$, $\lambda\in
 \hat X$.
The function $\psi$ satisfies ${\rm Re}\> \psi\ge0$ and
\begin{equation}
\psi(\lambda)=\lim_{t\downarrow 0} \frac{1}{t} (1-P_t(0,.)^\wedge(\lambda))\quad\quad
 (\lambda\in\hat X).
\end{equation}
Moreover,  $\psi$ is continuous because of
$$\Bigl( \int_0^\infty e^{-t}\> P_t(e,.)\> dt\Bigr)^\wedge(\lambda) =  \int_0^\infty  e^{-t}\>e^{-t\psi(\lambda)}\> dt
= (1+\psi(\lambda))^{-1}.$$
 $\psi$ is called
 the negative definite function associated with
$(P_t)_{t\ge0}$.
\end{remark}

\begin{proposition}\label{feller-semigroup}
 Each semigroup $(P_t)_{t\ge0}$  of
 kernels on $X$ related to $T$ is a Feller semigroup, i.e., for $f\in
 C_0(X)$ and $t\ge0$,  $P_tf\in C_0(X)$, and
\begin{equation}\label{Feller-limes}
\lim_{t\to 0} \|P_tf-f\|_\infty =0\quad\text{ for }\quad f\in C_0(X).
\end{equation}
\end{proposition}

\begin{proof} It suffices to check (\ref{Feller-limes}).
Taking $f=\check g$, $g\in L^1(\hat X,\pi)$, we have
\begin{align}
|P_tf(x)-f(x)|&= \Bigl|\int_{\hat X} g(\lambda)\Bigl( P_t(e,0)^\wedge(\lambda)
-1\Bigr) \psi_\lambda(x) \> d\pi(\lambda)\Bigr|
\notag\\
 &\le
\int_{\hat X} |g(\lambda)|\> |P_t(e,0)^\wedge(\lambda)-1| \> d\pi(\lambda)
\notag
\end{align}
which tends for $t\to0$ to $0$ independent of $x$. In other words,
(\ref{Feller-limes}) holds for  $f=\check g$, $g\in L^1(\hat X,\pi)$.
As the space of these functions is $\|.\|_\infty$-dense in $C_0(X)$, and the
operators $P_t$ on $C_0(X)$ satisfy $\|P_t\|\le1$, (\ref{Feller-limes})
follows by an $\epsilon$-argument for $f\in C_0(X)$.
\end{proof}

Proposition \ref{feller-semigroup} together with the theorem of
Dynkin-Kinney-Blumenthal (see, e.g., \cite{Dy})  imply:

\begin{corollary}
Each time-homogeneous Markov process   related with an abstract integral
transform admits a c\`adl\`ag modification, i.e. a modification  with  right continuous
paths and left limits.
\end{corollary}

Let  $(P_t)_{t\ge0}$  be a semigroup of   Markov
 kernels on $X$ related to some abstract integral transform $T$.
As this semigroup forms a positive contraction semigroup on $C_0(X)$, we
 introduce the generator
$$Lf:=\lim_{t\to 0} \frac{1}{t} (P_tf-f)$$
which is a closed operator with a $\|.\|_\infty$-dense domain in $C_0(X)$ by
the Hille-Yoshida theory.
 We also define the
following extended domains:
$$ D(L):=\{ f\in C(X):\>\>   \frac{1}{t}(P_t f-f) \>\>\>
 {\rm converges\>\> uniformly\>\> on}\>\>X{\rm\>\> for}\>\> t\to 0\},$$
and
$$D_b(L):= D(L)\cap C_b(X),\quad D_0(L):= D(L)\cap C_0(X).$$
 $D_0(L)$ is the domain of $L$ on $C_0(X)$, and
 $Lf\in C_0(X)$ for $ f\in D_0(L)$.

Moreover, for $t>0$, $x\in X$ and $\lambda\in\hat X$,
$$P_t \phi_\lambda(x)= P_t(x,.)^\wedge(\lambda)=
 P_t(e,.)^\wedge(\lambda)\phi_\lambda(x)=P_t \phi_\lambda(e) \cdot\phi_\lambda(x)$$
 and hence
$$\lim_{t\to 0} \frac{1}{t} (P_t\phi_\lambda(x) -\phi_\lambda(x))
= -\psi(\lambda) \phi_\lambda(x)$$
uniformly with the negative definite function $\psi$ of Remark
\ref{neg-definite}.
In particular, we have  $\{\phi_\lambda:\> \lambda\in\hat X\}\subset D_b(L)$ with
$L\phi_\lambda= -\psi\cdot \phi_\lambda.$

We next introduce  a notion of Gaussian semigroups  analog to locally compact
groups, which also works for commutative  hypergroups.

\begin{definition}\label{gauss-kernel-def}
Let $(P_t)_{t\ge0}$ be a semigroup of Markov kernels on $X$ related with some
abstract integral transform with the generator $L$ on $C_0(X)$.
\begin{enumerate}
\item[\rm{(1)}] $(P_t)_{t\ge0}$ is called  Gaussian if
$$\lim_{t\to0} \frac{1}{ t} P_t(e,X\setminus U)=0
\quad{\rm for\>\> all\>\> open \>\> subsets\>\>}
U\subset X\>\>{\rm with}\>\> e\in U.$$
\item[\rm{(2)}] $L$ is called of local type if
 $supp(Lf)\subset supp\> f$ for all $f\in D_c(L)$.
\end{enumerate}
\end{definition}

It is well-known from the theory of Feller processes (see e.g.~\cite{Dy}) that
 its generator$L$ is  of local type if and only if each associated Feller
 process admits an a.s.~continuous modification. Moreover,
the following characterization of Gaussian processes on commutative
hypergroups, which is  analog to the group case, can be found in
\cite{ReV}.

\begin{theorem}\label{gauss-characterization} Let $(\mu_t)_{t\ge0}\subset M^1(X)$ be a convolution
  semigroup on a commutative hypergroup $X$. Then the following  are equivalent:
\begin{enumerate}
\item[(1)] $(\mu_t)_{t\ge0}$ is Gaussian, i.e., the associated semigroup of
  kernels is Gaussian.
\item[(2)] $L$  is of local type.
\item[(3)] Each  L\'evy  process on $X$ associated with $(\mu_t)_{t\ge0}$
  admits an a.s. continuous modification.
\end{enumerate}
\end{theorem}

This  equivalence is not longer valid for time-homogeneous
Markov proces\-ses related with abstract integral transforms. In fact, for
multiplicities $k\ge 0$, the  heat kernels of Dunkl type satisfy
$$P_t(0, dy)= \frac{1}{(2t)^{\gamma+N/2}c_k} e^{-|y|^2/4t} w_k(y)\> dy,$$
i.e., like  classical Gaussian kernels they are Gaussian in the sense of
\ref{gauss-kernel-def}(1).
 On the other hand, the generator of this heat semigroup  is the Dunkl
 Laplacian, which is obviously of local type if and only if $k\equiv0$
 holds. Therefore, taking Theorem \ref{gauss-characterization} into account, we
 conclude:

\begin{corollary} For any root system and any multiplicity $k\ge 0$ with
  $k\not\equiv0$,
there  exists no commutative hypergroup structure on $\b R^N$ with the
Dunkl transform as hypergroup Fourier transform.
\end{corollary}

We notice that for $N=1$, explicit formulas for the Dunkl convolution are
known (\cite{R1} and \cite{Ros}; see Chapter \ref{generalized translation}), which immediately show that this convolution
is not positivity preserving.

\subsection{Martingales associated with  integral transforms}

Let $T=(X,\hat X, (\phi_\lambda)_{\lambda\in\hat X})$  be an abstract integral transform as introduced in the preceding
section. We here construct martingales related to Markov processes on $X$
related with $T$ in a  canonical way by using the functions
$\phi_\lambda$ which play the role of characters.

For the classical Brownian
motion on the  group $\b R^N$, the following  martingale characterization can be found in many
textbooks on stochastic analysis as a preparation to the L\'evy
characterization of Brownian motion.
Extensions of this
characterization to  L\'evy processes on locally compact groups or commutative hypergroups
can be found e.g.~in \cite{ReV}, \cite{V1}, \cite{HP}.
Moreover, a version for processes on $\b R^N$ related with the Dunkl transform,
can be found in \cite{RV1}.

\begin{theorem}\label{mart-characterization-1}
Let $I=[0,\infty)$ or $I=\b Z_+$. Let
 $(P_{s,t})_{s\le t\in I}$ a hemigroup of  Markov kernels on $X$ related to
 $T$. Then an arbitrary stochastic process
$(X_t)_{t\in I}$ on $X$ is a Markov process related with the hemigroup $(P_{s,t})_{s\le t\in I}$
if and only if
\begin{equation}\label{mult-mart}
\Bigl(\,\frac{1}{P_{0,t}(e,.)^\wedge(\lambda)} \cdot \phi_\lambda(X_t)\Bigr)_{t\geq 0}
\end{equation}
is a martingale  for each $\lambda\in\hat X$ w.r.t.~the canonical filtration
$({\cal F}_t)_{t\in I}$ of $(X_t)_{t\in I}$.
\end{theorem}

\begin{proof}
 Notice first that
 $P_{s,t}(x,.)^\wedge(\lambda)\ne 0$ for all
$t\ge s\ge 0$, $x\in X$, $\lambda\in \hat X$ by Lemma  \ref{Fourier-nicht-null}. This  ensures that the processes
 above are well-defined. Let $(\Omega,\cal A, P)$ be the probability space on
 which the process $(X_t)_{t\in I}$ is defined.

To check the only-if-part, take $s,t\in I$ and $\lambda\in\hat X$. Then
for a.e.~$\omega\in\Omega$,
\begin{align}
E(\phi_\lambda(X_{s+t})|{\cal F}_s)(\omega)&=
 E(\phi_\lambda(X_{s+t})|X_s)(\omega) =
 \int_{X} \phi_\lambda(x)\> P_{s,s+t}(X_s(\omega),dx)
\notag\\
&=\phi_\lambda(X_s(\omega)) \cdot P_{s,s+t}(e,.)^\wedge(\lambda).
\notag
\end{align}
Hence, as $ P_{0,s+t}(e,.)^\wedge= P_{0,s}(e,.)^\wedge\cdot P_{s,s+t}(e,.)^\wedge$,
 the process (\ref{mult-mart}) is a martingale.

To check the if-part, take again
 $s,t\in I$ and $\lambda\in\hat X$. Then, by our assumption and the preceding equation,
$$E(\phi_\lambda(X_{s+t})|{\cal F}_s) = \phi_\lambda(X_s) \cdot P_{s,s+t}(e,.)^\wedge(\lambda).
\quad{\rm a.s..}$$
Now take $F\in {\cal F}_s$ with $P(F)>0$. Define
 the probability measure
$P_F$ on $(\Omega,{\cal A})$ by $P_F(A):=\frac{ P(A\cap F)}{P(F)}$. The
 distributions $\mu_s^F, \mu_{s+t}^F\in M^1(X)$ of $X_s$
 and $X_{s+t}$ w.r.t.~$P_F$ satisfy
\begin{align}
(\mu_{s+t}^F)^\wedge (\lambda)&=\int_{X} \phi_\lambda(y) \> d\mu_{s+t}^F(y)
 =
 \frac{1}{P(F)} \int_F \phi_\lambda(X_{s+t}) \> dP\notag\\
&=  \frac{1}{P(F)} \int_F E( \phi_\lambda(X_{s+t})|{\cal F}_s) \> dP
\notag\\
&=
\frac{1}{P(F)} \int_F  P_{s,s+t}(e,.)^\wedge(\lambda)\cdot \phi_\lambda(X_{s})
 \>dP
\notag\\
&= P_{s,s+t}(e,.)^\wedge(\lambda)\cdot (\mu_{s}^F)^\wedge (\lambda)
=(P_{s,s+t}\circ \mu_s^F)^\wedge(\lambda).
\notag
\end{align}
As this holds for all $\lambda\in  \hat X$, the
 injectivity  of our integral transform
yields that $\mu_{s+t}^F=P_{s,s+t}\circ \mu_s^F$. Hence,
 for each Borel set $B\subset X$ and each $F\in{\cal F}_s$,
\begin{align}
 \int_F 1_{\{X_{s+t}\in B\}} \> dP &=
 P(\{X_{s+t}\in B\} \cap F) =P(F)\cdot \mu_{s+t}^F(B)
\notag\\
&=P(F)\cdot (P_{s,s+t}\circ \mu_s^F)(B)=
 \int_F P_{s,s+t}(X_s(\omega),B) \> dP(\omega).
\notag
\end{align}
As $\omega\mapsto P_t(X_s(\omega), B)$
 is $\sigma(X_s)$-measurable and
 ${\cal F}_s \supset \sigma(X_s)$, we obtain
 $$ P(X_{s+t}\in B| {\cal F}_s)= P(X_{s+t}\in B| X_s)=
 P_{s,s+t}(X_s,B)\quad \text{ a.s.}$$   for Borel
 sets $B\subset \b R^N$. Hence, $(X_t)_{t\in I}$ is a
 Markov process  associated with the hemigroup $(P_{s,t})$ as claimed.
 \end{proof}




We next rewrite Theorem \ref{mart-characterization-1} in the time-homogeneous
case, where we shall employ the negative definite function $\psi\in C(\hat X)$ of a
 semigroup $(P_t)_{t\ge0}$ of  kernels related to the integral transform $T$.

\begin{lemma}\label{mart-characterization-2}
 Let $(P_t)_{t\ge 0}$ be a  semigroup of Markov kernels on $X$  related with
 the integral transform $(X,\hat X, (\phi_\lambda)_{\lambda\in\hat X})$
 and  with negative definite function $\psi\in C(\hat X)$. Let $(X_t)_{t\ge 0}$ be
  a  c\`adl\`ag process on $X$.
Then, for each  $\lambda\in\hat X$,
$\bigl(\frac{1}{P_t(e,.)^\wedge(\lambda)} \cdot \phi_\lambda(X_t)\bigr)_{t\ge0}$
 is a martingale if and only if so is
 $$\Bigl(X_t^\lambda:=\phi_\lambda(X_t)+\psi(\lambda)\cdot \int_0^t \phi_\lambda(X_s)\>ds\Bigr)_{t\ge 0}.$$
\end{lemma}

\begin{proof}
By a boundedness argument,
 both  processes
 are  martingales if and only if the are local $L^2$--martingale.

Assume now that
 $\bigl((P_t(e,.)^\wedge(-\lambda))^{-1} \cdot \phi_\lambda(X_t)\bigr)_{t\ge0}$
   is a local $L^2$-martingale. Then
 $(\phi_\lambda(X_t))_{t\ge 0}$ is a semimartingale, and Ito integration
 yields
\begin{align}
d(\phi_\lambda(X_t)e^{t\psi(\lambda)})&=e^{t\psi(\lambda)}d \phi_\lambda(X_t) +
  \phi_\lambda(X_{t-})de^{t\psi(\lambda)}  \notag\\
&=e^{t\psi(\lambda)}\cdot(d\phi_\lambda(X_t) +
 \psi(\lambda) \phi_\lambda(X_t)dt).
\end{align}
Therefore,
$\displaystyle d\phi_\lambda(X_t) + \psi(\lambda) \phi_\lambda(X_t)dt=
 e^{-t\psi(\lambda)}\cdot d(\phi_\lambda(X_t)e^{t\psi(\lambda)})$
is the differential of a  local $L^2$-martingale as claimed.

The converse direction is similar.
\end{proof}

We next give a martingale characterization
 of time-homogeneous Markov processes $(X_t)_{t\ge0}$
 on $X$ associated with
a specific  semigroup  $(P_t)_{t\ge 0}$  with generator $L$ in the spirit of the martingale
 problem of Stroock
and Varadhan \cite{SV}. For this, we define  for any c\`adl\`ag process
 $(X_t)_{t\ge0}$ on $X$ and $f\in D(L)$
 the $\b C$--valued process
\begin{equation}
\Pi_X^{L,f}
=\Bigl(f(X_t)-f(X_0)- \int_0^t L(f)(X_s)ds\Bigr)_{t\ge 0}.
\end{equation}

\begin{theorem}\label{mart-characterization-3} Let $(P_t)_{t\ge 0}$ be a  semigroup of
Markov kernels on $X$  related with  $(X,\hat X,
  (\phi_\lambda)_{\lambda\in\hat X})$ and with
  negative definite function $\psi$ and generator $L$. Then the following
 are equivalent for  any  c\`adl\`ag process
$(X_t)_{t\ge 0}$  on $X$.
\begin{enumerate}
\item[{\rm (1)}]  $ (X_t)_{t\ge 0}$ is  Markov process
 associated with the semigroup $(P_t)_{t\ge 0}$.
\item[{\rm (2)}] For each $\lambda\in\hat X$, the process
$(\frac{1}{P_t(e,.)^\wedge(-\lambda)} \cdot \phi_\lambda(X_t))_{t\ge0}$
is a martingale.
\item[{\rm (3)}] $\Bigl(\phi_\lambda(X_t)+\psi(\lambda)\cdot \int_0^t \phi_\lambda(X_s)\>ds\Bigr)_{t\ge 0}$
 is a martingale
for each  $\lambda\in\hat X$ .
\item[{\rm (4)}]  $\Pi_X^{L,\phi_\lambda}$ is a martingale
 for each  $\lambda\in\hat X$.
\item[{\rm (5)}] $\Pi_X^{L,f}$
 is a   martingale for each  $f\in D_b(L)$.
\end{enumerate}
\end{theorem}

\begin{proof} The equivalence
of (1), (2) and (3) follows from \ref{mart-characterization-1} and \ref{mart-characterization-2}.
Moreover, $(3) \Leftrightarrow (4)\Leftarrow (5)$ is obvious, and
 $(1) \Rightarrow (5)$ is just the well-known Dynkin formula, see
 e.g.~Prop.~4.1.7 of \cite{EK}.
\end{proof}

Now consider a arbitrary root system on $\b R^N$ with multiplicity $k\ge0$, and
the generator $\Delta_k$  of the heat semigroup $(P_t)_{t\ge0}$ of  Dunkl type
on $\b R^N$. As $\Delta_k$
is a second-order differential-difference operator,  it is more convenient
here to restrict our attention to the subspaces
 $C_0^2(\b R^N)$,
  $C_b^2(\b R^N)$,  and  $C^2(\b R^N)$ instead
of the  domains $D_0(\Delta_k)$, $D_b(\Delta_k)$
and  $D(\Delta_k)$.
Moreover, it is possible to take
 test functions $f$ above which depend on the time.
If we define
$$\Pi_X^{\Delta_k,f}
=\Bigl(f(X_t,t)-f(X_0,0)- \int_0^t \bigl(\frac{\partial}{\partial s} +\Delta_k\bigr)f(X_s,s)\> ds
\Bigr)_{t\ge 0}$$
for  $f\in C^{2,1}(\b R^N\times [0,\infty))$,
the characterization above may be rewritten as follows by using standard
techniques (see Theorem 6.4 of \cite{RV1}):

\begin{theorem} Let $(P_t)_{t\ge 0}$ be the heat semigroup  of  Dunkl type
on $\b R^N$ with generator $\Delta_k$.
Then the following statements are equivalent for any c\`adl\`ag   process
$(X_t)_{t\ge 0}$  on $\b R^N$ whose radial part $(|X_t|)_{t\ge 0}$
 is continuous.
\begin{enumerate}
\item[$(1)$]   $X$ is a Markov process
 associated with  $(P_t)_{t\ge 0}$.
\item[$(5)$]  $\Pi_X^{\Delta_k,f}$
 is a   martingale for each  $f\in C_c^2(\b R^N)$.
\item[$(5')$]  $\Pi_X^{\Delta_k,f}$
 is a   martingale for each  $f\in C_c^{2,1}(\b R^N\times[0,\infty))$.
\item[$(6)$]  $\Pi_X^{\Delta_k,f}$
 is a   local  martingale for each  $f\in C^2(\b R^N)$.
\item[$(6')$]  $\Pi_X^{\Delta_k,f}$
 is a   local  martingale for each  $f\in C^{2,1}(\b R^N\times[0,\infty))$.
\end{enumerate}
\end{theorem}

For a detailed discussion of Dunkl processes, i.e., Markov processes on $\b R^N$
 associated with the Dunkl heat semigroup $(P_t)_{t\ge 0}$, we refer to
 \cite{Chy}, \cite{GY1},  \cite{GY2}, \cite{GY3}, and  to the
 surveys \cite{CGY},\cite{De} in this volume.

We finally note that a similar result is available for Gaussian processes on
Sturm-Liouville hypergroups, see \cite{ReV}.

\subsection{Moment functions}

In this section we consider a further  approach to  functions
$f\in C(X)$ which lead to martingales for Markov processes  on $X$
related to abstract integral transforms. Remember, that in Theorem \ref{mart-characterization-1}
 we used the functions $\phi_\lambda$ which  have multiplicative
 properties and which replace exponentials on the group $\b R^N$.
We now introduce analogs of polynomials, i.e.~functions with additive
properties. We begin with an example:

\begin{example}\label{group-moments-characters} Consider the usual group $X=\b R^N$ with characters
  $\phi_\lambda(x)=e^{-i<\lambda,x>}$, $\lambda\in\b R^n=\hat X$.
Then for $\nu\in\b  Z_+$, the monomials $m_\nu(x):=x^\nu$ satisfy
$$m_\nu(x)= i^{|\nu|}\cdot
\partial_\lambda^\nu\phi_\lambda(x)\bigl|_{\lambda=0} \quad\text{ and}\quad
\phi_\lambda(x)= \sum_{\nu\in\b  Z_+} (-i)^{|\nu|} \frac{m_\nu(x)}{\nu!}
\lambda^\nu$$
as well as the Leibniz rule
$$m_\nu(x+y)= \sum_{\rho\in \b  Z_+, \rho\le\nu} \binom{\nu}{\rho} m_\rho(x)\cdot
 m_{\nu-\rho}(y).$$
\end{example}

The Leibniz rule can be used to introduce so-called moment functions on
commutative hypergroups; see, e.g., \cite{Z1}, \cite{Z2}, \cite{ReV}, and
Ch.~7.2 of \cite{BH}:

\begin{definition}\label{def-moment-functions-hypergroups} Let $X$ be a commutative hypergroup. Define  $m_0:\equiv 1$.
\begin{enumerate}
 \item[(1)] A finite sequence $(m_i)_{i=1,\ldots ,n}\subset C(X)$
 is called a sequence of moment functions of length $n\in\b N$ if
\begin{equation}\label{leibniz-eindim}
(\delta_x*\delta_y)(m_i)=\sum_{j=0}^i
\binom{i}{j} m_j(x) m_{i-j}(y)
\quad \quad (i=1,\ldots,n; \>\>x,y\in X).
\end{equation}
 \item[(2)]  For a fixed sequence $(m_i)_{i=1,\ldots,n}$
 of moment functions, define the space
$$M_n^1(X):= \{\mu\in M^1(X):\> m_i\in L^1(X,\mu)
\quad{\rm for\>\> all\>\>} 0\le i\le n\}$$
 of probability measures for which
all moments up to the $n$-th  exist.
\end{enumerate}
\end{definition}

We collect a few  basic results for moment functions on commutative hypergroups
$X$ from \cite{ReV}:

\begin{remark}\label{moment-functions-hypergroups}
\begin{enumerate}
 \item[(1)]If  $(m_i)_{i=1,\ldots ,n}$ is a sequence of
 measurable, locally bounded functions on $X$ satisfying  (\ref{leibniz-eindim}),
then all $m_i$ are continuous.
\item[(2)] By induction, moment functions satisfy $m_i(e)=0$ for $i=1,\ldots,n$.
 \item[(3)] If $\mu,\nu\in M^1_n(X)$, then $\mu*\nu\in M_n^1(X)$ and
\begin{equation}
\mu*\nu(m_i)= \sum_{j=0}^i \binom{i}{j} \mu(m_j)\cdot
 \nu(m_{i-j}) \quad{\rm for}\>\> 0\le i\le n.
\end{equation}
 \item[(4)] If $(\mu_t)_{t\ge0}\subset M^1_n(X)$ is
 a convolution semigroup on $X$, then the functions
$f_i:[0,\infty)\to\b R$, $t\mapsto \int_X m_i\> d\mu_t$
 satisfy
\begin{equation}
f_i(s+t)= \sum_{j=0}^i  \binom{i}{j} f_j(s)\cdot
 f_{i-j}(t) \quad\quad(0\le i\le n,\>\> s,t\ge0).
\end{equation}
Moreover, if the  $f_i$ are continuous at $t=0$,  induction yields that the
$f_i$ are polynomials with
$$f_i(t)=\sum_{j=1}^i c_{i,j}t^{i+1-j} \quad
 (1\le i\le n, \> t\ge0)$$
for unique  $c_{i,j}\in\b R$. Moreover, the generator $L$ of the reflected semigroup
 $(\mu_t^-)_{t\ge0}$
satisfies
$$Lm_i(x) =\sum_{j=0}^{i-1} m_j(x)
 \binom{i}{j} c_{i-j,i-j} \quad{\rm for}\>\> 1\le i\le n,\> x\in X.$$
\end{enumerate}
\end{remark}

In practice, first and second moments play the most prominent role.
Under the conditions of Remark \ref{moment-functions-hypergroups}(4) we have
for $m_1,m_2$ and a convolution semigroup  $(\mu_t^-)_{t\ge0}$ with
 $\lim_{t\to0}\int_X m_i\> d\mu_t=0$  for $i=1,2$ that
\begin{equation}
\int_X m_1\> d\mu_t=c_1t \quad \int_X m_2\>
 d\mu_t=c_1^2t^2 +c_2t \quad (t\ge 0), \quad{\rm and}
\end{equation}
\begin{equation}
Lm_1(x)=c_1, \>\quad Lm_2(x)= 2c_1m_1(x)+c_2
 \quad{\rm for}\>\> x\in X.
\end{equation}
Moreover, we can construct martingales from moment functions. Here
is a result for the first and second moments due to \cite{Z2} which can be
easily checked:

\begin{lemma}\label{martingale-additive-hypergroup} Let $(X_t)_{t\ge0}$ be a L\'evy
 process on $X$ associated with  $(\mu_t)_{t\ge 0}\subset M^1_2(X)$ in the
 setting above.
Then  $(m_1(X_t)-E(m_1(X_t))_{t\ge0}$ and
 $$\bigl(m_2(X_t)-2m_1(X_t)\cdot E(m_1(X_t))-
E(m_2(X_t))+2E(m_1(X_t))^2\bigr)_{t\ge0}$$
are martingales.
\end{lemma}

Clearly, these result may be extended to higher moments.
We mention that Lemma \ref{martingale-additive-hypergroup} together with
martingale convergence theorems can be used to derive strong limit theorems for
the processes $(X_t)_{t\ge0}$ for $t\to\infty$. For details see the monograph \cite{BH}.

We next turn to the question, how we can  construct examples  moment functions on given
hypergroups, and how the concept of moment functions can be lifted to abstract integral
transforms.
Motivated by the observations for the group $\b R^N$  in
\ref{group-moments-characters}, we introduce the following notion:

\begin{definition}\label{def-diff-integral-trafo}  An abstract integral
transform  $T=(X,\hat X, (\phi_\lambda)_{\lambda\in\hat X})$ is called
differentiable, if the following holds:
\begin{enumerate}
\item[\rm{(1)}] $\hat X$ is a closed subset on $\b R^N$ with  $\hat e=0\in \hat X$.
\item[\rm{(2)}] There exist $\zeta_1,\ldots, \zeta_N\in\b R^N$
 and $\epsilon>0$ with $t_1\zeta_1+\ldots+t_N\zeta_N\in \hat X$ for $t_1,\ldots,t_N\in[0,\epsilon]$.
\item[\rm{(3)}] If for  $\nu\in\b  Z_+^n$  and
  $\lambda=t_1\zeta_1+\ldots+t_N\zeta_N\in \hat X$ we use the notation 
$$\partial_\lambda^\nu=\partial_{\zeta_1}^{\nu_1}\ldots
  \partial_{\zeta_N}^{\nu_N}$$
for partial derivatives, we assume that for $x\in X$ and  $\nu\in\b  Z_+^n$, 
$$m_\nu(x):=i^{|\nu|}\cdot\partial_\lambda^\nu \phi_\lambda(x)\bigl|_{\lambda=0} $$
exists, and $m_\nu$ is continuous on $X$.
\end{enumerate}
 For an abstract integral transform, the functions $(m_\nu)_{\nu\in\b  Z_+^n} $
 are called moment functions.
\end{definition}

\begin{examples}\label{example4-26}
\begin{enumerate}
\item[\rm{(1)}] For the group $\b R^N$, we have
 $\phi_\lambda(x)=e^{-i<\lambda,x>}$,
 $\lambda\in\b R^n=\hat X$. Taking the $\zeta_k$ as
    unit vectors $e_k$, we obtain
$m_\nu(x)=x^\nu.$
\item[\rm{(2)}] Consider the Bessel-Kingman hypergroup on $X=[0,\infty)$ of index
 $\alpha\ge-1/2$ with $\hat X=[0,\infty)$ and the characters
$\phi_\lambda(x):=j_\alpha(\lambda x)$
with
 the normalized Bessel function
\begin{equation}
j_\alpha(z)=\sum_{k= 0}^\infty     (-1)^k\,\frac{\Gamma(\alpha+1)}{2^{2k}k!\,
\Gamma(\alpha+k+1)}\,z^{2k}
\quad(z\in\b C).
\end{equation}
Then we obtain the moment functions  $m_{2n+1}\equiv 0$ and
\begin{equation}
m_{2n}(x):= \frac{\Gamma(\alpha+1) (2n)!}{2^{2n}n!\,
\Gamma(\alpha+n+1)} x^{2n}
\end{equation}
for $n\in\b  Z_+$  in the sense of the preceding
definition.
\item[(3)] Consider the Bessel hypergroups on the matrix cones $\Pi_q$
  introduced in Chapter \ref{section-matrix-cones}. A complete system of
  moment functions is given here in \cite{V2}. Moreover, applications to
  strong laws of large numbers for random walks in this setting are also given there.
\end{enumerate}
\end{examples}

Before turning to moment functions associated with Dunkl kernels, we notice
that the Leibniz rule for partial derivatives  leads to:

\begin{lemma} Let $X$ be a commutative hypergroup with dual $\hat X$ such that
  the hypergroup Fourier transform is differentiable in the sense of
  \ref{def-diff-integral-trafo}.
Then, for any nonnegative linear combination $\zeta$ of the vectors  $\zeta_1,\ldots, \zeta_N\in\b R^N$
in  \ref{def-diff-integral-trafo}(2), the functions
  $m_n(x):=\partial_{\zeta}^{n} \phi_\lambda|_{\lambda=0}$ ($n\in\b Z_+$) with
  partial derivatives w.r.t.~$\lambda$
  form  moment functions in the sense of Definition \ref{def-moment-functions-hypergroups}(1).
\end{lemma}

For Markov kernels  associated with a differentiable
integral transform, this additivity reads as follows:

\begin{lemma}\label{additivity-moments-allg}
 Let  $P,Q$ be Markov kernels on $X$  related with
 the differentiable integral transform $(X,\hat X,
 (\phi_\lambda)_{\lambda\in\hat X})$.
 Assume that for some $n\in\b N$ and  all $x\in X$, the measures $P(x,.), Q(x,.), Q\circ P(x,.)$ are contained
 in
\begin{equation}\label{moment-existence}
M_n^1(X):=\{\mu\in M^1(X):\> _\nu\in L^1(X,\mu) \>\> \text{for}\>\>
 |\nu|\le n\}.
\end{equation}
Then, for all
 $x\in X$ and $\nu\in\b Z_+^N$ with $|\nu|\le n$ and with the notion
 $m_\nu(\mu):=\int_X m_\nu\> d\mu$,
\begin{enumerate}
\item[\rm{(1)}] $\displaystyle m_\nu(P(x,.))= \sum_{\rho\le\nu}
 \binom{\nu}{\rho} m_\rho(P(e,.))\cdot m_{\nu-\rho}(x)$,
\item[\rm{(2)}]  $\displaystyle m_\nu(Q\circ P(x,.))= \sum_{\rho\le\nu}
\binom{\nu}{\rho} m_\rho(P(x,.))\cdot
 m_{\nu-\rho}(Q(e,.))$.
\end{enumerate}
 \end{lemma}

\begin{proof} (1) follows  from the Leibniz rule; in fact, as we may
  interchange differentiation and integration, we observe by
  differentiation w.r.t.~$\lambda$ that
\begin{align}
  m_\nu(P(x,.))  &= i^{|\nu|}\cdot \partial_\zeta^\nu\Bigl( P(e,.)^\wedge(\lambda)\cdot
 \phi_\lambda(x))\Bigr)\Bigl|_{\lambda=0} \notag\\
  &= \sum_{\rho\le\nu} \binom{\nu}{\rho} i^{|\rho|}\cdot
 \partial_\zeta^\rho( P(e,.)^\wedge(\lambda))|_{\lambda=0}\cdot
 i^{|\nu|-|\rho|}\cdot\partial_\zeta^{\nu-\rho} \phi_\lambda(x)|_{\lambda=0} \notag\\
 &= \sum_{\rho\le\nu} \binom{\nu}{\rho} m_\rho(P(e,.))\cdot m_{\nu-\rho}(x).
 \notag
\end{align}
Part (2) can be checked in the same way by using Lemma \ref{prop-kernels}(2).
\end{proof}

As in Lemma \ref{martingale-additive-hypergroup}, we can construct martingales:

\begin{proposition}\label{mart-moments-inttrafo}
 Let $(P_t)_{t\ge 0}$ be a semigroup of Markov kernels related with
 the differentiable integral transform $(X,\hat X,
 (\phi_\lambda)_{\lambda\in\hat X})$ with negative definite function $\psi$, and let
  $(X_t)_{t\ge0}$ be  an associated Markov process on $X$ with canonical
 filtration  on the underlying
probability space $(\Omega,{\cal A},P)$. Then  the moment functions of Section
 \ref{def-diff-integral-trafo}  satisfy:
\begin{enumerate}
\item[\rm{(1)}] If the measures $P_{X_0}$ and $P_t(x,.)$ are contained in
  $M^1_1(X)$ for $t>0$ and $x\in X$ in the sense of (\ref{moment-existence}), then
  for $l=1,\ldots,N$,
the process $(m_{e_l}(X_t)-E(m_{e_l}(X_t)))_{t\ge0}$
is a  martingale with
$$E(m_{e_l}(X_{t}))=E(m_{e_l}(X_0)) -it\cdot \partial_{\zeta_l}\psi(0)
\quad\quad\text{for}\quad t\ge0.$$
\item[\rm{(2)}] If the measures $P_{X_0}$ and $P_t(x,.)$ are contained in
   $M^1_2(X)$ for $t>0$ and $x\in X$, then  for $l,j=1,\ldots,N$,
\begin{align}
\Bigl( m_{e_l+e_j}&(X_t)- m_{e_l}(X_t)\> E(m_{e_j}(X_t))-  m_{e_j}(X_t)\> E(m_{e_l}(X_t))
 \notag\\
 &+   E(m_{e_l}(X_t))\> E(m_{e_j}(X_t)) -E(m_{e_l+e_j}(X_t))\Bigr)_{t\ge0}\notag
\end{align}
is a  martingale. In particular, the ``modified variances''
$$V^l(X_t):=   E(m_{2e_l}(X_t)) -  E(m_{e_l}(X_t))^2$$
 satisfy
$V_k^l(X_{t})= V_k^l(X_{0}) +t\cdot \partial_{\zeta_l}^2\psi(0)$
for $t\ge0.$
\end{enumerate}\end{proposition}

\begin{proof}\begin{enumerate}
\item[(1)] We have $\psi(0)=0$ and
 $P_t(e,.)^\wedge=e^{-t\psi}$
for $t\ge0$. Therefore, by the dominated convergence theorem,
$$m_{e_l}(P_t(e,.))= i\cdot\partial_{\zeta_l}(P_t(e,.)^\wedge)(0) = -it\> \partial_{\zeta_l} \psi(0).$$
 Now take  $s,t\ge0$. The preceding lemma ensures that for a.a.~$\omega\in\Omega$,
\begin{align}
  E(m_{e_l}(X_{s+t})|\> {\cal F}_t)(\omega)  &=  \int_{X} m_{e_l}\>
 dP_t(X_s(\omega),.) \notag\\
&= m_{e_l}(P_t(e,.))+
m_{e_l}(X_s(\omega))\notag\\
&= -it\cdot\partial_l\psi(0) + m_{e_l}(X_s(\omega)).
\end{align}
If we take the usual expectation of both sides  with $s=0$,
  we obtain the claimed formula
for $E(m_{e_l}(X_t))$, and that
   $(m_{e_l}(X_t)-E(m_{e_l}(X_t)))_{t\ge0}$
is a martingale.
\item[(2)] can be shown in a similar way; c.f. Proposition 7.5 of \cite{RV1}.
\end{enumerate}
\end{proof}

We now turn to moment functions in the sense of
  \ref{def-diff-integral-trafo} for the Dunkl transform.
Recall that the Dunkl kernel $E_k$ is analytic on
 $\b C^{N\times N}$, i.e., there
there exist unique analytic functions $m_\nu$ ($\nu\in \b Z_+^N$)  with
\begin{equation}
E_k(x,y)=\sum_{\nu\in\b Z_+^N}\frac{m_\nu(x)}{\nu!}y^\nu
\quad\quad (x,y\in\b C^N)
\end{equation}
with
\begin{equation}\label{moments-via-derivative}
m_\nu(x)= (\partial_y^\nu E_k(x,y))|_{y=0}= i^{|\nu|}\> (\partial_y^\nu E_k(x,-iy))|_{y=0}.
\end{equation}
Therefore, the  $m_\nu$ are  moment functions in the sense of \ref{def-diff-integral-trafo}.

We denote the $j$-th unit vector  by $e_j\in\b Z_+^N$. and the
  the moment functions of order 1 and 2 by
$m_{e_j}$ and $m_{e_j+e_k}$  ($j,k=1,\ldots,N$).

>From the description of the Dunkl kernel $E_k$ via the intertwiner $V_k$ (see
 the definition of the Dunkl kernel!)  we immediately obtain
\begin{equation}\label{moments-via-V}
m_\nu(x)=  V_k(x^\nu)\quad\quad\quad {\rm for}\quad \nu\in\b Z_+^N.
\end{equation}
In particular, for each $n\in\b Z_+$ the moment functions $m_\nu$
 with $|\nu|=n$ form a basis of the space
 ${\cal P}_n$ of all homogeneous polynomials of degree $n$.
This in particular implies that a measure $\mu\in M^1(\b R^N)$  is contained
 in $M^1_n(\b R^N)$ (i.e., it has moments up to order $n$ in the meaning
 above) if and only if it has usual moments up to order $n$.

Moreover, via the
 recurrence relation for $V_k$ in the beginning of Chapter \ref{Dunkl's intertwining operator}  (see
 \cite{Du2},\cite{DX}) it is possible to compute the
 moment functions $m_\nu$.

\begin{examples}\label{examples-moment-functions-Dunkl}
\begin{enumerate}
\item[(1)] If $k=0$, then $E_k(x,y)=e^{< x,y> }$ and
$m_\nu(x)=x^\nu$.
\item[(2)] If $N=1$, $W=\b Z_2$ and $k\ge0$,
then the explicit form of $E_k$ in terms of Bessel functions (see
Example \ref{E:1kern} and compare with \eqref{V-k-eindim}) implies
\begin{align}
m_{2n}(x)&=\frac{ \Gamma(k+1/2)\> (2n)!}{\Gamma(n+k+1/2)\>2^{2n}\> n!}x^{2n}
\notag\\
m_{2n+1}(x)&=\frac{ \Gamma(k+1/2)\> (2n+1)!}{\Gamma(n+k+3/2)\>2^{2n+1}\>
  n!}x^{2n+1}.\notag
\end{align}
\item[\rm{(3)}] {\bf The $A_{N-1}$-case:}
For the symmetric group  $W=S_N$ and multiplicity $k\in [0,\infty)$, some
  computation with the intertwiner yields
\begin{equation}
m_{e_l}(x)= V_kx_l = \frac{1}{1+kN}\Bigl( x_l +k\sum_{i=1}^N x_i\Bigr).
\end{equation}
for the moment functions of first order and similar formulas for that of
second order; see Section 7.1 of \cite{RV1} for details.
\item[\rm{(4)}] {\bf The $B_{N}$-case:}
Here the multiplicity consists of two parameters
$k_0, k_1\geq 0$, and  it follows  (see
\cite{Du7} and \cite{RV1})
 that for  $l,j\in\{1,\ldots,N\}$,
\begin{align}
m_{e_l}(x)&= V_kx_l = \frac{x_l}{1+2k_1+2k_0(N-1)},
\notag\\
m_{e_l+e_j}&=V_k(x_lx_j)= \frac{x_lx_j}{1+2k_1+2k_0(N-1)},
 \quad  \quad \quad {\rm for}\quad l\ne j,
\notag\\
m_{2e_l}(x)&= V_kx_l^2 = \frac{x_l^2 + k_0\sum_{i=1}^N x_i^2}{(1+Nk_0)(1+2(N-1)k_0 +2k_1)}.
\notag
\end{align}
\end{enumerate}
\end{examples}

We next collect some properties of moment functions.
We mention that similar results are also available for
 Sturm-Liouville hypergroups on $[0,\infty)$; see \cite{BH}, \cite{ReV}, \cite{Z1}, \cite{Z2}.

\begin{proposition}\label{taylor-etc-moment}
 For all $x\in\b R^N$, $\nu\in\b Z^N_+$, and $l\in\{1,\ldots,N\}$:
\begin{enumerate}
\item[\rm{(1)}] $T_lm_{\nu+e_l}=(\nu_l+1)\cdot m_\nu$.
\item[\rm{(2)}] $|m_\nu(x)|\le |x|^{|\nu|}$  and $0\le m_\nu(x)^2 \le m_{2\nu}(x)$.
\item[\rm{(3)}] Taylor formula: If $f\in C^{(n)}(\b R^N)$ for $n\in\b N$, then
$$f(y) = \sum_{\nu\in\b Z_+^N,\> |\nu|\le n} \frac{m_\nu(y)}{\nu!} \> T^\nu\!f(0) +o(|y|^n)
\quad\quad\text{ for} \quad y\to0.$$
Moreover, if $f:\b C^N\to\b C$ is analytic in a neighborhood of 0, then
$$f(y)=\sum_{n=0}^\infty \sum_{ |\nu|= n}  \frac{m_\nu(y)}{\nu!}  \> T^\nu\!f(0)$$
where the series  converges absolutely and locally uniformly.
\end{enumerate}
\end{proposition}

\begin{proof}
\begin{enumerate}
\item[(1)] The  properties of $V_k$   and (\ref{moments-via-V}) yield
$$T_l m_{\nu+e_l}= T_l V_kx^{\nu+e_l} = V_k\partial_lx^{\nu+e_l}= (\nu_l+1)\cdot V_kx^l=(\nu_l+1)\cdot
m_\nu.$$
\item[(2)]  The positive integral representation of $E_k$ in Proposition \ref{C:Bochner}
 together with (\ref{moments-via-derivative}) imply that for
 $x\in\b R^N$  there exists
 $\mu_x\in  M^1(\b R^N)$ with
$supp\>\mu_x\subset \{ z\in\b R^N:\> |z|\le|x|\}$ and
$$m_\nu(x)= \int_{\b R^N} y^\nu\> d\mu_x(y)
 \quad\quad\text{for all}\quad \nu\in\b Z_+^N, \>  x\in\b R^N.$$
The first inequality  is now clear
 from the support condition on $\mu_x$
 while the second one follows from Jensen's inequality.
\item[(3)] See Corollary \ref{C:Taylor}.
\end{enumerate}
\end{proof}

\begin{example}\label{exampes-moments-gauss} Let $(X_t)_{t\ge0}$ be a Dunkl process on $\b R^N$ with $X_0=0$
 associated with the Dunkl
 heat semigroup $(P_t^\Gamma)_{t\ge 0}$.
In this case, all moments exists, and
$$ P_t^\Gamma(0,.)^\wedge(y)= e^{-t|y|^2} =
\sum_{\nu\in\b Z_+^N} \frac{(-t)^{|\nu|}}{\nu!}\> y^{2\nu}
\quad\quad{\rm for}\quad t\ge0,\>\> y\in\b R^N.$$
 This yields that
\begin{equation}
E(m_{2\nu}(X_t))=m_{2\nu}(  P_t^\Gamma(0,.)) = \frac{(2\nu)!}{\nu!}\> t^{|\nu|} \quad\quad
(\nu\in\b Z_+^N,\> t\ge0)
\end{equation}
and  $E(m_{\nu}(X_t))=0$ whenever at least one component of $\nu$ is odd.

 Proposition \ref{mart-moments-inttrafo}(1) now implies
that the processes $(m_{e_l}(X_t))_{t\ge0}$ are martingales for
$l\in\{1,\ldots,N\}$. Moreover, as the $m_{e_l}$ form a basis of
the space ${\cal P}_1$ of all homogeneous polynomials
 of degree 1,   $(X_t)_{t\ge0}$ itself is an $N$-dimensional martingale.

  Moreover, Proposition \ref{mart-moments-inttrafo}(2)
and $E(m_{e_l}(X_t))=0$ show that
$$(m_{e_l+e_j}(X_t) - E(m_{e_l+e_j}(X_t)))_{t\ge0}$$
is a martingale for  $l,j\in\{1,\ldots,N\}$.
As the moment functions $m_{e_j+e_l}$ form a basis of ${\cal P}_2$,
it follows that for all  $l,j\in\{1,\ldots,N\}$, the processes
$$(X_t^l\cdot X_t^j - E(X_t^l\cdot X_t^j ))_{t\ge0}$$
are martingales. For higher moments, results of this type are more complicated
and will be considered  in the next section.
\end{example}

\subsection{General Appell characters}

Let $(X_t)_{t\ge0}$ be time-homogeneous Markov process on $X$ associated with
some abstract integral transform as above.
Based on concept of  moment functions  and certain generating functions,
 we  construct a system $(R_\nu)_{\nu\in\b Z_+^N}$
of functions on $\b R\times X$ associated with $(X_t)_{t\ge0}$ such that the
processes $(R_\nu(t, X_t))_{t\ge0}$ become martingales.
These systems, called Appell characters, generalize the well-known
heat polynomials, which are connected with Brownian motion and  given
in terms of classical Hermite polynomials.

 The concept of Appell characters is quite old and has its origin in umbral
 calculus; see for instance     \cite{FS}, \cite{Rom}.

We begin with a general definition. Later we shall restrict our attention
mainly to  Dunkl processes.

\begin{definition}\label{setting-appell}
 Let $(P_t)_{t\ge 0}$ be a semigroup of Markov kernels related with
 the differentiable integral transform $(X,\hat X,
 (\phi_\lambda)_{\lambda\in\hat X})$ with negative definite function $\psi\in
 C(\hat X)$: we use the notions of \ref{def-diff-integral-trafo} and
assume that  $P_t(x,.)\in M_n^1(X)$ for
 $t\ge0$ and $x\in X$, i.e., that all moments up to order $n\ge1$ exist.
 We know from the preceding section that
 $P_t(e,.)^\wedge= e^{-t\!\psi}\in C^{n}(\hat X)$  for  $t\ge0$. Therefore,
$$\lambda\longmapsto \frac{\phi_\lambda(x)}{P_t(e,.)^\wedge(\lambda)}= \phi_\lambda(x)\cdot e^{t
  \psi(\lambda)}$$
is  $n$-times continuously
 differentiable for  $t\ge0$, $x\in X$. By Taylor's formula,
\begin{equation}\label{taylor-appell}
\phi_\lambda(x)\cdot e^{t\psi(\lambda)}= \sum_{\nu\in\b Z_+^N,\> |\nu|\le n}
 \frac{(-i\lambda)^\nu}{\nu!} R_\nu(t,x) + o(|\lambda|^n)
\quad{\rm for}\>\> y\to0;
\end{equation}
with unique functions
\begin{align}\label{alg-prop-p}
 R_\nu(t,x)  &=i^{|\nu|} \partial_\lambda^\nu\Bigl(  \phi_\lambda(x)\cdot e^{t \psi(\lambda)}  \Bigr)
\Bigr|_{\lambda=0} \notag\\
&= i^{|\nu|}
\sum_{\rho\in\b Z_+^N,\> \rho\le \nu}
 \binom{\nu}{\rho} \partial_\lambda^\rho(\phi_\lambda(x))\Bigr|_{\lambda=0}\cdot
 \partial_\lambda^{\nu-\rho}( e^{t \psi(\lambda)})\Bigr|_{\lambda=0}
\notag\\
 &=\sum_{\rho\in\b Z_+^N,\> \rho\le \nu}
 \binom{\nu}{\rho}  m_\rho(x)\cdot a^\psi_{\nu-\rho}(t)
\end{align}
by Definition \ref{def-diff-integral-trafo}  with the polynomials
\begin{equation}
a_\rho^\psi(t):= i^{|\rho|}\cdot \partial_\lambda^\rho( e^{t \psi(\lambda)})\bigr|_{\lambda=0}
\quad\quad\quad(\rho\in\b Z_+^N,\> |\rho|\le n)
\end{equation}
in $t$ of degree at most $|\rho|$.
Note that  $a_\nu^\psi(-t)=
m_\nu(P_t(e,.))\in\b R$  holds for $t\ge0$.
The functions  $R_\nu$ will be called  {\it Appell characters}
associated with the semigroup $(P_t)_{t\ge0}$.
\end{definition}
 We next collect some basic properties of Appell characters:

\begin{lemma}\label{prop-r-apell}
 In the setting above, the following  holds for
  $\nu\in  \b Z_+^N$
with $|\nu|\le n$:
\begin{enumerate}
\item[\rm{(1)}] Inversion formula:
 For all  $x\in X$ and $t\in\b R$,
$$m_\nu(x)=  \sum_{\rho\in\b Z_+^N,\> \rho\le \nu}
 \binom{\nu}{\rho} R_\rho(t,x)\cdot a_{\nu-\rho}^\psi(-t).$$
\item[\rm{(2)}]  For  $x\in X$ and  $t\ge0$,
$\displaystyle \int_{X} R_\nu(t,y)\> dP_t(x,dy)= m_\nu(x).$
\end{enumerate}
\end{lemma}

\begin{proof}
\begin{enumerate}
\item[(1)]  Write down the Taylor expansion of order $n$ of
$$\phi_\lambda(x) =
e^{-t\psi(\lambda)}\cdot (\phi_\lambda(x)\cdot e^{t\psi(\lambda)})$$
as above in two ways and compare the coefficients; see Lemma 8.1 of \cite{RV1}.
\item[(2)]   Recall that $a_\nu^\psi(-t)=
m_\nu(P_t(e,.))$. Eq.~(\ref{alg-prop-p}) and Lemma
\ref{additivity-moments-allg}(1) yield
{\small \begin{align}
\int_{X} R_\nu(t,y)\>  dP_t(x,dy)&= \sum_{\rho\le\nu}  \binom{\nu}{\rho}
 a_{\nu-\rho}^\psi(t) \cdot\int_{X} m_\rho(y) \>  dP_t(dy)\notag\\
&= \sum_{\rho\le\nu}  \binom{\nu}{\rho}
 a_{\nu-\rho}^\psi(t) \cdot\Bigl( \sum_{\phi\le\rho}  \binom{\rho}{\phi}
m_\phi(P_t(e,.))\cdot m_{\rho-\phi}(x)\Bigr)\notag\\
&= \sum_{\rho\le\nu}  \binom{\nu}{\rho}
 a_{\nu-\rho}^\psi(t) \cdot\Bigl( \sum_{\phi\le\rho}  \binom{\rho}{\phi}
 a_\psi^\phi(-t) \cdot m_{\rho-\phi}(x)\Bigr).\notag
\end{align}}
The assertion now follows from Part (1).
\end{enumerate}
\end{proof}

\begin{remark}\label{T-und-R} For Dunkl theory on $X=\b R^N$, part (1) of the
  preceding result implies that for   $t\in\b R$ and suitable $l\in\b N$, the
 $(R_\nu(t,.))_{\nu\in\b Z_+^N,\> |\nu|\le l}$
form  a basis of the space $\bigoplus_{j=0}^l {\cal P}_j$ of all
polynomials of degree  at
most $l$.

Moreover, for  all  $x\in\b R^N$, $t\in\b R$, and
  $j\in\{1,\ldots,N\}$,
\begin{equation}T_j R_{\nu+e_j}(t,x)=
(\nu_j+1)\cdot R_\nu(t,x)
\end{equation}
 where $T_j$ acts with
respect to the variable $x$.

In fact, the expansion of $R_{\nu}$ and Proposition \ref{taylor-etc-moment}(1)
yield
\begin{align}
T_j  R_{\nu+e_j} &=\sum_{\rho\le \nu+e_j} \binom{\nu+e_j}{\rho} T_j
m_\rho\cdot  a_{\nu+e_j-\rho}^\psi =
 \sum_{\rho\le \nu} \binom{\nu+e_j}{\rho+e_j}(\rho_j+1) m_\rho \cdot  a_{\nu-\rho}^\psi\notag\\
& =  (\nu_j+1)\cdot  \sum_{\rho\le \nu} \binom{\nu}{\rho}   m_\rho\cdot   a_{\nu-\rho}^\psi
  =(\nu_j+1)\cdot R_\nu\,.
\notag
\end{align}
\end{remark}

We return to the general setting:

\begin{theorem}\label{martingales-appell-allg}
 Let $n\ge1$ and  $(P_t)_{t\ge0}$  a
 semigroup of
 Markov kernels on $X$ as in Definition \ref{setting-appell} above.
Let $(X_t)_{t\geq 0}$ be an associated Markov process with the following
 property:
\begin{enumerate}
\item[\rm{(*)}]
There exists $\epsilon>0$ with
$\lambda=\sum_{j=1}^n t_j\zeta_j\in\hat X$
for all  $0\le t_j\le\epsilon$ ($j=1,\ldots, N$) and the directions $\zeta_1,\ldots,\zeta_N\in\b
R^n$ in \ref{def-diff-integral-trafo}. Moreover, 
for all $\nu\in\b Z_+^N$ with $|\nu|\le n$, the functions
$$x\mapsto \sup_ {\lambda=\sum_{j=1}^n t_j\zeta_j, 0\le t_j\le\epsilon}  |\partial_\lambda^\nu \phi_\lambda(x)|$$
are integrable w.r.t.~the distributions $P_{X_t}$ of all random variables $X_t$.
\end{enumerate}
\noindent Then for each $\nu\in\b Z_+^N$ with $|\nu|\le n$,
 $(R_\nu(t,X_t))_{t\ge0}$ is a martingale.
\end{theorem}

\begin{proof} We prove  more generally by induction on  $|\nu|$ that for each $\lambda\in\hat X$
 and $\nu\in\b Z_+^N$ with $|\nu|\le n$, the process
\begin{equation}\label{def-martingale-process}
\Bigl(W_t^{\nu,\lambda}:= \partial_\lambda^\nu(\phi_\lambda(X_t)\cdot e^{t\phi(\lambda)})\Bigr)_{t\ge0}
\end{equation}
is a martingale for $\lambda=\sum_{j=1}^n t_i\zeta i$ with $0\le t_j<\epsilon$.
The theorem then follows for $\lambda=0$.

 In fact,
 the case $\nu=0$ follows from Proposition \ref{mart-characterization-1}.
For the induction step,  consider some direction $\zeta_j$ as in Definition
 and assume that  $W_t^{\nu,\lambda}$ is a martingale for all
 $\lambda\in\b R^N$ as above and some $\nu\in\b Z_+^N$ with $|\nu|\le n$.
To prove that
 $(W_t^{\nu+e_j,\,\lambda})_{t\ge0}$ is a martingale for these $\lambda$, we observe that for $t\ge0$,
$$\lim_{h\to0} \frac{1}{h} \Bigl( W_t^{\nu,\,\lambda}- W_t^{\nu,\,\lambda+h\cdot \zeta_j}\Bigr) =
W_t^{\nu+e_j,y}\quad\quad{\rm pointwise}.$$
Moreover, by the mean value theorem, we find $r\in[0,h]$ with
\begin{equation}
\Bigl| \frac{1}{h} \Bigl( W_t^{\nu,\,\lambda}- W_t^{\nu,\,\lambda+h\cdot \zeta_j}\Bigr)\Bigr|=
\bigl| W_t^{\nu+e_j,\,\lambda+r\cdot \zeta_j}\bigr|.
\end{equation}
Condition (*)
 ensures that the dominated convergence theorem may be applied to
the limit  above, and hence
$$\lim_{h\to 0}\Bigl\| \frac{1}{h}  \Bigl( W_t^{\nu,\,\lambda}- W_t^{\nu,\,\lambda+h\cdot \zeta_j}\Bigr)
-W_t^{\nu+e_j,\,\lambda}\Bigr\|_1 =0 \quad\quad \text{for all }\> t\ge0.$$
It follows for the canonical filtration $({\cal F}_t)_{t\ge0}$ of
 $(X_t)_{t\ge0}$ that for $s,t\ge0$,
$$E\Bigl( \frac{1}{h}  \Bigl( W_{s+t}^{\nu,\,\lambda}- W_{s+t}^{\nu,\,\lambda+h\cdot \zeta_j}\Bigr)\Bigl|{\cal F}_t\Bigr)
 \longrightarrow E(  W_{s+t}^{\nu+e_j,\, \lambda}|{\cal F}_t)
 \quad\quad\text{a.s}$$
 Hence,
 $(W_t^{\nu+e_j, \lambda})_{t\ge0}$ is   a martingale.
\end{proof}

\begin{remarks}
\begin{enumerate}
\item[\rm{(1)}]
For concrete abstract integral transforms, condition (*) above
can be simplified considerably.

For instance, using bounds for derivatives of Dunkl kernels (see Corollary \ref{C:growthbounds}), it can
be easily checked that in the Dunkl setting condition (*) holds if and only if
all distributions $P_{X_t}$ of the process admit moments up to order $n$;
c.f. also the proof of Theorem 8.2 of \cite{RV1}.

Moreover, for Sturm-Liouville hypergroups on $[0,\infty)$, an analog result is
available; we refer to \cite{Z1}, \cite{Z2}, \cite{ReV}, and Section 7.2 of \cite{BH}.
In particular, for the Bessel hypergroups on $[0,\infty)$, the same result is
available as in the Dunkl setting; see also below.
\item[\rm{(2)}] For $|\nu|=1,2$, the martingales $R_\nu(t,X_t)$ of the theorem above
  agree with the martingales of Proposition \ref{mart-moments-inttrafo}.
\item[\rm{(3)}] Consider the Dunkl setting. Then for each polynomial $f\in{\cal P}$,  the
 polynomial function $u(x,t):= e^{t\Delta_k}f(x)$  satisfies $u_t=\Delta_ku$
 on $\b R^N\times \b R$ (see e.g.~Theorem 3.1(2) of \cite{RV1}). This fact
 together with Proposition \ref{dual-nature} below thus imply that
the Appell characters $R_\nu^\Gamma$ for
 the Dunkl heat semigroup satisfy
 $$(\partial_t+\Delta_k)R_\nu^\Gamma=0,$$
i.e., they form so-called heat polynomials for the Dunkl heat semigroup. This
again reflects the close connection between
Theorems \ref{martingales-appell-allg} and \ref{mart-characterization-3}.
\item[\rm{(4)}] In the Dunkl setting there is a close connection between general Appell characters $R_\nu$
and the intertwiner $V_k$: Let $n\ge0$ and  $(P_t^k)_{t\ge 0}$
  a semigroup of Dunkl-Markov kernels on $\b R^N$ such that
$P_t^k(x,.)\in M_n^1(\b R^N)$ for $t\ge0$ and $x\in\b R^N$.
Let $R_\nu^k$ be the associated Appell characters for $|\nu|\le n$.
By Theorem \ref{T:Main2}(2), there exist probability measures $\mu_x\in M^1(\b R^N)$
 such that the negative definite function $\psi$ associated with  $(P_t^k)_{t\ge 0}$
satisfies
$$e^{-t \psi(\lambda)}= P_t^k(0,.)^\wedge(\lambda)=\int_{\b R^N}\int_{\b R^N}
e^{\langle z,-i\lambda\rangle}\> d\mu_x(z)\> dP_t^k(0,.)(x)$$
for $t\ge0$, $\lambda\in\b R^N$. Hence, the $e^{-t \psi}$ are positive
definite in the classical sense, and
 by Bochner's theorem, there is a semigroup  $(P_t^0)_{t\ge 0}$ of
 group-translation invariant
 Markov kernels. If the associated Appell characters are
 denoted by $R_\nu^0$, we obtain
$$R_\nu^k(t,x) = \sum_{\rho\le\nu} \binom{\nu}{\rho} m_\rho(x)\> a_{\nu-\rho}^\psi(t) =
\sum_{\rho\le\nu} \binom{\nu}{\rho} (V_kx^\rho)\>  a_{\nu-\rho}^\psi(t) = V_kR_\nu^0(t,x).$$
\end{enumerate}
\end{remarks}

We next discuss Theorem \ref{martingales-appell-allg} for  Bessel processes on
$[0,\infty)$:

\begin{example}
For $\alpha\ge-1/2$, consider the Bessel hypergroup on
$[0,\infty)$. On this hypergroup there exists up to time normalization a
unique Gaussian convolution semigroup in the sense of Theorem   \ref{gauss-characterization}, namely
\begin{equation}
d\rho_t^\alpha(x)=
\frac{1}{\Gamma(\alpha+1)}\ \frac{2^\alpha}{t^{\alpha+1}}\ x^{2\alpha+1}\,
e^{-x^2/(2t)}\,dx\quad{\rm on}\>\> [0,\infty)
\quad{\rm for}\>\> t>0
\end{equation}
with generator
$Lf:= \frac{1}{2} \bigl( f^{\prime\prime}+\frac{2\alpha+1}{x}f^{\prime}\bigr)$.
The associated Markov processes are Bessel processes $(X_t^\alpha)_{t\ge0}$ of order $\alpha$. In
particular, for the  $N$-dimensional Brownian motion $(B_t)_{t\ge0}$,
$(|B_t|)_{t\ge0}$
is a Bessel process of index $\alpha=N/2-1$.
Using the moment functions of Example \ref{example4-26}(2) and Theorem \ref{martingales-appell-allg}, we obtain the
following well-known martingale connection between Bessel processes and the
Laguerre polynomials
\begin{equation}\label{def-laguerre}
L_n^{(\alpha)}(x)=
\frac{1}{n!} x^{-\alpha} e^x \cdot \frac{d^n}{dx^n}
\bigl( x^{n+\alpha}e^{-x}\bigr) =
\sum_{j=0}^n \binom{n+\alpha}{n-j} \frac{(-x)^j}{j!}
\quad (n\ge 0)
\end{equation}
 of index $ \alpha\ge0$ (see \cite{Sz}).
\end{example}

\begin{lemma}
If
$\displaystyle R_n^{(\alpha)}(t,x):=t^n L_n^{(\alpha-1/2)}(x^2/(2t))$
$(t,x\in\b R,\> n\ge0)$
is the "$n$-th heat polynomial of Laguerre-type", then for  $n\ge 0$,
$( R_n^{(\alpha)}(t, X_t))_{t\ge0}$ is a martingale.
\end{lemma}

This observation was the motivation in Section 7.2 of \cite{ReV} for the
following
L\'evy-type characterizations of Bessel processes:

\begin{theorem}\label{levy-characterization-bessel1} Let  $\alpha\ge -1/2$. Then
an  a.s.~continuous process $X$ on $[0,\infty)$
is  a Bessel process of order $\alpha$  if and only if
 $(X_t^2-2(\alpha+1)t)_{t\ge 0}$ and
 $(X_t^4-4(\alpha+2)tX_t^2 +4(\alpha+1)(\alpha+2)t^2)_{t\ge 0}$ are
 martingales (or local martingales).
\end{theorem}

\begin{proof} By the (local) martingale conditions it is possible to
 compute the quadratic variation $[X_t^2-2(\alpha+1)t]_t$ which allows to verify
 one of the conditions of the martingale characterization
 \ref{mart-characterization-3}.
\end{proof}

\begin{theorem}\label{levy-characterization-bessel2}
 For  $\alpha\ge -1/2$,
let  $(X_t)_{t\ge0}$ be an arbitrary process on $[0,\infty)$ such that
$( R_n^{(\alpha)}(t, X_t))_{t\ge0}$ is a martingale for $n=1,2,3,4$.
Then  $(X_t)_{t\ge0}$ is a  Bessel process of index $\alpha$.
\end{theorem}

\begin{proof}
A straightforward calculation  yields
\begin{equation}
 E((X_t^2-X_s^2)^4)
= 16(\alpha+1)(\alpha+2) (t-s)^2\cdot g_\alpha(s,t)
\end{equation}
for some concrete polynomial $g_\alpha$.
 Kolmogorov's criterion now ensures that
  $(X_t^2)_{t\ge0}$, and hence   $(X_t)_{t\ge0}$, admits an a.s.~continuous
 modification.
Theorem \ref{levy-characterization-bessel1}  completes the proof.
\end{proof}

A similar result for Brownian motion on $\b R$ and Hermite polynomials can be
found in
Wesolowski \cite{W}, and  for Brownian motions on compact Lie
groups   in \cite{V1}. For further polynomial martingale relations of concrete
processes we also refer to \cite{Scho}.
An extension to one-dimensional Dunkl processes is given below.

We finally notice that the characterizations
\ref{levy-characterization-bessel1} and
\ref{levy-characterization-bessel1} can be also extended to the
(a.s. continuous) Dunkl-Bessel
processes on Weyl chambers as well as to Wishart processes on the matrix cones
$\Pi_q$, which are the Gaussian processes on the Bessel-type hypergroups on
$\Pi_q$; see \cite{V2}.

\subsection{Appell characters associated with Dunkl processes}

In this section we restrict our attention to the Dunkl Gaussian semigroup
$(P_t^\Gamma)_{t\ge0}$
for some multiplicity $k\ge0$.
 Here, all
moments exist, and the Taylor expansion (\ref{taylor-appell}) actually
 becomes a power series. The coefficients $a_\nu^\Gamma(t)$ of
the associated  Appell characters $R_\nu^\Gamma$
satisfy $a_\nu^\Gamma(-t) = m_\nu(P_t^\Gamma(0,.))$ for $t\ge0$. It follows
from  \ref{exampes-moments-gauss}  that for all $t\in\b R$,
$$a_{2\nu}^\Gamma(t)= \frac{(2\nu)!}{\nu!}\cdot (-t)^{|\nu|} \quad\quad\quad{\rm and}\quad\quad\quad
a_{\lambda}^\Gamma(t)=0 \quad\quad{\rm otherwise,}$$
i.e., if at least one component of $\lambda\in\b Z_+^N$ is odd. Therefore,
\begin{equation}\label{repro-eq-8.7}
R_\nu^\Gamma(t,x)= \sum_{\rho\in\b Z_+^N,\>\> 2\rho\le\nu} \frac{\nu!}{(\nu-2\rho)!\> \rho!}\>
(-t)^{|\rho|}\> m_{\nu-2\rho}(x) \quad\quad{\rm for}\>\> \nu\in\b R_+^N.
\end{equation}
In particular the homogeneity of the  $m_\nu$ yields that
\begin{equation}\label{homogeneity-moment-functions}
R_\nu^\Gamma(t,x)= \sqrt{t}^{|\nu|}\cdot R_\nu^\Gamma(1,x/\sqrt
t)\quad\quad\quad(x\in\b R^N,\>\> t >0).
\end{equation}

\begin{examples}\label{examples-cocharacters}
\begin{enumerate}
\item[\rm{(1)}]  In the group case $k=0$ with  $m_\nu(x):= x^\nu$, Eq.~(\ref{repro-eq-8.7}) implies
$\displaystyle R_\nu^\Gamma(t,x) = \sqrt t^{|\nu|}\cdot\widetilde
H_\nu\Bigl(\frac{x}{2\sqrt t}\Bigr)$ for
 $x\in\b R^N,\> \nu\in\b Z_+^N,\>\> t\in\b R$
with the classical $N$-dimensional Hermite polynomials
$$\widetilde H_\nu(x) =\prod_{i=1}^N    H_{\nu_i}(x_i) \quad\quad{\rm with}\quad\quad
  H_{n}(y) = \sum_{j=0}^{\lfloor n/2\rfloor} \frac{(-1)^j\> n!}{j!\> (n-2j)!}\> (2y)^{n-2j};$$
cf.~Section 5.5 of \cite{Sz} for $N=1$.
\item[\rm{(2)}] For  $N=1$, $W=\b Z_2$  and $k\ge0$, Example
\ref{examples-moment-functions-Dunkl}(2) shows
\begin{align}
 R_{2n}^{\Gamma,k}(t,x) &= (-1)^n 2^{2n} n! \> t^n\>  L_n^{(k-1/2)}(x^2/4t),\notag\\
R_{2n+1}^{\Gamma,k}(t,x) &=(-1)^n 2^{2n+1} n! \> t^n\>x\>  L_n^{(k+1/2)}(x^2/4t)
\quad\quad (n\in\b Z_+),
\end{align}
 with the Laguerre
polynomials   in (\ref{def-laguerre}).
 The  $(R_n^{\Gamma,k})_{n\ge0}$ are called often generalized Hermite or heat
 polynomials (see e.g. \cite{Ros}).
For each $t>0$ the  $(  R_{n}^{\Gamma,k}(t,.))_{n\ge0}$ are orthogonal
 w.r.t.
$$dP_t^\Gamma(0,.)(x) = \frac{\Gamma(k+1/2)}{(4t)^{k+1/2}}\>|x|^{2k}\>  e^{-x^2/4t}\,dx.$$
\end{enumerate}
\end{examples}

Theorem \ref{martingales-appell-allg} implies that for an one-dimensional Dunkl
process  $(X_t)_{t\ge0}$  associated with the Dunkl heat semigroup of index $k\ge0$
the  $(R_{n}^{\Gamma,k}(t,X_t))_{t\ge0}$ are martingales  for $n\in\b N$.
We can generalize the L\'evy-type characterization
\ref{levy-characterization-bessel2} of Bessel processes as follows:

\begin{theorem}\label{levy-characterization-bessel3}
 For  $k\ge0$,
let  $(X_t)_{t\ge0}$ be an arbitrary process on $\b R$ such that
$( R_n^{\Gamma,k}(t, X_t))_{t\ge0}$ is a martingale for $n=1,2,4,6,8$.
Then  $(X_t)_{t\ge0}$ is a  Dunkl process of index $k$.
\end{theorem}

\begin{proof} The process $(X_t^2)_{t\ge0}$  satisfies the conditions of
  Theorem \ref{levy-characterization-bessel2}, i.e., it is a Bessel process in
  distribution and a submartingale. On the other hand, as $(X_t)_{t\ge0}$ is a
  martingale, Theorem 3 of \cite{CGY} yields the claim.
 \end{proof}

One might suggest from the examples in  \ref{examples-cocharacters} that
the $R_{\nu}^\Gamma(t,.)$ are always orthogonal w.r.t.~$dP_t^\Gamma(0,.)$  for
$t>0$. We shall see below that this is not correct
in many cases. For this we introduce so-called Appell
cocharacters, which turn out to form a biorthogonal system for the
Appell characters.

\begin{definition}\label{cocharacters}
Consider the quotient
\begin{equation}\label{hom-s}
\theta_t(x,y):= \frac{\Gamma_k(t,x,y)}{\Gamma_k(t,0,y)}
= e^{-|x|^2/4t} \> E_k(x,y/2t) =
\sum_{\nu\in\b Z_+^n}\frac{m_\nu(x)}{\nu!} \> S_\nu^\Gamma(t,y)
\end{equation}
where, by Proposition \ref{taylor-etc-moment}(3), the coefficients satisfy
$$S_\nu^\Gamma(t,y) = T_x^\nu\bigl( e^{-|x|^2/4t} \>
E_k(x,y/2t)\bigr)\bigl|_{x=0}.$$
The $S_\nu^\Gamma(t,.)$ are also  polynomials of degree
$|\nu|$ and will be called the {\it Appell cocharacters} of the
 Dunkl heat semigroup.
\end{definition}

Using the homogeneity of $m_\nu$, we obtain
\begin{equation}\label{homogeneity-S-nu}
S_\nu^\Gamma(t,y)= \Bigl(\frac{1}{\sqrt t}\Bigr)^{|\nu|}\cdot
S_\nu^\Gamma(1, y/\sqrt t\>) \quad\quad (y\in \b R^N,\> t>0).
\end{equation}
A comparison of the homogeneous parts of degree $n$ in the expansions (\ref{hom-s}) and (\ref{taylor-appell})
shows that the linear spaces generated by
$(S_\nu^\Gamma(t,.))_{|\nu|=n}$ and\\
$(R^\Gamma_\nu(t,.))_{|\nu|=n}$
are equal for  $t>0$. Hence,
$(S_\nu^\Gamma(t,.))_{|\nu|\le n}$ is also a  basis
 of  $\bigoplus_{j=0}^n {\cal P}_j$.

The $S_\nu^\Gamma(t,.)$  and $R^\Gamma_\nu(t,.)$  are related by  the following
 biorthogonality:

\begin{theorem}\label{biorthogonality}
 Let $t>0$, $\nu,\rho\in\b Z_+^N$, and $p\in {\cal P}$  with
 $deg\> p<|\nu|$.  Then:
\begin{enumerate}
\item[\rm{(1)}]  $\displaystyle \int_{\b R^N} R_\nu^\Gamma(t,y) \cdot
  S_\rho^\Gamma(t,y)\> dP_t^\Gamma(0,.)(y)
 = \nu!\> \delta_{\nu,\,\rho};$
\item[\rm{(2)}]  $\displaystyle \int_{\b R^N} p(y)\cdot
  S_\nu^\Gamma(t,y)\> dP_t^\Gamma(0,.)(y)=
\int_{\b R^N} p(y)\cdot R_\nu^\Gamma(t,y)\> dP_t^\Gamma(0,.)(y)=0$.
\end{enumerate}\end{theorem}

\begin{proof}
The definition of $\theta_t$ and Lemma \ref{prop-r-apell}(2)
yield
\begin{align}\label{eq-8.12}
m_\nu(x) &= \int_{\b R^N} R_\nu^\Gamma(t,y) \theta_t(x,y)\> \>
dP_t^\Gamma(0,.)(y)
\\
 &= \int_{\b R^N} \sum_{n=0}^\infty\sum_{|\rho|=n} R_\nu^\Gamma(t,y)
 S_\rho^\Gamma(t,y)\> \frac{m_\rho(x)}{\rho!}  \>  dP_t^\Gamma(0,.)(y)
\notag\\
 &= \sum_{n=0}^\infty\sum_{|\rho|=n} \frac{m_\rho(x)}{\rho!}  \> \int_{\b R^N}
 R_\nu^\Gamma(t,y)\>
S_\rho^\Gamma(t,y)\> dP_t^\Gamma(0,.)(y)\notag
\end{align}
where we must justify that  summation and  integration commute.
For this, we may restrict our attention to the case $t=1/4$ by normalization and
decompose $\theta_{1/4}(x,y)$ into its
 $x$-homogeneous parts
$$\theta_{1/4}(x,y)= \sum_{n=0}^\infty L_n(y,x) \quad\quad{\rm with} \quad\quad
L_n(y,x)= \sum_{|\nu|=n} \frac{m_\nu(x)}{\nu!} S_\nu^\Gamma(1/4,y).$$
The estimations of Corollary \ref{C:growthbounds} imply
$$|L_{2n}(y,x)|\le \frac{|x|^{2n}}{n!}\cdot
(1+2|y|^2)^{n}\quad\>\text{for }\,n\in \b Z_+\,,$$
and a similar estimation  for odd indices (see the proof of 3.8 in \cite{R2}).
Therefore,
$$\sum_{n=0}^\infty\int_{\b R^N}  |L_n(y,x)|\> R_\nu^\Gamma(1/4,y)\>
dP_{1/4}^\Gamma(0,.)(y) <\infty.$$
The dominated convergence theorem now justifies the last step in the equation above
 for $t=1/4$, which yields  Part (1).
  Part (2) follows from Part (1).
\end{proof}

\begin{remark} For $t=1/2$, the preceding result
 shows that $(R_\nu^\Gamma(1/2,.))_{\nu\in\b Z_+^N}$ is
 orthogonal w.r.t~
$dP_{1/2}^\Gamma(0,.)$ if and only if $R_\nu^\Gamma(1/2,x)=c_\nu
S_\nu^\Gamma(1/2,x)$ with suitable constants $c_\nu$.
A comparison of (\ref{hom-s}) and (\ref{taylor-appell})  shows that this is
 equivalent to $m_\nu(x)=c_\nu x^\nu$ which holds
 for the examples in \ref{examples-cocharacters}.
On the other hand, this is not correct for the $A_{N-1}$- and the
 $B_N$-cases
for $N\ge3$.
\end{remark}

The following result reflects the dual nature of Appell characters and cocharacters.

\begin{proposition}\label{dual-nature}
 Let $t\in\b R$, $x\in\b R^N$, and   $\nu\in\b Z_+^N$. Then
$$R_\nu^\Gamma(t,x)= e^{-t\Delta_k}m_\nu(x) \quad\quad \text{and}\quad\quad
S_\nu^\Gamma(t,x)= \bigl(\frac{1}{2t}\bigr)^{|\nu|} \cdot  e^{-t\Delta_k}x^\nu.$$
\end{proposition}

\begin{proof} By Proposition \ref{heat_polynomial},  $\Delta_k$ is also the generator of the heat
  semigroup acting on  $\cal P$ (instead of $C_0(\b R^N)$).  Therefore, by Lemma \ref{prop-r-apell}(2),
$e^{t\Delta_k}R_\nu^\Gamma(t,x)=m_\nu(x)$ for $t\ge0$.
This yields the first statement for  $t\ge0$. As both sides  are polynomials
in $t$, this holds in general.

 Let $\Delta_k^y$ be the Dunkl Laplacian acting on the variable $y$,
and  $V_x$ be the intertwiner
w.r.t.~$x$. Then
\begin{align}
 e^{t\Delta_k^y}\Bigl( e^{-|x|^2/4t}\, E_k(x, y/2t)\Bigr)\,&=\,
e^{-|x|^2/4t}\cdot e^{|x|^2/4t}\,E_k(x,y/2t)\notag\\
\,&=\, E_k(x,y/2t)\,=\,
V_x(e^{\langle x, \,y/2t\rangle}).
\notag
\end{align}
Consider on both sides the homogeneous part $W_n$  of degree $n$
in the variable $x$.
 Using the left hand side, we obtain from (\ref{hom-s}) that
$$W_n= e^{t\Delta_k^y}\Bigl(\sum_{|\nu|=n} \frac{m_\nu(x)}{\nu!}\>
S_\nu^\Gamma(t,y)\Bigr)=
 \sum_{|\nu|=n} \frac{m_\nu(x)}{\nu!}\>  e^{t\Delta_k^y} \> S_\nu^\Gamma(t,y).$$
Moreover, using the right hand side, we conclude from Section  \ref{Dunkloperators}  and
$V_x(x^\nu)=m_\nu(x)$ that
$$W_n= V_x\Bigl(\sum_{|\nu|=n} \frac{x^\nu}{\nu!}\>  (y/2t)^\nu\Bigr) =
\sum_{|\nu|=n} \frac{m_\nu(x)}{\nu!}\>  (y/2t)^\nu.$$
A comparison of the  coefficients leads to the second statement.
\end{proof}

We give a further application of Theorem \ref{biorthogonality} for $t=1/2$. For this, we
employ the adjoint operator
$T_j^*$ of the Dunkl operator $T_j$ ($j=1,\ldots,N$) in
$L^2(\b R^N,dP_{1/2}^\Gamma(0,.))$ which is given by
\begin{equation}\label{form-tstern}
T_j^*f(x) =x_j f(x)-T_j f(x) = -e^{|x|^2/2}\cdot T_j\Bigl(
e^{-|x|^2/2}\>f(x)\Bigr) \quad\quad(f\in{\cal P});
\end{equation}
see Lemma 3.7 of \cite{Du2} and use for the second  equation the product rule \ref{L:Produktregel}.)

\begin{corollary}\label{rodrig-s} For  $\nu\in\b Z_+^N,\> j=1,\ldots,N$,
  $x\in\b R^N$, and $t>0$,
\begin{enumerate}
\item[\rm{(1)}] $S_{\nu+e_j}^\Gamma(1/2, x) = T_j^*S_\nu^\Gamma(1/2, x)$;
\item[\rm{(2)}] Rodriguez formula: $S_\nu^\Gamma(t, x)= (-1)^{|\nu|}\,
  e^{|x|^2/4t}\, T^\nu\bigl( e^{-|x|^2/4t}\bigr)$.
\end{enumerate}\end{corollary}

\begin{proof} For simplicity, we suppress the time parameter $t=1/2$ in  (1).
 Theorem \ref{biorthogonality}(1)  and Remark \ref{T-und-R} yield that for all
 $\rho\in\b Z_+^N$,
\begin{align}
 \int_{\b R^N} R_{\rho+e_j}^\Gamma\cdot  T_j^*S_\nu^\Gamma\>\,
 dP^\Gamma\,& =  \int_{\b R^N} T_j R_{\rho+e_j}^\Gamma\cdot
 S_\nu^\Gamma\>\, dP^\Gamma
\notag\\
&
= (\rho_j+1)\int_{\b R^N}  R_{\rho}^\Gamma\cdot S_\nu^\Gamma\,\>
 dP^\Gamma = \delta_{\rho,\,\nu}\cdot (\rho+e_j)! \notag\\
&=
\int_{\b R^N} R_{\rho+e_j}^\Gamma\cdot S_{\nu+e_j}^\Gamma\>\> dP^\Gamma.\notag
\end{align}
As $\cal P$ is dense in  $L^2(\b R^N,dP_{1/2}^\Gamma(0,.))$, Part (1)
is clear. Part (2) for $t=1/2$ follows now from
(\ref{form-tstern}), and the general case is a consequence of the homogeneity
of the $S_\nu^\Gamma$.
\end{proof}

Theorem \ref{biorthogonality} and orthogonalization within the spaces
$$V_n:=e^{-\Delta_k/2}{\cal P}_n\subset {\cal P}$$
 leads
to systems of orthogonal polynomials on $\b R^N$ w.r.t~ $P_{1/2}^\Gamma(0,.)$,
namely the  generalized Hermite polynomials $(H_\nu)_{\nu\in\b Z_+^N}$ from Definition
 \ref{Hermitepolys}.
In this way results for
generalized Hermite polynomials
 can be transfered to  Appell characters and cocharacters and conversely. Here
 is a list of a few facts in this direction:

\begin{proposition} For all $t\in\b R$, $x,y\in\b R^N$, $n\in\b N$,
 and $\nu\in\b Z_+^N$:
\begin{enumerate}
\item[\rm{(1)}]  $\displaystyle R_\nu^\Gamma(t,x)= e^{-t\Delta_k}m_\nu(x)\,$
 and $\,\displaystyle S_\nu^\Gamma(t,x)= \bigl(\frac{1}{2t}\bigr)^{|\nu|} e^{-t\Delta_k}x^\nu$;
\item[\rm{(2)}] Rodriguez formula for $R_\nu^\Gamma$:
  Let  $m_\nu(T)$
  denote the operators which is obtained from $m_\nu(x),\>\phi_\nu(x)$ by replacing
 the  $x_j$ by the Dunkl operators $T_j$. Then
$$R_\nu^\Gamma(t,x)= (-2t)^{|\nu|} e^{|x|^2/4t}\,  m_\nu(T)e^{-|x|^2/4t}.$$
\item[\rm{(3)}] Eigenfunctions of a CMS-type Schr\"odinger operator:
$R_\nu^\Gamma(t,.)$ and $S_\nu^\Gamma(t,.)$ satisfy
 $$\bigl(2t\Delta_k - \sum_{l=1}^N x_l\partial_l\bigr)f= -|\nu|\cdot f.$$
\item[\rm{(4)}]  The functions $e^{-|x|^2/8t}\,R_\nu^\Gamma(t,x)$ and
 $e^{-|x|^2/8t}\,S_\nu(t,x)$ \,
 satisfy
$$ (4t\,\Delta_k -|x|^2)f= -(2|\nu|+2\gamma+N) f.$$
\item[\rm{(5)}] Mehler formula: For  $r\in\b C$ with $|r|<1$,
\begin{align}
\sum_{\nu\in \b Z_+^N} &\frac{R_\nu^\Gamma(t,x)\, S_\nu^\Gamma(t,y)}
{\nu!}\, r^{|\nu|} \notag\\
&\,=
\, \frac{1}{(1-r^2)^{\,\gamma+N/2}}\,
{\rm exp}\left\{-\frac{t r^2(|x|^2 + |y|^2)}{1-r^2}\right\}
E_k\!\left(\frac{2trx}{1-r^2},\,y\right).\notag
\end{align}
\end{enumerate}
\end{proposition}

\begin{proof}
 (1) is Proposition \ref{dual-nature} and
 (2) follows from Corollary \ref{rodrig-s}.  Moreover, the generalized Hermite polynomials
 satisfy the equations in (3) and (4) for $t=1/2$  (see Corollary \ref{C:diffgleichungen}) where this equation only depends on
 the degree $|\nu|$. Therefore, (3) and (4) hold for $R_\nu^\Gamma(t,.)$ and $S_\nu^\Gamma(t,.)$
 for $t=1/2$. Renormalization by Eq.~(\ref{homogeneity-moment-functions})
  and (\ref{homogeneity-S-nu}) then leads
 to the case $t>0$ and  analytic continuation to the general case.
For (5) we may again assume  $t=1/2$ by (\ref{homogeneity-moment-functions})
  and (\ref{homogeneity-S-nu}).  In this case, we first note that for any  orthonormal basis
$\{\phi_\nu\,, \nu\in \b Z_+^N\}$  of $\mathcal P_{\b R}$
with respect to  $[.\,,.]_k$ with $\phi_\nu\in
\mathcal P_{|\nu|}$, we have
$$\sum_\nu \phi_\nu(x)\phi_\nu(y) = E_k(x,y)=\sum_\nu
\frac{m_\nu(x)y^\nu}{\nu!}$$
where the second equation follows from the definition of moment functions.
Therefore,  by Proposition \ref{dual-nature} and
the definition of the $H_\nu$,
$$\sum_{|\nu|=n} H_\nu(x)H_\nu(y)=\sum_{|\nu|=n}\frac{R_\nu^\Gamma(1/2,x)\, S_\nu^\Gamma(1/2,y)}
{\nu!},$$
which leads together with Theorem \ref{T:Hermiteproperties} to (5) for $t=1/2$.
\end{proof}

\vfill\newpage

\section{Notation}

\noindent
We denote by $\b Z,\, \b R$ and $\b C$ the sets of integer, real and complex
numbers respectively.
Further, $\, \b Z_+ = \{n\in \b Z: n\geq 0\}.$ For a multi-index $\nu = (\nu_1, \ldots, \nu_N)\in \mathbb Z_+^N$, we write $\,|\nu| = \nu_1 + \ldots + \nu_n.$

For a locally compact Hausdorff space $X,$ we denote by
 $\,C(X), C_b(X),$ $ C_c(X),$ $ C_0(X)$  the spaces of continuous complex-valued
 functions on $X,$ those which are bounded,
those with compact support, and those which vanish at infinity, respectively.
Further, $M_b(X), \, M_b^+(X),\, M^1(X)$ are the spaces of regular bounded Borel measures
on $X,$ those which are positive, and those which are probability-measures, respectively. Finally,
$\scr B(X)$ stands for the $\sigma$-algebra of Borel sets on $X$.

The vector space of complex polynomials on $\mathbb R^N$ is denoted by
${\cal P}=\mathbb C[\b R^N],$ and
and ${\cal P_n}$ stands for the subspace of those polynomials which are homogeneous of degree $n$.
Finally,  $\mathcal S(\b R^N)$ is the Schwartz space of rapidly decreasing functions on $\b R^N$.

\end{document}